\documentclass[11pt, reqno]{amsart} 
\usepackage{marvosym}
\usepackage{bm}
\usepackage{tikz, tikz-cd}
\usepackage{float}
\usetikzlibrary{arrows.meta,calc}

\newcommand{\verticalbond}{
  \begin{tikzpicture}[baseline=0.1ex, scale=1]
    \draw[thick] (-1ex, -0.2ex)--(-1ex, 0.8ex)--(0,0.8ex) -- (0,1.8ex);
    \draw[thick] (0, -0.2ex)--(0,0.8ex) -- (0,1.8ex);
    \fill (-1ex, -0.2ex) circle (0.2ex);
    \fill (-1ex, 0.8ex) circle (0.2ex);
    \fill (0,-0.2ex) circle (0.2ex);
    \fill (0,0.8ex) circle (0.2ex);
    \fill (0,1.8ex) circle (0.2ex);
  \end{tikzpicture}
}

\newcommand{\samedirNotUnique}{
  \begin{tikzpicture}[baseline=0.1ex, scale=1]

    \draw[thick] (-1ex, -0.2ex)--(-1ex, 0.8ex) -- (0, 0.8ex) -- (0, 1.8ex);
    \draw[thick] (-1ex, -1.2ex)--(-1ex, -0.2ex)--(0, -0.2ex)--(0, 0.8ex) -- (0, 1.8ex);
    
    \fill (0, 1.8ex) circle (0.2ex);  
    \fill (0, 0.8ex) circle (0.2ex);  
    \fill (-1ex, 0.8ex) circle (0.2ex);  
    \fill (-1ex, -0.2ex) circle (0.2ex);  
    \fill (0, -0.2ex) circle (0.2ex); 
    \fill (-1ex, -0.2ex) circle (0.2ex); 
    \fill (-1ex, -1.2ex) circle (0.2ex); 
   
  \end{tikzpicture}
}

\newcommand{\nocoal}{
  \begin{tikzpicture}[baseline=0.3ex, scale=1]
    \coordinate (origin) at (1.2ex, 1.8ex);

    \draw[thick] (0.2ex, 1.8ex) -- (0.2ex, 0.8ex) -- (-0.8ex, 0.8ex)--(-0.8ex, -0.2ex);

    \fill (0.2ex, 1.8ex) circle (0.2ex);
    \fill (0.2ex, 0.8ex) circle (0.2ex);
    \fill (-0.8ex, 0.8ex) circle (0.2ex);
    \fill (-0.8ex, -0.2ex) circle (0.2ex);
    
    \draw[thick] (origin) -- (1.2ex, 0.8ex) -- (1.2ex, -0.2ex)-- (0.2ex, -0.2ex);

    \fill (1.2ex, 0.8ex) circle (0.2ex);
    \fill (1.2ex, -0.2ex) circle (0.2ex);
    \fill (0.2ex, -0.2ex) circle (0.2ex);

    \fill (origin) circle (0.2ex);
  \end{tikzpicture}%
}

\newcommand{\nocoaltouch}{%
  \begin{tikzpicture}[baseline=0.3ex, scale=1]
    \coordinate (origin) at (1.2ex, 1.8ex);

    \draw[thick] (0.2ex, 1.8ex) -- (0.2ex, 0.8ex) -- (-0.8ex, 0.8ex)--(-0.8ex, -0.2ex);

    \fill (0.2ex, 1.8ex) circle (0.2ex);
    \fill (0.2ex, 0.8ex) circle (0.2ex);
    \fill (-0.8ex, 0.8ex) circle (0.2ex);
    \fill (-0.8ex, -0.2ex) circle (0.2ex);

    \draw[thick] (origin) -- (1.2ex, 0.8ex) -- (0.2ex, 0.8ex)-- (0.2ex, -0.2ex);

    \fill (1.2ex, 0.8ex) circle (0.2ex);
    \fill (0.2ex, 0.8ex) circle (0.2ex);
    \fill (0.2ex, -0.2ex) circle (0.2ex);

    \fill (origin) circle (0.2ex);
    
  \end{tikzpicture}%
}

\usepackage[T1]{fontenc}
\usepackage{lmodern}
\usepackage{mathtools, amsfonts, amssymb, mathrsfs, bm, stmaryrd}

\usepackage[utf8]{inputenc}
\usepackage[T1]{fontenc}
\usepackage{graphicx}
\usepackage{tikz}
\usetikzlibrary{shapes, arrows.meta, positioning, calc, backgrounds, lindenmayersystems, decorations.pathreplacing}
\usepackage{pgfplots, pgfplotstable}
\pgfplotsset{compat=1.17}

\usepackage{float}
\usepackage{svg}
\usepackage{wrapfig}
\usepackage{subcaption}
\usepackage{hyperref}
\usepackage{marginnote}

\pgfdeclarelindenmayersystem{cantor set}{
  \rule{F -> FfF}
  \rule{f -> fff}
}
\numberwithin{equation}{section}
\raggedright
\allowdisplaybreaks

\usepackage[letterpaper,hmargin=1.0in,vmargin=1.0in]{geometry}
\parskip 	\smallskipamount

\newcounter{dummy} \numberwithin{dummy}{section}
\newtheorem{corollary}[dummy]{Corollary}
\newtheorem{proposition}[dummy]{Proposition}
\newtheorem{theorem}[dummy]{Theorem}
\newtheorem{definition}[dummy]{Definition}
\newtheorem{lemma}[dummy]{Lemma}

\newtheorem*{remark}{Remark}

\newcommand{\N}{\ensuremath{\mathbb{N}}}
\newcommand{\R}{\ensuremath{\mathbb{R}}}
\newcommand{\Z}{\ensuremath{\mathbb{Z}}}
\newcommand{\Q}{\ensuremath{\mathbb{Q}}}

\newcommand{\PP}{\ensuremath{\mathbb{P}}}
\newcommand{\F}{\mathcal{F}}

\DeclareMathOperator{\supp}{supp}

\DeclareMathOperator{\diam}{diam}

\newcommand{\norm}[1]{\left\lVert#1\right\rVert}

 \newcommand\restrict[2]{% make the whole thing an ordinary symbol
   \left.\kern-\nulldelimiterspace % automatically resize the bar with \right
   #1% the function
   \littletaller % pretend it's a little taller at normal size
   \right|_{#2}%
   }
 \newcommand{\littletaller}{\mathchoice{\vphantom{\big|}}{}{}{}}

 \makeatletter
 \newcommand*{\scaleddelims}[3]{%
   \ensuremath{%
     \mathpalette{\@scaleddelims{#1}{#2}}{#3}%
   }%
 }   
 \newcommand*{\@scaleddelims}[4]{%
   \begingroup
     #3%
     \sbox0{$\m@th#3\vphantom{A}#4$}%
     \setbox2\vbox{\hbox{$\m@th#3#1$}\kern\z@}%
     \setbox4\vbox{\hbox{$\m@th#3#2$}\kern\z@}%
     \setbox6\hbox{$#3\vcenter{}$}%
     \ifx\downharpoonleft#1\relax  
       \let\DelimLeft=L%
     \else\ifx\upharpoonleft#1%
       \let\DelimLeft=L%
     \else\ifx\downharpoonright#1%
       \let\DelimLeft=R%
     \else\ifx\upharpoonright#1%
       \let\DelimLeft=R%
     \fi\fi\fi\fi
     \ifx\downharpoonleft#2\relax
       \let\DelimRight=L%
     \else\ifx\upharpoonleft#2\relax
       \let\DelimRight=L%
     \else\ifx\downharpoonright#2\relax
       \let\DelimRight=R%
     \else\ifx\upharpoonright#2\relax
       \let\DelimRight=R%
     \fi\fi\fi\fi
     \ifx\DelimLeft L%
       \wd2=.6\wd2
     \fi
     \ifx\DelimRight L%
       \wd4=.6\wd4
     \fi
     \ifx\DelimLeft R%
       \sbox2{\kern-.4\wd2\box2}%
     \fi
     \ifx\DelimRight R%
       \sbox4{\kern-.4\wd4\box4}%
     \fi
     \dimen0=\ht0 %
     \advance\dimen0 by -\ht6 %
     \dimen2=\dp0 %
     \advance\dimen2 by \ht6 %
     \ifdim\dimen2>\dimen0 %  
       \dimen0=\dimen2 %
     \else
       \dimen0=\dimen0 %
     \fi
     \dimen2=\ht6 %
     \advance\dimen2 by -\dimen0 %
     \dimen0=2\dimen0 %
     \def\DelimCorr{%  
       \mskip.5\thinmuskip
       \nonscript\mskip.5\thinmuskip
     }%
     \mathopen{%
       \ifx\DelimLeft R\DelimCorr\fi
       \raisebox{\dimen2}{\resizebox{!}{\dimen0}{\box2}}%
       \ifx\DelimLeft L\DelimCorr\fi
     }%
     \begingroup
       #3#4%
     \endgroup
     \mathclose{%
       \ifx\DelimRight R\DelimCorr\fi
       \raisebox{\dimen2}{\resizebox{!}{\dimen0}{\box4}}%
       \ifx\DelimRight L\DelimCorr\fi
     }%
   \endgroup
 }\makeatother

 \newcommand{\restr}[2]{#1\scaleddelims{\kern-0.5\nulldelimiterspace\upharpoonright}{\vphantom{.}}{_{#2}}}
\newcommand{\diff}{\ensuremath{\operatorname{d}\!}}

\begin{document}

\title{Geometry of semi-infinite geodesics in the Airy line ensemble}

\author{Pantelis Tassopoulos}
\address{Department of Pure Mathematics and Mathematical Statistics, 
University of Cambridge, Cambridge, United Kingdom}
\email{pkt28@cam.ac.uk}

\author{Sourav Sarkar}
\address{Department of Pure Mathematics and Mathematical Statistics, 
University of Cambridge, Cambridge, United Kingdom}
\email{ss2871@cam.ac.uk}

%{$82C22$,  and  $60H15$.}
\date{}

%{$82C22$,  and  $60H15$.}
\date{}

\begin{abstract}
We study the geometry of semi-infinite geodesics in the Airy line ensemble. 

We establish sharp estimates for disjointness probabilities involving geodesics with nearby directions and endpoints in the top line of the Airy line ensemble. Moreover, we obtain a power law upper bound on the coalescence depth of two semi-infinite geodesics with same asymptotic direction and establish exponential concentration estimates for geodesic jump times more generally.

We also establish pathwise H\"{o}lder $1/2^-$-continuity of Busemann profiles associated to semi-infinite geodesics with common asymptotic directions (in terms of the asymptotic direction). We characterize exceptional
behaviour of semi-infinite geodesics through their Busemann profiles and establish geometric properties of the associated sets of non-coalescence and non-uniqueness. 

Finally, we determine the Hausdorff dimensions of collections of semi-infinite geodesics: collections with a fixed top-line endpoint and arbitrary positive asymptotic direction almost surely have Hausdorff dimension $1$, while collections with a fixed positive direction and arbitrary endpoint almost surely have Hausdorff dimension $3/2$.
\end{abstract}

\maketitle

\begin{figure}[ht]
    \centering
    \includegraphics[width = 0.7\linewidth]{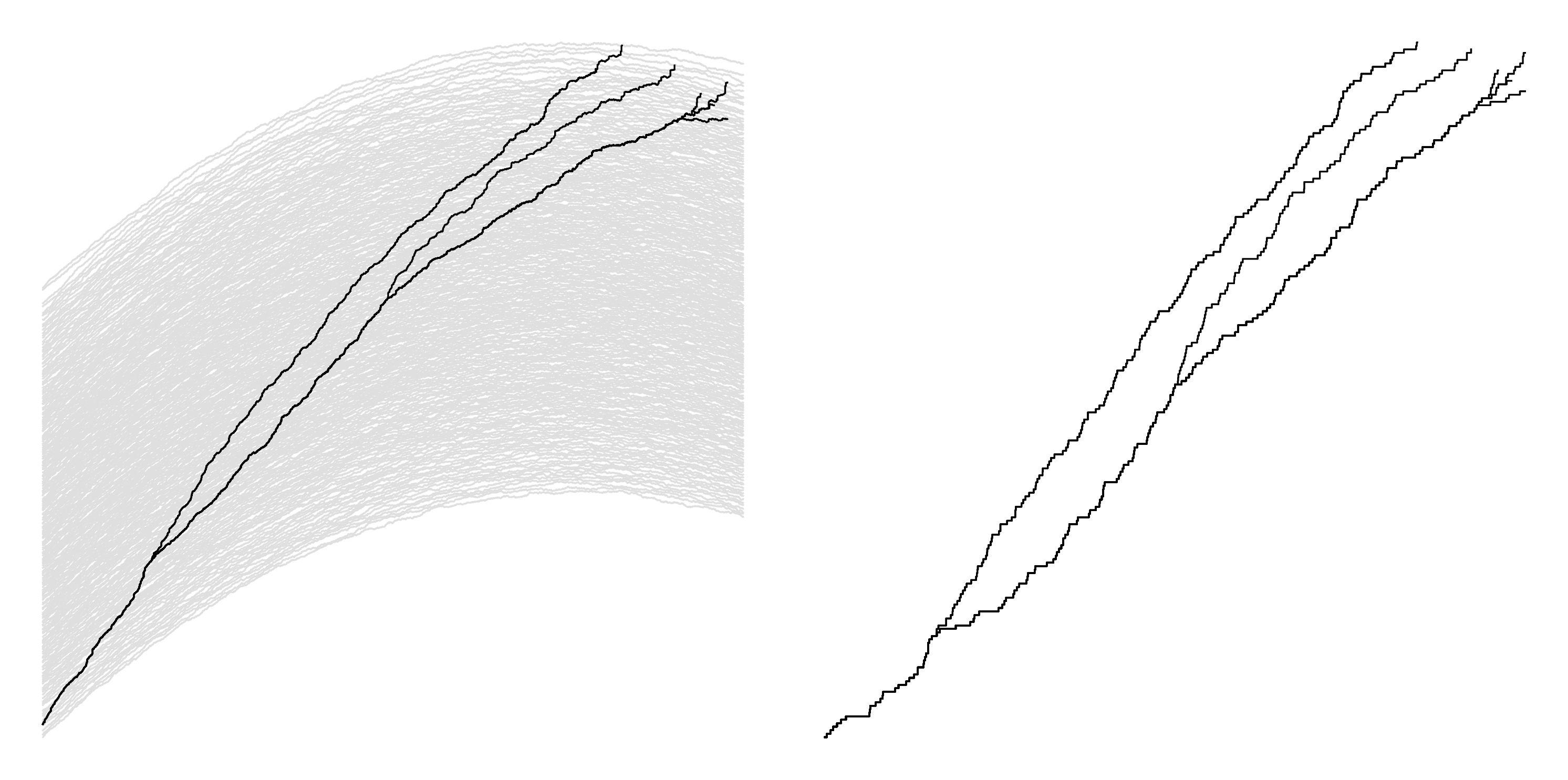}
    \caption{Restrictions of semi-infinite geodesics on the Airy line ensemble.  Plotted \textbf{Left:} on the Airy line ensemble itself and  \textbf{Right:} as integer-valued `zig-zags'.}
    \label{fig: geod coals airy}
\end{figure}

\tableofcontents

\section{Introduction}\label{sec: introduction}
In 1986, Kardar, Parisi and Zhang \cite{kardar1986dynamic} predicted that many planar random growth processes possess universal scaling behaviour. Central to the KPZ universality class is the Airy line ensemble,  \(\mathcal{A} \). It is a non-intersecting family of random continuous functions 
$\mathcal{A}_1 > \mathcal{A}_2 > \cdots$, supported on some Polish probability space $(\Omega, \mathcal{F}, \PP)$, where we assume $\F$ is complete with respect to null sets. 
It was introduced by Pr\"{a}hofer and Spohn \cite{prahofer2002scale} in the version \((\mathcal{A}^{\mathrm{stat}}_i)_{i\in \N}\coloneqq (\mathcal{A}_i(\cdot)+(\cdot)^2)_{i\in \N}\), which is stationary
in time, see also \cite{corwin2014brownian} and \cite{corwin2014ergodicity}. We will thus call it the \textbf{stationary Airy line ensemble}. The top line $\mathcal{A}_1$ is known as the parabolic $\text{Airy}_2$ process that appears as the
limiting spatial fluctuation of random growth models starting from a single point. The Airy sheet was constructed by Dauvergne, Ortmann and Vir\'ag \cite{DOV} as the scaling limit of one of the KPZ-class models, Brownian last passage percolation. It is a random continuous function on $\R^2$ and enjoys a rich fractal structure, see \cite{BatesGangulyHammond2022}, \cite{ganguly2021local} and \cite{ganguly2021random}. Moreover, it can be thought of as random `directed metric' on pairs of space-time points $(x, 0)$ and $(y, 1)$, $x, y \in \R$. Importantly, the Airy sheet $\mathcal{S}:\R^2\mapsto \R$ is a random continuous function and can defined in terms of the Airy line ensemble so that $\mathcal{S}(0,\cdot)=\mathcal{A}_1(\cdot)$.

Following the construction of the Airy sheet and its coupling to the Airy line ensemble in \cite{DOV}, a rich structure of coalescing semi-infinite paths in the Airy line ensemble was uncovered. These paths ended at the top line of the Airy line ensemble, have asymptotic parabolic shapes and were locally last passage geodesics. They were later called \emph{semi-infinite geodesics} in \cite{sarkar2021brownian} and \cite{dauvergne2022disjointoptimizersdirectedlandscape}, wherein their qualitative properties were studied. In particular, for a fixed asymptotic direction and endpoint, there exists an almost surely unique semi-infinite geodesic, and geodesics with a common fixed direction eventually coalesce.

The coalescing nature of semi-infinite paths with the same asymptotic shape meant certain `Busemann' limits (a notion originating in hyperbolic geometry, first appearing in the context of first-passage percolation in \cite{newman1995surface},\cite{hoffman2008geodesics}, and later in last-passage percolation in \cite{georgiou2017stationary}) of relative last passage values across the Airy line ensemble along the prelimiting paths are well-defined, thereby allowing for a consistent notation of `relative length' of two semi-infinite paths. Importantly, these limits appear as increments of the Airy sheet, also known as `Airy difference profiles', see for instance \cite{ganguly2021local}.

Informally, for $x> 0, y\in \R, \ell \ge 1$, an $x$-directed semi-infinite geodesic ending at $(y,\ell)$, where $x>0$, is a non-increasing path $\pi[x\to(y,\ell)]:(-\infty,y]\to\{\ell,\ell+1,\ldots\}$ whose finite restrictions are geodesics and has asymptotic shape
\[
\lim_{t\to-\infty}\frac{\pi[x\to(y,\ell)](t)}{t^2}=2x\,.
\]
We write $\pi[x\to(y,\ell)]$ when the path is unique. We can also define formally the Busemann limits
\[
\Delta^{(t, m)}_{(t', m')}(x):= 
    \displaystyle\lim_{k\to\infty}(\mathcal{A}[x_k\to (t',m')]-\mathcal{A}[x_k\to (t,m)])\,,\quad x>0, (t,m), (t', m')\in \R\times \N\,.
\]
By the preceding discussion, for $m, m' = 1$, we have by the coupling with the Airy sheet,
\[
\Delta^{(t, 1)}_{(t', 1)}(x) = \mathcal{S}(x, t) - \mathcal{S}(x, t')\,,\qquad t, t'\in \R, x> 0\,.
\]

In this paper, we construct semi-infinite geodesics with arbitrary endpoints in the Airy line ensemble, and study their geometry, coalescence structure and fractal properties. In particular, our main results establish sharp estimates on the disjointness probabilities of semi-infinite geodesics, the existence and pathwise regularity of Busemann profiles with arbitrary endpoints, give concentration bounds on coalescence depths and geodesic jump times, study exceptional events of semi-infinite geodesics and geometric properties of collections of semi-infinite paths. We now gives a brief description of the statements of the main results in the paper, in the above order.

\subsection{Main results}

\noindent\emph{Rarity of disjoint semi-infinite geodesics} We also obtain sharp estimates on disjointness when the directions
are close. For fixed $x>0$ and $y_1<y_2$, Theorem~\ref{thm: semi-inf tail disj proba} gives for any fixed $\eta \in (0, 1/2)$ and all $\varepsilon> 0$ sufficiently small depending on $\eta$,
\[
c\varepsilon^{\frac{1}{2}+\eta}\le \mathbb P\bigl(
\pi[x-\varepsilon\to(y_1,1)]
\text{ and }
\pi[x+\varepsilon\to(y_2,1)]
\text{ are disjoint}
\bigr)
\le
C\varepsilon^{1/2+o(1)}
\]
as $\varepsilon\to 0$. This estimate is used both to
study exceptional geodesic directions and to obtain the coalescence-depth
bound.

\vspace{0.1cm}
\noindent \emph{Regularity of Busemann profiles} In Theorem~\ref{thm: busemann cont} we establish simultaneously for
$s\in\mathbb R$ and $\ell\ge2$, that the Busemann profiles
\[
x\longmapsto\Delta^{(s,1)}_{(s,\ell)}(x)
\]
are locally H\"older continuous on $(0,\infty)$ with every exponent
$\alpha<1/2$.

Additionally, Theorem~\ref{thm: Busemann equality in law} gives on a suitable coupling between the AIry line ensemble $\mathcal{A}$ and Airy sheet $\mathcal{S}$,
\[
\Delta^{(0,1)}_{(0,\ell)}(x)+\mathcal S(x,0)-\mathcal A_\ell(0)
\overset d =  U_\ell(x),
\qquad x>0,
\]
 on paths $\mathcal{C}(0, \infty)$, where $U_\ell$ has the process law of $x\mapsto\mathcal A[(0,\ell)\to(x,1)]$. This relates Busemann profiles to finite-depth last
passage values across the Airy line ensemble and is used in the concentration and coalescence estimates for semi-infinite geodesic jump times.

\vspace{0.1cm}
\noindent \emph{Geodesic concentration bounds and coalescence estimates} For fixed $x>0$ and $\ell\ge2$, let $N_\ell(x)$ be the first depth
after which the geodesics $\pi[x\to(0,1)]$ and
$\pi[x\to(0,\ell)]$ coalesce. In
Theorem~\ref{thm: coal depth}, we prove that, for every $\eta>0$,
\[
\mathbb P(N_\ell(x)>n)
\le C_{x,\ell,\eta}\,n^{-1/56+\eta}\,.
\]
We also give a more explicit bound retaining the dependence
on $\ell$. 

Next, we obtain one-sided concentration estimates for semi-infinite geodesic jump times. In particular, for $k\ge 2$, with $Z_k(x,y)$ the jump time of $\pi[x\to(y, 1)]$ from level $k$ to level $k-1$, Theorem~\ref{thm: exponential intercept tails} gives universal constants $c,C>0$ such that
\[
\mathbb P(Z_k(x,y)>y-r)
\le
\min\{1,C\exp(Cx^3+Cr^3-ck)\},
\qquad x,r>0,\quad k\ge2.
\]
These estimates controls the likelihood of `large deviations' of semi-infinite geodesics near their endpoint.

\vspace{0.1cm}
\noindent \emph{Exceptional behaviour of semi-infinite geodesics} In Proposition \ref{prop: local variation geod disj} we show that the points of non-constancy of Busemann profiles encode disjointness of semi-infinite geodesics. We use this relation to identify exceptional
sets of geodesic directions and endpoints, where uniqueness or coalescence may fail. In particular, for any fixed $y_1<y_2$ and $\ell\ge1$, Corollary~\ref{corollary: geod two end points disjoint} gives
\[
\dim_{\mathscr H}
\left\{
x>0:
\pi^L[x\to(y_1,1)]
\text{ and }
\pi^R[x\to(y_2,\ell)]
\text{ are disjoint}
\right\}
=\frac12
\quad\text{almost surely}\,,
\]
where $\dim_{\mathscr H}$ denotes the Euclidean Hausdorff dimension and the superscripts $L, R$ denote leftmost and rightmost geodesics, see Lemma \ref{lemma: semi-inf geod properties}. Theorem~\ref{thm: geod no coal h dim} also shows that the set of
directions admitting two non-coalescing paths with the same asymptotic direction has Hausdorff dimension $1/2$. Moreover, the subset where such paths intersect without coalescing is at most countable.

Non-uniqueness has a different description on fixed slices. For a
fixed top-line endpoint, the exceptional directions form a dense
countably infinite set
(Proposition~\ref{prop: multiple geod fixed endpt}). For a fixed
direction, the horizontal coordinates admitting multiple paths
to some terminal level form an at most countable set
(Theorem~\ref{thm: fixed dir coal}). We further relate joint
direction--endpoint exceptional geometry to the support of the
measure defined by rectangular increments of the Airy sheet
(Theorem~\ref{thm: airy sheet monotone haus dim} and
Corollary~\ref{cor: geod disj same endpt}).

\vspace{0.1cm}
\noindent \emph{Geometry of collections of geodesics} Finally, we study geometric aspects of collections of semi-infinite geodesics. The collection of semi-infinite geodesics ending at a fixed top-line point in the Airy line ensemble forms a tree almost surely and has Hausdorff dimension $1$
(Theorem~\ref{thm: geodesic tree}). By contrast, the collection with
a fixed positive direction and varying top-line endpoint has
Hausdorff dimension $3/2$. However, it can be expressed as the union of two trees (Theorem~\ref{thm: geodesic gasket}). Excluding certain trivial extremal geodesics, the full network of geodesics with finite positive directions also has Hausdorff dimension
$3/2$ (Theorem~\ref{thm: geodesic network haus dim}). We also establish measurability of
the embedded downward trees certain ergodic properties for translates thereof.

\subsection{Related work}
Aside from the Airy line ensemble, for which our work is the first to provide quantitative control on the coalescence of geodesics, one can consider such coalescence in other geodesic environments; it is a well-studied and significant problem in probability. They have been studied in various models of the KPZ universality class: the prelimiting models of the last passage percolation \cite{Pimentel2016}, \cite{basu2018coalescencegeodesicsexactlysolvable}, \cite{SeppalainenShen2020}, \cite{Zhang2020}, \cite{BusaniFerrari2022}, in polymer models \cite{RassoulAghaSeppalainenShen2023}, \cite{GangulyHegdeZhang2023} and in the related first passage percolation models \cite{Newman1995}, \cite{DembinElboimPeled2024} and in first passage percolations in higher dimensions \cite{Alexander2023}. Geodesic networks are widely investigated in other random surface models, like the Brownian map \cite{MillerQian2023}, \cite{AngelKolesnikMiermont2017}, \cite{LeGall2010} and the Liouville quantum gravity \cite{Gwynne2021}, that arise as scaling limits of random planar maps. 

Airy sheet difference profiles were studied in \cite{Basu2021Fractal}, where the authors established that their sets of points of non-constancy on the real line have Hausdorff dimension one half almost surely. Moreover, in \cite{ganguly2021local} the authors relate these profiles to Brownian local time and Busemann limits of Airy line ensemble last passage values with arbitrary endpoint. Their treatment gives a pointwise construction and does not consider any pathwise regularity of these profiles with respect to the geodesic asymptotic direction. Our regularity results precisely address this, obtaining H\"{o}lder $1/2^-$-modifications for the Busemann profiles. Furthermore, by Theorem \ref{thm: Busemann equality in law} and the recently discovered coupling between the Airy sheet and the Airy line ensemble in \cite{dauvergne2026downthemiddle}, we anticipate connections to the study of disjoint optimizers in \cite{DOV,dauvergne2022disjointoptimizersdirectedlandscape}.

Exceptional geometry in the directed landscape is closely related
to the geometry studied here. In particular,
\cite{BatesGangulyHammond2022} investigates Hausdorff dimensions of
sets of endpoint pairs admitting disjoint directed geodesics, expressing these exceptional sets in terms of the Airy sheet. Using the coupling between the Airy line ensemble, we give a different description of these exceptional sets in terms of exceptional behaviour of semi-infinite paths in the Airy line ensemble. 

\subsection{Organisation}
Section~\ref{sec: notation} establishes notation used throughout, and
Section~\ref{sec: preliminaries} collects key last-passage,
Brownian-melon and Airy-line estimates used throughout the paper.
Section~\ref{sec: semi-inf geod} establishes existence and qualitative properties
of semi-infinite geodesics.
Then in Section~\ref{sec: geod geom}, we establish sharp disjointness estimates for semi-infinite geodesics.
Section~\ref{sec: Busemann} treats the formally defined Busemann profiles, establishing their existence, regularity and relate their laws to finite last passage percolation across the Airy line ensemble.
Section~\ref{sec: quantitative geodesics} proves concentration for geodesic jump times
and upper tail estimates for coalescence-depths of semi-infinite geodesics.
Section~\ref{sec: exceptional geod geom} describes exceptional
directions and endpoints, and Section~\ref{sec: geod collections}
studies geometric properties of semi-infinite geodesics. The appendix,
(Section~\ref{sec: polish meas}) provides the measure-theoretic details allowing us to treat embedded, downward geodesic collections as proper random objects.

\section{Notation}\label{sec: notation}

We introduce some notation and conventions we will be using throughout.

When in some estimates a constant appears that will depend on some parameters $a,b,c,\cdots$, it will be denoted by $C_{a,b,c,\cdots}$, unless otherwise specified. Constants without subscripts are deemed to be universal. Additionally, for ease of notation, such constants are allowed to change from line to line. Moreover, for ease of notation such constants may be dropped and instead replaced with the symbols $\lesssim_{a,b,c,\cdots}(\equiv O_{a,b,c,\cdots}(\cdot))$ and $\gtrsim_{a,b,c,\cdots}$ for some parameters $a,b,c,\cdots$ which stand for $\le C_{a,b,c,\cdots}$ and $\ge C'_{a,b,c,\cdots}$ for some positive constants $C_{a,b,c,\cdots}, C'_{a,b,c,\cdots}$ respectively. We will also enlarge or reduce constants as needed, without necessarily relabelling them.

We take the set of natural numbers $\N$ to be $\{1,2, \ldots \}$.

We now turn to some notational conventions for the path spaces that will be used throughout. For general domains of paths $J$, we denote the space of continuous paths in $\R^n$ $n\ge 1$ in the usual local uniform topologies, by $\mathcal{C}^n J$. For random functions taking values in these spaces, we will always endow them with their respective Borel $\sigma$-algebras generated by the above topologies. For $a<b$, $k \in \N$ and $\underline{x},\underline{y}\in \R^k_{>}$, let $\mathcal{C}^k_{\underline{x}, \underline{y}}[a,b]$ denote the space 
\[
\Big\{g \in \mathcal{C}^k[a,b]: \forall 1\le i \le k,\,  g_{i}(a) = x_i \; \text{and}\; g_{i}(b) = y_i\Big\}\,.
\] 

We say that a Brownian motion or a Brownian bridge has \textit{rate $v$} (or \emph{diffusivity} $v$ if its quadratic variation in an interval $[s,t]$ is equal to $v(t-s)$. We say that a Dyson's Brownian motion or a Brownian $k$-melon has rate $v$ if the component Brownian motions have rate $v$. From now on, all Brownian motions are rate two unless stated otherwise. 

Finally, we will always on probability spaces where the underlying probability measure has been completed with respect to null sets.

\section{Preliminaries}\label{sec: preliminaries}
In this section, we will recall some basic definitions that appear in the KPZ universality class, namely, last passage percolation, the Pitman transform and melons; and collect some basic results that will be useful later on including some elementary estimates involving Brownian bridge, Radon-Nikodym derivatives (against Brownian motion) estimates and properties of the Airy line ensemble. We start with the central probabilistic object of study, namely random line ensembles.

The following definition makes precise the notion of a \textit{random line ensemble}. It is a random variable taking values in an indexed (at most countably infinite) family of continuous paths defined on a common subset of $\R$. 

\begin{definition}[Random ensemble]\label{def: random ensemble}
    Let $\Sigma$ be a (possibly infinite) interval of $\Z$, and let $\Lambda$ be an interval of $\R$. Consider the set \(X:=\mathcal{C}^\Sigma\) of continuous functions $f:\Sigma\times \Lambda \rightarrow \R$. We endow it with the topology of uniform convergence on compact subsets of $\Sigma\times\Lambda$. Let $\mathcal{C}$ denote the sigma-field  generated by Borel sets in $X$.

    A {\it $\Sigma$-indexed line ensemble} $L$ is a random variable defined on a probability space $(\Omega,\mathfrak{B},\PP)$, taking values in $X$ such that $L$ is a $(\mathfrak{B},\mathcal{C})$-measurable function. Furthermore, we write $L_i:=(L(\omega))(i,\cdot)$ for the line indexed by $i\in\Sigma$. 
\end{definition}% Random Ensemble

\subsection{Last passage percolation} We begin with the collection of some preliminary facts regarding last passage percolation (sometimes abbreviated as LPP in the paper) over ensembles of functions following \cite{DOV}.

Formally, let \(I\subset \mathbb{Z}\) be a possibly finite index set and define the space \(\mathcal{C}^{I}\) of sequences of continuous functions with real domains, that is, the space
\[f: \mathbb{R}\times I\to \mathbb{R}\quad (x,i)\mapsto f_i(x).\]
Let \(x\le y \in \mathbb{R}\), and \(m\le \ell \in \mathbb{Z}\) respectively. A \emph{path} from \((x,\ell)\) to \((y,m)\) is a non-increasing  function \(\pi: [x,y] \to I\) which is cadlag on \([x,y)\) and takes the values \(\pi(x) \le \ell\) and \(\pi(y)= m\). We adopt the convention that $\lim_{z\nearrow x}\pi(z) = \ell$ and $\lim_{z\searrow y}\pi(z) = m$. The above convention can be thought of as extending the path $\pi$ in a cadlag way from its original domain $[x, y]$ to the entire real line, starting at $\ell$ level before $x$ and ending at $m$ after $y$.

Equivalently, associated to a path $
\pi$ as above, is a unique \emph{zig-zag}, 
\begin{equation}\label{eq: zig-zag}
    \mathsf{ZigZag}(\pi) := \{(t, z)\in [x,y]\times \R: \pi(t)\le z\le \lim_{s\nearrow t}\pi(s)]\}\,.
\end{equation}
Informally, $\mathsf{ZigZag}(\pi)$ is obtained from the graph of $\pi$ by filling in its gaps during jumps.

We now define two notions of disjointness between geodesics in last passage percolation.

\begin{definition}[Disjointness]\label{def: disjointness}
    Fix $x_1\le x_2$, $x'_1\le x'_2$ and $\ell \ge m$. Then, we say the paths $\pi_1$ and $\pi_2$ from $(x_1, \ell)$ to $(x_2, m)$ and $(x'_1, \ell)$ to $(x'_2, m)$, respectively are
    \begin{itemize}
        \item \emph{essentially disjoint} if the intersection of zigzags $\mathsf{ZigZag}(\pi_1)$ and $\mathsf{ZigZag}(\pi_2)$ contains no horizontal line segment of positive length,
        \item \emph{disjoint} if the zigzags $\mathsf{ZigZag}(\pi_1)$ and $\mathsf{ZigZag}(\pi_2)$ have empty intersection.
    \end{itemize}
    Note that being disjoint implies being essentially disjoint, as the latter is equivalent to the zig-zags $\mathsf{ZigZag}(\pi_1)$ and $\mathsf{ZigZag}(\pi_2)$ sharing no interior of any of their horizontal segments.
\end{definition}

In what is to follow, since we will primarily be considering the Airy line ensemble (see Section \ref{subsec: Airy line ensemble} for a definition), we will take the indexing set to be \(I = \mathbb{N}\). We now define an important quantity associated to each such path, namely, its \textit{length} as the sum of increments of \(f\) along \(\pi\). This also leads one to naturally define a derived quantity, namely the \textit{last passage value}.

\begin{definition}[Length]\label{def: length} Let \(x\le y \in \mathbb{R}\) and \(m < k\in\mathbb{Z}\). For each \(m\le i \le k-1\), let \(t_{k-i}\) denote the jump of the path \(\pi\), on an ensemble \((f_i)_{i\in I}\), from \(f_{i+1}\) to \(f_{i}\). Then the length of \(\pi\) is defined as
\begin{equation}\label{eq: last passage}
\ell(\pi) = f_m(y)-f_m(t_{k-m}) + \displaystyle\sum_{i = 1}^{k-m-1}(f_{k-i}(t_{i+1})-f_{k-i}(t_{i}))+f_{k}(t_{1})-f_{k}(x).
\end{equation}
\end{definition}%Length definition
With \(x\le y, m<k\) as before and \(f\in \mathcal{C}^{I}\), define the \textbf{last passage value} of \(f\) from \((x,k)\) to \((y,m)\) as
    \[
    f[(x,k)\to(y,m)] \coloneqq \displaystyle \sup_{\pi}\ell(\pi)\,,
    \]
where the supremum is over precisely the paths \(\pi\) from \((x,k)\) to \((y,m)\).
If $m=k$, we set
\[
    f[(x,k)\to(y,k)] \coloneqq f_k(y)-f_k(x)\,,
\]
Any path \(\pi\) from \((x,k)\) to \((y,m)\) such that its length is equal to its last passage value is called a \textbf{geodesic}, and can be formally identified with its corresponding zig-zag, $\mathsf{ZigZag}(\pi)$. Note the restriction of any geodesic  \(\pi\) from \((x,k)\) to \((y,m)\) to a subinterval $[z, w] \subseteq [x, y]$ is again a geodesic (after possibly readjusting $\pi$ outside $z, w]$, in line with the conventions above). We may refer to this as the \emph{restriction property} of geodesics. The existence of geodesics follows from a deterministic compactness argument.

First we define a path space for geodesics with shared start and end points, say $(x, \ell)$ and $(y, m)$ for $x\le y$ and $\ell \ge m$, which we take as fixed in what is to follow. Each zig-zag induced by some geodesic is a closed subset of $\R\times [1, \infty)$. Consider the set $\mathcal{Z}$ of `zig-zags' of admissible paths $\pi: [x, y] \to \N$ starting from $(x, \ell)$ and ending at $(y, m)$, as defined above.

By the \emph{topology of jump times} we mean that a family of zig-zags $(\mathsf{ZigZag}_m)_m$ converges to a zig-zag $\mathsf{ZigZag}$ if and only if the jump times of $\mathsf{ZigZag}_m$ converge to the jump times of $\mathsf{ZigZag}$.

Since every zig-zag is determined by the path's finite collection of jump times in the interval $[x, y]$, the topology of jump times on the collection of zig-zags is generated by a complete metric, and is \emph{compact}.

One can now proceed by noticing that the length of a path \(\ell(\pi)\), can be viewed as a continuous map from \(\mathcal{Z}\) to $\R$. A compactness argument now implies that the supremum over admissible paths is indeed attained.

Last passage percolation enjoys the following \textbf{metric composition law}, Lemma 3.2 in DOV \cite{DOV}.

\begin{lemma}[Metric composition law]\label{lemma: Metric Composition}
    Let \(x\le y \in \mathbb{R}\), \(m < \ell\in\mathbb{Z}\) and \(f\in \mathcal{C}^I\). If \(k\in \{m, \dots, \ell\}\), then we have
    \[
    f[(x,\ell)\to(y,m)] = \displaystyle \sup_{z\in[x,y]}(f[(x,\ell)\to(z,k)]+f[(z,k)\to(y,m)])\,,
    \]
    and if \(k\in \{m+1, \dots, \ell\}\), then 
    \[
    f[(x,\ell)\to(y,m)] = \displaystyle \sup_{z\in[x,y]}(f[(x,\ell)\to(z,k)]+f[(z,k-1)\to(y,m)]).
    \]
    Furthermore for any \(z\in [x,y]\), 
    \begin{equation}\label{eq: composition}
    f[(x,\ell)\to(y,m)] = \displaystyle \sup_{k\in \{m, \dots, \ell\}}(f[(x,\ell)\to(z,k)]+f[(z,k)\to(y,m)])
    \end{equation}
\end{lemma}%Metri composition Law

We end with a preliminary result on the locations of geodesics in last passage percolation, that is an immediate consequence of the metric composition law, Lemma \ref{lemma: Metric Composition}.

\begin{lemma}\label{lemma: geod location brown env lpp}
    Fix $2\le \ell \le n$, $a< b\in \R$ and some $f\in\mathcal{C}^n([a, b])$. Then for a last passage geodesic $\pi$ from $(a, n)$ to $(b, 1)$, the jump times $Z_\ell(a, b)$ from level $\ell$ to level $\ell-1$ satisfy 
    \begin{equation}\label{eq: geod jump argmax brownian abs env lpp}
        Z_\ell(a, b) \in \operatorname*{argmax}_{s\in [a, b]}(f[(a, n)\to (s, \ell)]+f[(s, \ell-1)\to (b,1)])\,.
    \end{equation}
\end{lemma}

\begin{remark}
    For fixed $a< b$, if the environment $f\in\mathcal{C}^n[a, b]$ is random and $f-f(a)$ is absolutely continuous with respect to $n$ independent Brownian motions on $[a, b]$, then by \cite[Lemma B.1]{hammond2019patchwork} (also Lemmas \cite[2.15, 2.16]{dauvergne2022disjointoptimizersdirectedlandscape}), there is a unique geodesic almost surely, that is the left-most and right-most geodesics coincide almost surely. For instance, the above is true for the Airy line ensemble and Dyson Brownian motion (for Dyson Brownian motion, this claim is only true for intervals away from the origin, that is $a> 0$), discussed in Sections \ref{subsec: melons} and \ref{subsec: Airy line ensemble}. Note Brownian $n$-melons and $n$-level Dyson Brownian motion are equal in law for all $n\ge 1$. For Dyson Brownian motion, this is essentially due to the Brownian Gibbs property, see Definition 2.2 in \cite{corwin2014brownian} and Lemmas 2.15-2.16 in \cite{dauvergne2022disjointoptimizersdirectedlandscape}. 
    
    Under this supposed local absolute continuity with respect to Brownian motion, one obtains almost surely for all rational $q< r$ in $[a, b]$ and $m \le k$, the existence of unique geodesics in $f$ with endpoints $(q, k)$ and $(r, m)$. Henceforth, we will tacitly assume in such environments, the existence of unique geodesics starting and ending at countably many endpoints.
\end{remark}

Geodesics in last passage percolation exhibit a `tree-like' structure. This is the content of the following lemma.

\begin{lemma}\label{lemma: geod tree}
    Let $f:\N\times\R\to\R$ be any sequence of random continuous functions such that when restricted to $\{1, \ldots, n\}\times [a, b]$ for $a< b$, $f-f(a)$ is absolutely continuous with respect to independent Brownian motions. Then, for any fixed $x< y< z \in \R$ and $m\le n$, almost surely the set of overlap of geodesics in the environment $f$ starting from $(x, n), (y,n)$ and ending at $(z, m)$
    \begin{align*}
        I:= \bigg\{ s\in [y, z]: [&\pi[(x, n)\to(z,m)](s), \lim_{r\to s^-}\pi[(x, n)\to(z,m)](r)]\\
        \cap  [&\pi[(y, n)\to(z,m)](s), \lim_{r\to s^-}\pi[(y, n)\to(z,m)](r)]\neq \emptyset\bigg\}
    \end{align*}
    is a non-empty interval (with a possibly empty interior).
\end{lemma}

\begin{proof}
    Since the paths $\pi_1: = \pi[(x, n)\to(z,m)]$, $\pi_2:= \pi[(y, n)\to(z,m)]$ have the same endpoint, the interval $I$ contains $y$, and is not empty. Now, suppose $s< t$ both in $[x, y]$ are also in $I$. Then, we show that all $u\in (s, t)$ belong to $I$ also. Indeed, suppose for a contradiction $u$ was not in $I$, then the restrictions of the zig-zags of $\pi_1, \pi_2$ to $[s, t]$ would not coincide. This would result in a `geodesic bubble' (see Figure \ref{fig: geod tree}, and so by the restriction property of geodesics, one would then be able to violate geodesic uniqueness in $f$, by switching between the two (red and black) line segments forming the `geodesic bubble'. 
\end{proof}

\begin{remark}
    This lemma also gives as a result the \emph{geodesic monotonicity} property on a locally Brownian environment $f$ as in the statement therein. This means for any endpoints $(x, n)$, $(y, n)$ as in the above statement, the zig-zags of the geodesic starting from $x$ is to the left and below (above in the figure due to the reverse ordering) that of $y$. One can obtain an entirely analogous ordering of geodesics when considering distinct endpoints $(z, m)$ and $(w, m)$, $z < w$.
\end{remark}

\begin{figure}[hb]
    \centering
    \includegraphics[width=0.5\linewidth]{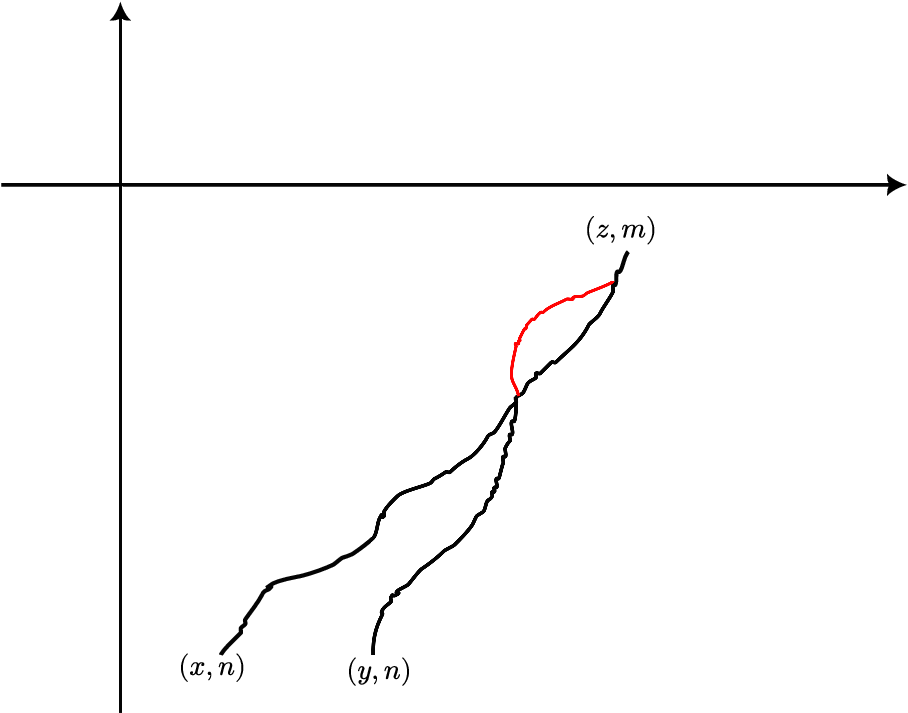}
    \caption{"Once they touch, they coalesce." Two geodesics $ \pi[(x, n)\to(z,m)]$, $\pi[(y, n)\to(z,m)]$ in a Brownian environment $f$ as in the statement of Lemma \ref{lemma: geod tree} must coalesce after their first point of contact. For otherwise, this would result in the existence of some \color{red}red \color{black} path, creating a `geodesic bubble'; this would violate up to a null set the uniqueness of geodesics in the environment $f$.}
    \label{fig: geod tree}
\end{figure}

We now state a certain inequality between last passage paths which follows from Lemma \ref{lemma: Metric Composition} and `planarity' of geodesics. For a proof, see \cite{Basu2021Fractal,DOV,balazs2020non,cator2019attractiveness}.

\begin{lemma}[Crossing lemma]\label{lemma: crossing}
Let $f:\N\times\R\to\R$ be any sequence of continuous functions. Let $m_1\le m_2\le n_1\le n_2$, and $y_1\le y_2\le x_1\le x_2$. Then we have
\[
f[(y_1,n_1)\to(x_1,m_1)] + f[(y_2,n_2)\to(x_2,m_2)] \ge
f[(y_1,n_1)\to(x_2,m_2)] + f[(y_2,n_2)\to(x_1,m_1)].
\]
\end{lemma}

We end with a characterisation of the disjointness of last passage geodesics in terms of a strict inequality between last passage values. 

\begin{lemma}\label{lemma: crossing equiavlence}
        Let $f:\N\times\R\to\R$ be any sequence of random continuous functions such that when restricted to $\{1, \ldots, n\}\times [a, b]$ for $a< b$, $f-f(a)$ is absolutely continuous with respect to independent Brownian motions. Then, with $1\le m_1\le  m_2\le n_1\le n_2$ and $y_1\le y_2\le x_1\le x_2$ as in Lemma \ref{lemma: crossing}, the event
    \[
    \pi[(y_1,n_1)\to(x_1,m_1)] \; \text{and}\; \pi[(y_2,n_2)\to(x_2,m_2)]\;\text{are disjoint},
    \]
    is equivalent (up to null sets) to the event
    \[
    f[(y_1,n_1)\to(x_1,m_1)] + f[(y_2,n_2)\to(x_2,m_2)] >
    f[(y_1,n_1)\to(x_2,m_2)] + f[(y_2,n_2)\to(x_1,m_1)]
    \]
    where $\pi[\cdot \to \cdot]$ denote the almost sure unique geodesics in the environment $f$. Here, we emphasise disjointness is in terms of the \emph{zig-zags} of the paths in question.
\end{lemma}

\begin{proof}
    We prove the forward implication. Observe that the metric composition law, Lemma \ref{lemma: Metric Composition} implies that for any $a< b$, $m\le n$ and geodesic $\pi$ from $(a, n)$ to $(b, m)$ in the environment $f$, (recall Definition \ref{def: length})
    \begin{align}\label{eq: last passage zigzag}
        &\ell(\pi|_{[a, z]}) = f[(a, n)\to (z, \pi(z))]\; \text{and}\nonumber\\
        &\ell(\pi|_{[z, b]}) = f[(z, j)\to (b, m)]\;\text{for all}\; j\in \{\pi(z), \ldots, \lim_{t\to z^-}\pi(t)\}.
    \end{align}

    Now, by planarity, there exists some $(z, \ell)$ that lies on the intersection of the induced zigzags (recall their definition in \eqref{eq: zig-zag}) of $\pi[(y_1,n_1)\to(x_2,m_2)]$ and $\pi[(y_2,n_2)\to(x_1,m_1)]$. 
    
    On the event 
    \[
    \pi[(y_1,n_1)\to(x_1,m_1)] \; \text{and}\; \pi[(y_2,n_2)\to(x_2,m_2)]\;\text{are disjoint},
    \]
    $(z, \ell)$ must not be in the zig-zag of at least one of $\pi[(y_1,n_1)\to(x_1,m_1)]$ and $\pi[(y_2,n_2)\to(x_2,m_2)]$. Since geodesics are unique, we can without loss of generality, suppose that $(z, \ell)$ does not lie on $\pi[(y_1,n_1)\to(x_1,m_1)]$. This in conjunction with \eqref{eq: last passage zigzag} gives the inequalities
    \[
    f[(y_1,n_1)\to(x_1,m_1)] > f[(y_1,n_1)\to(z,\ell)] + f[(z,\ell)\to(x_1,m_1)]
    \]
    and
    \[
    f[(y_2,n_2)\to(x_2,m_2)] \ge  f[(y_2,n_2)\to(z,\ell)] + f[(z,\ell)\to(x_2,m_2)]\,.
    \]
    Adding the above inequalities gives the conclusion by observing (using \eqref{eq: last passage}) that 
    \begin{align*}
        f[(y_1,n_1)\to(z,\ell)] + &f[(z,\ell)\to(x_1,m_1)] + f[(y_2,n_2)\to(z,\ell)] + f[(z,\ell)\to(x_2,m_2)]\\
        = &f[(y_1,n_1)\to(x_2,m_2)] + f[(y_2,n_2)\to(x_1,m_1)].
    \end{align*}

    The reverse implication is proved in an entirely analogous manner. In fact, this is a deterministic result and can be established by considering the contrapositive of this implication, by a switching argument (see the proof of Lemma \ref{lemma: geod tree}).
\end{proof}

\subsection{Pitman transform}

Recall that with $f=(f_1,f_2)$ where $f_i:[0,\infty)\mapsto \R$ for $i=1,2$, for $f\in \mathcal{C}^2[0,\infty)$, we define $\mathrm{W}f=(\mathrm{W} f_1,\mathrm{W} f_2)\in \mathcal{C}^2[0,\infty)$, the \textbf{Pitman transform} of $f$ as follows. Define the maximal gap size
\[G(f_1,f_2)(x,y):=\max\left(\max_{s\in [x,y]}\big(f_2(s)-f_1(s)\big)\,,\,0\right)\,,\qquad 0\le x \le y\,.
\]
Then define
\begin{equation}\label{eq: pitmantrans}
\mathrm{W} f_1(t)=f_1(t)+G(f_1,f_2)(0,t)\,,
\end{equation}
$$\mathrm{W} f_2(t)=f_2(t)-G(f_1,f_2)(0,t)\,,$$
for all $t\in [0,\infty)$.

One can express the top line of the Pitman transform in terms of last passage values. The proof is a straightforward application of Lemma \ref{lemma: Metric Composition}, and so we omit it.

\begin{lemma}\label{lemma: Pitman melon}
    Let \(f\in \mathcal{C}^2([0, \infty))\) and let \(Wf = (Wf_1, Wf_2)\) be as above. Then for all \(t\in[0,\infty)\),
    \[
    Wf_1(t) = \displaystyle \max_{i=1,2}\{f_i(0) + f [(0, i) \to (t, 1)]\} .
    \]
\end{lemma}%Two curve Melon

In the case where \(f_1(0) = f_2(0) = 0\), we obtain that 
\[
Wf_1(t) = f [(0, 2) \to (t, 1)].
\]
\(Wf\) is commonly referred to as the \textbf{2-melon} (which will be generalised in the following section to the so-called $n$-melons) of $f$, since paths in \(Wf\) do not cross (and so preserve their ordering), they resemble the stripes of a watermelon.

\subsection{Melons}\label{subsec: melons}An application of the above that is of interest is that of two independent Brownian motions (starting from zero) \(B = (B_1, B_2)\).  Let \(\hat{B}=(\hat{B}_1, \hat{B}_2)\) be two independent Brownian motions conditioned not to collide, in the sense of Doob with entrance law $(0, 0)$ (in other words, a $2$-Dyson Brownian motion with the same diffusivity as $B$ starting from the origin). Then, the law of the melon \(WB^2\) as defined above in (\ref{eq: pitmantrans}) is the same as that of \(\hat{B}\). In \cite{o2002representation}, a generalisation was proved for \(n\) Brownian motions, using a continuous analogue of the Robinson–Schensted–Knuth (RSK) correspondence, where each level in the \textit{$n$-melon} $WB^n = (WB^n_1, WB^n_2, \cdots, WB^n_n)$ is obtained from a family of $n$ Brownian motions by a sequence of deterministic operations that are analogous to the sorting algorithm `bubble sort'. Moreover, by Theorem 7 in \cite{o2002representation} for any fixed $n\ge 1$, the lw of the Brownian $n$-melon $WB^n_1$ coincides with that of an $n$-level Dyson Brownian motion. For a more complete definition of melons involving the remaining lines, see \cite[sec. 2]{DOV} and \cite{o2002representation}.

\subsection{Airy line ensemble}\label{subsec: Airy line ensemble}
After appropriate rescaling, the melons $WB^n$ converge in law as $n\to \infty$ to a non-intersecting random
ensemble \(\mathcal{A} = (\mathcal{A}_1, \mathcal{A}_2, \dots)\) in $\mathcal{C}^{\N}$ (Theorem 2.1 in \cite{DOV}), such that $\mathcal{A}_1 > \mathcal{A}_2 > \cdots$. The random ensemble $\mathcal{A}$ is called the (\textbf{parabolic)
Airy line ensemble}. It was introduced by Pr\"{a}hofer and Spohn \cite{prahofer2002scale} in the version \((\mathcal{R}_i)_{i\in \N}\coloneqq (\mathcal{A}_i(\cdot)+(\cdot)^2)_{i\in \N}\), which is stationary
in time, see also \cite{corwin2014brownian} and \cite{corwin2014ergodicity}. We will thus call it the \textbf{stationary Airy line ensemble}. The top line $\mathcal{A}_1$ is known as the parabolic $\text{Airy}_2$ process that appears as the
limiting spatial fluctuation of random growth models starting from a single point.

\begin{theorem}\label{thm: melon scaling Airy}
    Let \(WB^n\)  be a Brownian \(n\)-melon constructed from a family of independent standard (rate one) Brownian motions $B_1, \ldots, B_n$.  Define the rescaled melon \(A^n = (A^n_1, \dots, A^n_n)\) for $1\le i \le n$ and $y\ge -n^{1/3}/2$ by
\[
A^n_i(y) = n^{1/6} \left((WB^n)_i(1 + 2yn^{-1/3}) - 2\sqrt{n} - 2yn^{1/6} \right).
\]
Then \(A^n\) converges to a random sequence of functions \(\mathcal{A} = (\mathcal{A}_1, \mathcal{A}_2, \dots) \in \mathcal{C}^\mathbb{N}\) in law with respect to product of uniform-on-compact topology on \(\mathcal{C}^\mathbb{N}\). Almost surely, for all \(y \in \mathbb{R}\) and \(i < j\), we have \(\mathcal{A}_i(y) > \mathcal{A}_j(y)\). The function \(\mathcal{A}\) is called the \textbf{(parabolic) Airy line ensemble.} 
\end{theorem}

We now recall the Brownian Gibbs resampling property enjoyed by the Airy line ensemble, first established in \cite{corwin2014brownian}. Informally, it states that for $a< b$, $k\in \N$, the law of the Airy line ensemble restricted to $\{1,2,\cdots,k\}\times(a,b)$, ${\mathcal{A}}|_{\{1,2,\cdots,k\}\times(a,b)}$, conditionally on all the data generated by the Airy line ensemble outside of this region, $\mathscr{F}_k:= \sigma(\{\mathcal{A}_i(x): (i,x)\notin \llbracket 1,k\rrbracket \times (a,b)\})$, is given by non-intersecting Brownian bridges with entry data $\underline{x} = (\mathcal{A}_i(a))_{1\le i\le k}$, $\underline{y} = (\mathcal{A}_i(b))_{1\le i\le k}$ and also conditioned to stay above $f = \mathcal{A}_{k+1}$ on $(a,b)$.

\subsection{Airy counting estimates}

We now record some useful lemmas regarding the concentration of points at a given spatial location in the Airy line ensemble. Crucial inputs come from asymptotics of the expected number of particles above a certain threshold, \cite[Lemma 5.1]{Dauvergne_2021}, in particular the Bernoulli representation of this threshold \cite[Lemma 5.2]{Dauvergne_2021}. 

\begin{lemma}\label{lemma: airy counting 1}
Fix $\ell \ge 1$. Then there exist universal constants $c, C>0$ such that for all $a> 0$, 
    \begin{align*}
 \PP\bigl(\mathcal{A}_\ell(0)\le-a\bigr)\le C\exp\bigl(\ell-ca^{3/2}\bigr)\,.
\end{align*}

\end{lemma}
\begin{proof}
    Let $N_a=\#\{i:\mathcal{A}_i(0)\ge-a\}$ count the number of points in the Airy line ensemble at zero that exceed the threshold $-a$. By
\cite[Lemma~5.1(i)]{Dauvergne_2021}, the expectation of $N_a$ satisfies the asymptotics
\[
 \mathbb{E} N_a=\frac{2}{3\pi}a^{3/2}+O(1)
 \qquad a\to\infty\,.
\]
Moreover, $N_a$ has the distribution of a sum of independent Bernoulli
random variables, as recalled in the proof of
\cite[Lemma~5.2]{Dauvergne_2021}. Denoting their success probabilities by $(p_j)_{j\ge 1}$,
then
\[
 \mathbb{E} \mathrm{e}^{-N_a}
 =\prod^\infty_{j=1}\bigl(1-(1-\mathrm{e}^{-1})p_j\bigr)
 \le \exp\bigl(-(1-\mathrm{e}^{-1})\mathbb{E} N_a\bigr).
\]
Markov's inequality applied to $\mathrm{e}^{-N_a}$ gives
\begin{align*}
 \PP\bigl(\mathcal{A}_\ell(0)\le-a\bigr)
 &\le\PP(N_a\le\ell)\le e^\ell\mathbb{E} \mathrm{e}^{-N_a}\le C\exp\bigl(\ell-ca^{3/2}\bigr)\,, \qquad a\to \infty
\end{align*}
for some universal constants $c, C> 0$. Enlarging $C$ extends the above bound to all $a> 0$.
\end{proof}

For every $k\ge 1$, the typical value of point on the Airy line ensemble at level $k$ is $\mathcal{A}_k(0) \sim -(3\pi k/2)^{2/3}$, see \cite{Dauvergne_2021}. We now prove concentration for upper tail events of these points at the same scale as above.
\begin{lemma}\label{lemma: airy counting 2}
For every sufficiently small $a>0$, there exist constants
$c_a,C_a>0$ such that
\[
    \PP\bigl(\mathcal{A}_k(0)>-4a k^{2/3}\bigr)
    \le C_a\exp\bigl(-c_a k\log(k+1)\bigr),
    \qquad k\ge1.
\]
\end{lemma}

\begin{proof}
As in Lemma \ref{lemma: airy counting 1}, set $  N_R:=\#\{i\ge1:\mathcal{A}_i(0)\ge-R\}$.
The Airy counting estimates in \cite[Lemma~5.1]{Dauvergne_2021} give
\[
    \mathbb{E} N_R=(R/\kappa)^{3/2}+O(1),
    \qquad
    \operatorname{Var}(N_R)=O(\log R),
    \qquad R\to\infty\,.
\]
where $\kappa = (3\pi/2)^{2/3}$. Moreover, recall from Lemma \ref{lemma: airy counting 1}, $N_R$ has the law of a sum of independent Bernoulli
random variables. Choose $a>0$ so small that $(4a/\kappa)^{3/2}<\frac12$ and set $R_k=4a k^{2/3}$. For all sufficiently large $k$, we estimate the mean and variance of of $N_{R_k}$
\[
    \mu_k:=\mathbb{E} N_{R_k}\le k/2,
    \qquad
    v_k:=\operatorname{Var}(N_{R_k})\le C_a\log k\,,
\]
for some constant $C_a> 0$.

Now, for a Bernoulli random variable $X$ with success probability $p$, we can estimate the moment generating function of the centred random variable $Y:= X-p$, obtaining by a Taylor expansion using $|Y|\le 1$, $\mathbb{E}Y = 0$ and $\mathbb{E}Y^2 = p(1-p)$,
\[
\mathbb{E}\mathrm{e}^{\theta Y} \le 1 + p(1-p)(\mathrm{e}^\theta - 1- \theta)\,,\qquad \theta \ge 0\,.
\]
By the above Bernoulli representation for $N_R$, the above bound gives
\[
    \log\mathbb{E} \mathrm{e}^{\theta(N_{R_k}-\mu_k)}
    \le v_k(e^\theta-1-\theta),
    \qquad \theta\ge 0\,.
\]
The Markov inequality applied to $\mathrm{e}^{\theta(N_{R_k}-\mu_k)}$ for a suitable choice of $\theta \ge 0$ gives
\[
    \PP(N_{R_k}-\mu_k\ge b)
    \le
    \exp\left[-v_k h\left(\frac{b}{v_k}\right)\right]\,, \quad b \ge 0\,, \qquad h(r):=(1+r)\log(1+r)-r\,.
\]
For $b=k-\mu_k\ge k/2$ and $v_k\le C_a\log k$, the exponent is
at most $-c_a k\log k$, for some constant $c_a> 0$. Indeed, this follows from the bounds
\[
\begin{aligned}
v_k h(b/v_k)
&= (v_k+b)\log(1+b/v_k)-b\\
&\ge b\bigl[\log(1+b/v_k)-1\bigr]\\
&\ge \frac{k}{2}
\left[
\log\!\left(1+\frac{k}{2C_a\log k}\right)-1
\right]\\
&\ge c_a k\log k,
\end{aligned}
\]
for all sufficiently large \(k\). Hence we arrive at the tail bounds,
\[
    \PP(N_{R_k}\ge k)\le C_a \mathrm{e}^{-c_a k\log k}\,,
\]
for all $k$ large enough. Since the inclusion $\{\mathcal{A}_k(0)>-R_k\}\subseteq\{N_{R_k}\ge k\}$ holds for any $k\ge 1$, the result follows after adjusting the constants to accommodate the remaining finitely many values of $k$.
\end{proof}

We now establish sub-gaussian concentration bounds for the upper tail of last passage values over the (stationary) Airy line ensemble using the uniform modulus of continuity result in \cite{wu2025applicationsoptimaltransportdyson}. For a similar statement, the reader is referred to Lemma 5.7 in \cite{dauvergne2026downthemiddle}.

\begin{lemma}
\label{lemma: finite-depth-tail}
Fix $k\ge 1$ and $T>0$. Then, there exist universal constants $c,C>0$ such that
\[
    \PP\bigl(\mathcal{A}[(0, k)\to (T, 1)]+T^2>u\bigr)
    \le \exp\left(Ck-\frac{cu^2}{T}\right)
\]
for every $k\ge1$, $T>0$, and $u>0$.

In particular, for every fixed $a,T>0$, there exist constants
$c_{a,T},C_{a,T}>0$ such that
\[
    \PP\bigl(\mathcal{A}[(0, k)\to (T, 1)]>a k^{2/3}\bigr)
    \le C_{a,T}\exp(-c_{a,T}k^{4/3}).
\]
\end{lemma}

\begin{proof}
Fix $k\ge 1, T> 0$ and $u> 0$.

Fix a finite mesh $D\subset[0,T]$ containing $0$ and $T$. Moreover, for a continuous environment $f=(f_1,\ldots,f_k)$, define the `coarse-grained' last passage values with respect to the mesh $D$,
\[
    \Lambda_{k, D}(f;T)
    :=
    \sup_{\substack{0=t_k\le t_{k-1}\le\cdots\le t_0=T\\t_i \in D}}
    \sum_{i=1}^k
       \bigl(f_i(t_{i-1})-f_i(t_i)\bigr).
\]
Since the environment is continuous, we have the monotone convergence of $\Lambda_{k, D_n}(f;T)$ to the actual last passage value $f[(0, k) \to (T, 1)]$ for a family $(D_n)_n$ of nested meshes with mesh size tending to zero as $n\to \infty$. For every $\lambda>0$, $n\ge 1$ the functional
\[
    f\longmapsto
    \exp\bigl(\lambda\Lambda_{k,D_n}(f;T)\bigr)
\]
is convex and depends on finitely many evaluations of $f$.
Indeed, $\Lambda_{k,D_n}$ is a finite maximum of linear
functionals, and the exponential is convex and increasing. 

Fix $\lambda > 0$ and some mesh $D$ of $[0, T]$. Let $\underline{x} = (x_1, \ldots, x_n)$ be such that $x_1 > x_2 > \ldots > x_n$. Let $\Gamma^{n, \underline{x}}: =(\Gamma^{n, \underline{x}}_1, \ldots \Gamma^{n, \underline{x}}_n)$ by a $n$-level rate one Dyson Brownian motion starting from $\underline{x}$. Since its law on $\{1, \ldots, n\}\times D$ is a log-concave perturbation of independent Brownian motions (see Lemma 3.2 in \cite{wu2025applicationsoptimaltransportdyson}), an application of the contraction result in Proposition 2.9 of \cite{wu2025applicationsoptimaltransportdyson} gives the estimate
\[
    \mathbb{E}\exp\bigl(\lambda\Lambda_{k,D}(\Gamma^{n, \underline{x}} - \mathbb{E}[\Gamma^{n, \underline{x}}];T)\bigr)
    \le
    \mathbb{E}\exp\bigl(\lambda\Lambda_{k,D}(B;T)\bigr),
\]
where $B:= (B_1, \ldots, B_n)$ is an independent family of standard Brownian motions. By \cite[Proposition 4.3.5]{anderson2010introduction}, as $\underline {x}\to \underline{0}$ pointwise, the laws of $\Gamma^{\underline{x}}$ converge in path space $\mathcal{C}^n(\R_{\ge 0})$ to those of $n$-level Dyson Brownian motion starting from the origin (recall it has the same law as the $n$-melon $WB^n$ from Subsection \ref{subsec: melons}). Also, using Lemma 3.2 in \cite{konig2001eigenvalues}, it is not hard to establish the convergence of expectations $\mathbb{E}[\Gamma^{n, \underline{x}}(t)]$ to $\mathbb{E}[\Gamma^{n, \underline{0}}(t)] = \mathbb{E}[WB^n(t)]$ for all $t\ge 0$. Thus, by Fatou's lemma, taking $\underline{x}\to 0$, we have 
\[
    \mathbb{E}\exp\bigl(\lambda\Lambda_{k,D}(WB^n - \mathbb{E}[WB^n];T)\bigr)
    \le
    \mathbb{E}\exp\bigl(\lambda\Lambda_{k,D}(B;T)\bigr)\,.
\]
Now, in accordance with the edge scaling in Theorem \ref{thm: melon scaling Airy}, the above estimate in conjunction with the invariance of Brownian motion under diffusive scaling gives
\[
    \mathbb{E}\exp\bigl(\lambda\Lambda_{k,D}(A^n - \mathbb{E}[A^n];T)\bigr)
    \le
    \mathbb{E}\exp\bigl(\lambda\Lambda_{k,D}(\sqrt{2}B;T)\bigr)\,,
\]
where
\[
A^n_i(y) = n^{1/6} \left((WB^n)_i(1 + 2yn^{-1/3}) - 2\sqrt{n} - 2yn^{1/6} \right)\,,\quad 1\le i \le n\,.
\]
By Skorokhod's representation theorem, we can take the $A^n_i$ to converge locally uniformly to some coupled version of the Airy line ensemble $\mathcal{A}$. Moreover, by
Lemmas 2.2-2.3 \cite{gustavsson2005gaussian} and \cite{ledoux2010small} (for details, see the proof of Lemma 5.7 in \cite{dauvergne2026downthemiddle}), the expectations $\mathbb{E}[A^n_i(y)]\to \mathbb{E}[\mathcal{A}_i(y)]$ as $n\to \infty$ for all $i\ge 1$, $y\in \R$. Thus, another application of Fatou's lemma gives the estimate (recalling that the last passage values of an environment where each level is shifted by the same function are related to the original last passage values simply by the addition of the increment of the perturbing function from the initial to the terminal spatial values)
\[
     \mathbb{E}\exp\bigl(\lambda\Lambda_{k,D}(\mathcal{R};T)\bigr) = \mathbb{E}\exp\bigl(\lambda\Lambda_{k,D}(\mathcal{A}-\mathbb E\mathcal{A};T)\bigr)
    \le
    \mathbb{E}\exp\bigl(\lambda\Lambda_{k,D}(\sqrt{2} B;T)\bigr)\,,
\]
with $\mathcal{R}(\cdot) = \mathcal{A}(\cdot) + (\cdot)^2$ the stationary Airy line ensemble. 

Taking nested meshes $D_n$ with dyadic times, continuity and monotone convergence yield
\[
    \mathbb{E}\exp\bigl(\lambda\mathcal{A}[(0, k) \to (T, 1)]+\lambda T^2\bigr)=\mathbb{E}\exp\bigl(\lambda\mathcal{R}[(0, k) \to (T, 1)]\bigr)
    \le \mathbb{E}\exp\bigl(\sqrt{2}\lambda B[(0,k)\to(T,1)]\bigr)\,.
\]
By the Brownian last-passage/GUE identity (see \cite{baryshnikov2000random}), the last passage value $B[(0,k)\to(T,1)]$ has the
law of the largest eigenvalue of a suitably normalized
$k\times k$ GUE matrix. Gaussian concentration, \cite[Proposition 3.1]{ledoux2007deviation}, gives
\[
\PP\big(B[(0,k)\to(1,1)] \ge r + \mathbb{E}B[(0,k)\to(1,1)]\big)\le 2\mathrm{e}^{-\frac{r^2}{2}}\,,\quad r\ge 0\,.
\]
Since by equation (3.4) in \cite[p. 194]{ledoux2007deviation}, $\mathbb{E}B[(0,k)\to(T,1)] \le 2\sqrt{2kT}$, we estimate by Fubini's theorem and the invariance of Brownian motion under diffusive rescaling,
\[
\mathbb{E}\exp\bigl(\sqrt{2}\lambda B[(0,k)\to(T,1)]\bigr)
\le C\exp\bigl(A\lambda\sqrt{kT}+D\lambda^2T\bigr)\,,\qquad \lambda > 0\,,
\]
for absolute constants $A,C, D > 0$. An application of Markov's inequality gives for any $\lambda > 0$
\begin{align}\label{eq: moment gen func airy lpp}
     \PP\bigl(\mathcal{A}[(0, k)\to (T, 1)]+T^2>u\bigr)
    &\le \frac{\mathbb{E}\exp\bigl(\lambda\mathcal{A}[(0, k) \to (T, 1)]+\lambda T^2\bigr)}{\exp \lambda u}\\
    &\le  \frac{\mathbb{E}\exp\bigl(\sqrt{2}\lambda B[(0,k)\to(T,1)]\bigr)}{\exp \lambda u}\\
    &\le C\mathrm{e}^{A\lambda\sqrt{kT}+ D\lambda^2 T-\lambda u}\qquad u\ge0\,.
\end{align}
Taking $\lambda= (u-A\sqrt{kT})/(2DT)$, we have
\[
\PP\bigl(\mathcal{A}[(0,k)\to(T,1)]+T^2>u\bigr)
\le C\exp\left(
    -\frac{(u-A\sqrt{kT})^2}{4DT}
\right)\,, \qquad u \ge A\sqrt{kT}\,.
\]
Extending the bound trivially for all $0 \le u \le A\sqrt{kT}$ we obtain
\begin{align*}
    \PP\bigl(\mathcal{A}[(0, k)\to (T, 1)]+T^2>u\bigr)\le \exp\left(Ck -\frac{cu^2}{T}\right)\qquad u\ge 0\,,
\end{align*}
for some absolute constants $c, C > 0$.

For the final part, simply use $u = T^2 + ak^{\frac{2}{3}}$.
\end{proof}

The above upper tail concentration, which matches its Brownian counterpart has a two-sided counterpart, albeit with worse decay. In particular, using the bridge representation of \cite{Dauvergne_2021}, we obtain the following concentration for last passage values across the Airy line ensemble.

\begin{lemma}[A quantitative last-passage estimate]\label{lemma: shape}
Fix $0<a<b<\infty$, then there exist some (possibly $a, b$-dependent) constants $c,C>0$ such that for
$T\in[a,b]$, $k\ge3$ and $1\le s\le k^{1/14}$,
\begin{align}\label{eq:shape}
 &\PP\left(\left|\mathcal{A}[(0,k)\to(T,1)]+T^2
                 -2\sqrt{2kT}\right|>Ck^{3/7}s^{3/2}\right)\notag\\
 &\hspace{25mm}\le Ck^2e^{-cs}+Cke^{-ck^{1/63}}.
\end{align}
\end{lemma}

\begin{proof}
The bridge representation of \cite[Theorem~7.2]{Dauvergne_2021} couples the top $k$ curves of the Airy line ensemble with another bridge ensemble $\mathcal B$ on $[0,T]$, close in total variation distance. Motivated by the Brownian Gibbs property for the Airy line ensemble, $\mathcal{B}$ is defined by first sampling $\mathcal{A}$ on a lattice approximation of $[0, T]\times \{1, \ldots, k\}$ and then for every line interval on the lattice, sampling independent rate two Brownian bridges with endpoints the values of $\mathcal{A}$ on the lattice. Finally, for adjacent endpoints that are sufficiently close together, a non-intersection condition on the Brownian bridges is imposed.

One can precisely construct $\mathcal{B}$ as follows. Fix $0< a< b<\infty$, $k, \ell\ge 3$, $\delta > 0$. With $s_j = Tj/\ell$, $0\le j \le \ell$, sample the points $\mathcal{A}_i(s_j)$, $1\le i \le 2k$, $0\le j \le \ell$. Let $G_{2k}$ be the random sub-graph of $\{1, \ldots, 2k\}\times \{1, \ldots, \ell\}$ where two edges $(i, j)$ and $(i+1, j)$, $1\le i \le 2k-1$, $1\le j \le \ell$ are connected if an only if at least one of $|\mathcal{A}_{i+1}(s_j)-\mathcal{A}_{i}(s_{j})|$, $|\mathcal{A}_{i+1}(s_{j-1})-\mathcal{A}_{i}(s_{j-1})|$ is at most $\delta$. Then for every $(i, j) \in \{1, \ldots, 2k\}\times \{1, \ldots, \ell\}$ sample independent rate two Brownian bridges $B_{i, j}: [s_{j-1}, s_j]\to \R$ vanishing at both endpoints. The ensemble $\mathcal{B}$ is obtained by connecting up all of the bridges $B_{i, j}$ by random affine terms $L_{i, j}$ so that for all $(i, j)\in \{1, \ldots, 2k\}\times \{1, \ldots, \ell\}$, $\mathcal{B}_i(s_j) = \mathcal{A}_i(s_j)$ and then imposing the non-intersection condition on the shifted bridges $B_{i, j} + L_{i, j}$ and $B_{i+1, j}+L_{i+1, j}$, for $1\le i \le 2k-1$, $1\le j \le \ell$ if an only if $(i, j)$ and $(i+1, j)$ are in the same connected component of the graph $G_{2k}$.

Theorem~1.1 in \cite{Dauvergne_2021} then gives that with the choice of parameters 
\[
 \gamma>0,\qquad
 \ell=\lceil Tk^{2/3+\gamma}\rceil,\qquad \delta = k^{-1/3-\gamma/4}\,,
\]
for $\gamma> 0$ to be chose later, the ensembles $\mathcal{B}|_{ \{1, \ldots, k \}\times [0, T]}$ and $\mathcal{A}|_{ \{1, \ldots, k \}\times [0, T]}$ are within $Ck^{2/3+\gamma}\mathrm{e}^{-ck^{\gamma/12}}$ in total variation distance, for some absolute constant $c>0$ and some $C>0$ (whence the second term in \eqref{eq:shape} for appropriate $\gamma$). Here we absorbed $T$ into $C$, and taking into account that $a\le T
\le b$, suppressing dependence on $a, b$. There thus exists a coupling between $\mathcal{A}, \mathcal{B}$ such that the first $k$ lines agree on $[0, T]$ with probability at least $1-Ck^{2/3+\gamma}\mathrm{e}^{-ck^{\gamma/12}}$. We will henceforth work with this coupling.

We enlarge the probability space with $N_{i, j}$ independent centred Gaussian with unit variance and independent rate two Brownian bridges $W_{i,j}$, $1\le i \le k$, $1\le j \le \ell$ with domains $[s_{j-1}, s_j]$ vanishing at both endpoints, also jointly independent from the bridge ensemble $\mathcal{B}$ and the $N_{i, j}$. Condition on the values of the Bridge ensemble on the lattice $\{\mathcal{B}_i(s_j) : 0\le j \le \ell, 1\le i \le 2k\}$. Note in particular, the random graph $G_{2k}$ depends on the Airy line ensemble on this lattice. Then, for any isolated vertex $(i, j)$, observe that the Bridge ensemble $\mathcal{B}$ on $[s_{j-1}, s_j]\times\{i\}$ Brownian is a deterministic affine shift of a rate two Brownian bridge vanishing at both endpoints also independent from $\mathcal{B}$ on $[s_{j-1}, s_j]\times\{i'\}$, $i'\neq i$. If on the other hand $(i, j)$ is a non-isolated vertex of $G_{2k}$, one can express the Bridge ensemble $\mathcal{B}$ on $[s_{j-1}, s_j]\times\{i\}$ as the sum of a deterministic affine term depending on the Bridge ensemble lattice values, the Brownian Bridge $W_{i, j}$ and another random continuous $X_{i, j} := B_{i, j} - W_{i, j}$. In this case, the Brownian bridge $B_{i, j}$, vanishing at both endpoints, is conditioned to not intersect $U_{i, j}-1$ other bridges, where $U_{i, j}$ denotes the size of the connected component in $G_{2k}$ containing $(i,j)$. Concatenating the above paths, and undoing the conditioning on the Bridge ensemble lattice values gives the path decomposition for the line $\mathcal{B}$,
\begin{equation}\label{eq: bridge decomp}
    \mathcal B_{i}(s)= H_{i}(s)+R_{i}(s)+X_i(s)\,,\qquad 1\le i\le k\,, 0\le s\le T\,,
\end{equation}
where the $H_i$ are independent rate two Brownian motions on $[0, T]$ constructed from the Brownian bridges and affine shifts of Gaussian $N_{i,j}$, $R_{i}$ is a random continuous, piecewise affine term and $X_{i}=0$ on $[s_{j-1}, s_j]$ for $(i, j)$ isolated in $G_{2k}$ and $X_{i} = B_{i, j}-W_{i,j}$ on $[s_{j-1}, s_j]$ for non-isolated $(i,j)$. 

The uniform modulus of continuity estimates for the Airy line ensemble in \cite[Corollary~1.3]{wu2025applicationsoptimaltransportdyson}, 
gives that exist universal constants $C_1, C_2>0$ such that,
for any $i\ge 1$, $1\le j \le \ell$, $K\geq 0$, 
    \begin{align*}
    \mathbb{P}\left(\sup_{t,s\in [s_{j-1},s_j], t\neq s}  \frac{|\mathcal{A}_i(t)+t^2-\mathcal{A}_i(s)-s^2|}{\sqrt{|t-s| \log (2(s_j-s_{j-1})/|t-s|)}} > K \right)\leq C_1\mathrm{e}^{-C_2K^2}\,.
\end{align*}  
and similarly for the rate two Brownian motions $H_i$.

These estimates, combined with the elementary bounds
\begin{equation*}
\sqrt{|t-s|\log 2h/(t-s)}\le \sqrt{h}\sup_{0< \varepsilon \le 1}\sqrt{\varepsilon\log2/\varepsilon}\le Ch^{1/2}\,,\qquad 0 < t-s\le h
\end{equation*}
where $h = T/\ell$ (taking $k$ sufficiently large, $\ell \ge T$ and so $h\le 1$, which we henceforth assume), in conjunction with a union bound over $i, j$ gives for fixed $s\ge 1$ on an event $\mathrm{Osc}_{s, k}$ of probability at least $1- Ck\ell \mathrm{e}^{-cs}$,
\[
 \max_{1\le i\le k,\,1\le j\le \ell}
 \left\{\mathsf{osc}_{[s_{j-1},s_{j}]}(\mathcal{A}_i),
        \mathsf{osc}_{[s_{j-1},s_{j}]}(H_i)\right\}
\le C\sqrt{hs}+C\max_{j}|s^2_j-s^2_{j-1}| \le C\sqrt{hs}+Ch\,,
\]
where for $f\in \mathcal{C}[a,b]$, we set
\[
\mathsf{osc}_{[a, b]}(f):= \sup_{t,s\in [a,b], t\neq s}|f(t)-f(s)| \,.
\]
On the event where $\mathcal{B}|_{ \{1, \ldots, k \}\times [0, T]}= \mathcal{A}|_{ \{1, \ldots, k \}\times [0, T]}$, the above bounds and the bridge decomposition in \eqref{eq: bridge decomp} imply the $X_i$ and the slopes $R'_i$ of the linear terms $R_i$ satisfy the bounds
\[
 \max_{1\le i\le k,1\le j\le \ell}\mathsf{osc}_{[s_{j-1},s_{j}]}X_i\le C\sqrt{hs}+Ch\,,\qquad \text{ and } \displaystyle \max_{1\le i\le k}\|R_i'\|_{\infty, T}\le C\sqrt{s/h}+C\,.
\]
Here we use $\norm{\cdot}_{\infty, T}$ to denote the uniform norm of a continuous function on $[0, T]$. 

\begin{figure}[hb]
    \centering
    \includegraphics[width=0.65\linewidth]{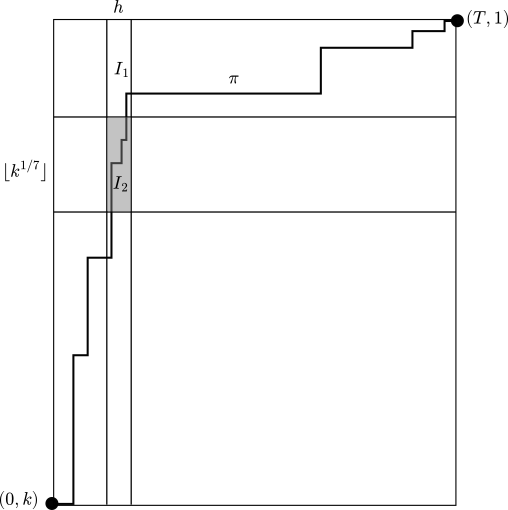}
    \caption{Illustration of path $\pi$ from $(0, k)$ to $(T, 1)$ superimposed on a `coarse-grained' lattice given by slices $[s_{j-1}, s_j]\times \{i\}$ of width $h$, for $(i, j) \in G_{2k}$. For $1\le s \le k^{1/14}$, on the event $E_{k,s}$, the Airy line ensemble in any box $I_n$, $1\le n \le N_k$ is up to $Cs$ lines, simply independent rate two Brownian motions with random linear terms interpolating between the values of $\mathcal{A}$ on the lattice $\{(s_j, i): 0\le j \le \ell\,, 1\le i \le k\}$. Now, the boxes $I_n$ hit by $\pi$ (shaded in \color{gray} gray\color{black}) are at most $N_k+\ell$, as they themselves can be thought of as a `coarse-grained' version of $\pi$.}
    \label{fig: airy lpp coarse graining}
\end{figure}

Fix $\alpha=1/7$ and $3\gamma/4=1/7$, $1\le s \le k^{1/14}$. Partition $1\le i \le k$ into $N_k$ blocks $(I_{n})_{1\le n \le N_k}$ of $\lfloor k^{1/7}\rfloor$ consecutive indices (except for say some residual block, which can be augmented to a block of size $\lfloor k^{1/7} \rfloor$ using indices from $\{k+1, \ldots, 2k\}$ if necessary). By \cite[Proposition 7.4]{Dauvergne_2021}, for any fixed $1\le j \le \ell$, the size of the set
\[
V_j := \{i: 1\le i \le k\,, (i, j) \text{ is not isolated in } G_{2k}\}
\]
intersected with $I_n$ exceeds $Cs$ with probability at most $Ce^{-cs}$. The restriction $s\le k^{1/14}$ is precisely the parameter restriction in the statement of \cite[Proposition~7.4]{Dauvergne_2021}. A union bound over the $N_k$ such blocks and $1\le j \le \ell$ gives
\[
\displaystyle \PP\left(\bigcup_{\substack{1\le n \le N_k\\ 1\le j \le \ell}}\{|V_j\cap I_n|\ge Cs\}\right) \le Ck^2\cdot \mathrm{e}^{-cs}\,.
\]

Define events
\begin{align*}
    A_{k} &:= \{\mathcal{B}|_{ \{1, \ldots, k \}\times [0, T]} \neq \mathcal{A}|_{ \{1, \ldots, k \}\times [0, T]}\}\\
    V_{k, s}&:= \left\{\begin{aligned}
        &\text{There exist } 1\le j \le \ell, 1\le n \le N_k \text{ such that the total}\\
        &\text{number of non-isolated points } (i, j)\in G_{2k} \text{ where } i \text{ is in}\\
    &\text{the block } I_n \text{ of }\lfloor k^{1/7}\rfloor \text{ consecutive indices exceeds } Cs
    \end{aligned}\right\}\,.
\end{align*}
We have thus established that with the choice $\gamma = 4/21$,
\[
\PP(A_k\cup V_{k, s}\cup \mathrm{Osc}^c_{s, k})\le Ck^2\cdot \mathrm{e}^{-cs}+Ck\mathrm{e}^{-ck^{1/63}}\,,\qquad 1\le s \le k^{1/14}\,.
\]

We henceforth suppose we are in the good event $E_{k,s}:= A^c_k\cap V^c_{k, s}\cap \mathrm{Osc}_{s,k}$. Now the zig-zag (recall from Definition \ref{eq: zig-zag}) of a non-increasing path $\pi:[0, T] \to \N$ from $(0, k)$ to $(T, 1)$ (recall the definition of last passage percolation, Definition \ref{def: length} in Section \ref{sec: preliminaries}) visits at most $N_k+\ell$ rectangles $I_n\times [s_{j-1}, s_j]$, which with the above choice of $\gamma = 4/21$, is at most $Ck^{6/7}$. For an illustration, see Figure \ref{fig: airy lpp coarse graining}. In each visited rectangle, $X_i$ does not identically vanish on at most $Cs$
line segments, on the event $E_{k, s}$. Thus, on $E_{k, s}$ the lengths $\ell^R(\pi)$, $\ell^X(\pi)$ of $\pi$ in the environments give by the ensembles $R_i$ and $X_i$, respectively satisfy
\[
 |\ell^R(\pi)|\le CT\bigl(\sqrt{s/h}+1\bigr),\qquad
 |\ell^X(\pi)|\le Ck^{6/7}s(\sqrt{hs}+h).
\]
Fix $0 < a < b$. Then for $a\le T\le b$, we have $C_a k^{-6/7} \le h \le  C_b k^{-6/7}$, for some constants $C_a, C_b> 0$. Taking suprema over $\pi$ yields on the event $E_{k, s}$
\[
 \left|\mathcal B[(0,k)\to(T,1)]+T^2-H[(0,k)\to(T,1)]\right|
 \le Ck^{3/7}s^{3/2}\,,
\]
absorbing dependence on $a, b$ into $C$ and the $T^2$ term can be absorbed by enlarging the constant $C$, since $k^{3/7}s^{3/2}\ge 1$, where we again suppress dependence on $a, b$. Finally, the Brownian last-passage tail estimates of Theorem 2.5 in \cite{DOV} give the bounds
\[
\sup_{a\le T\le b}\PP(|H[(0,k)\to(T,1)]-2\sqrt{2kT}|\ge Ck^{3/7}s^{3/2})\le C\mathrm{e}^{-cs}\,,
\]
again, absorbing $a, b$ dependence into $c, C>0$.

Combining the above two-sided estimates and rarity of $E^c_{k, s}$ gives the bounds \eqref{eq:shape}.
\end{proof}

\subsection{Airy sheet}\label{subsec: Airy sheet}
The standard Airy sheet $\mathcal{S}:\R^2\to \R$ is a random continuous function and can be thought of as assigning `distances' between pairs of space-times points $(x, 0)$ and $(y, 1)$. It is closely related to the Airy line ensemble, a connection which we will use frequently.

We collect some invariance properties of the Airy sheet. From Lemma 9.1 \cite{DOV} and Proposition 14.1 in \cite{dauvergne2022scalinglimitlongestincreasing}, we have that as a random continuous function in $\R^2$, for any $c \in \R$ the Airy sheet satisfies
\begin{equation}\label{eq: Airy sheet symmetry}
\mathcal{S}(x, y) \stackrel{\mathrm{d}}{=} \mathcal{S}(-y, -x)\,, \quad \mathcal{S}(x, y) \stackrel{\mathrm{d}}{=} \mathcal{S}(-x, -y) \quad \mathrm{and} \quad \mathcal{S}(x, y) \stackrel{\mathrm{d}}{=} \mathcal{S}(x, y + c) + 2c(y-x) + c^2.
\end{equation}
Also, almost surely for all $x \le x'$, $y\le y'$, we have
\begin{equation}\label{eq: quadrangle ineq airy sheet}
\mathcal{S}(x, y) + \mathcal{S}(x', y') \ge \mathcal{S}(x, y') + \mathcal{S}(x', y)\,.
\end{equation}
Observe in particular that the symmetries \eqref{eq: Airy sheet symmetry}, give that the process $\mathcal{S}(x, \cdot) + (\cdot -x)^2$ is stationary for all $x\in \R$. Moreover, by Lemma 5.3 \cite{sarkarthrehalves2022} the Airy sheet satisfies the almost sure pointwise bounds 
\begin{equation}\label{eq: airy shape bnds}
|\mathcal{S}(x,y)+(x-y)^2|\le \mathfrak{C}+c\log^{2/3}(2+|x|+|y|)\,,\qquad \text{ for all } x,y\in \R 
\end{equation}
for some universal constant $c>0$ and some $\mathfrak{C}\ge 0$ satisfying $\mathbb{E}[a^{\mathfrak{C}^{3/2}}]<\infty$ for some universal constant $a>1$, \cite{sarkarthrehalves2022}.  Moreover, by Proposition 10.5 in \cite{DOV} Airy sheet has a jointly H\"{o}lder $1/2^-$-continuous modification.

We end this subsection by making more precise the connection between the Airy sheet and the Airy line ensemble. In fact the law of the Airy sheet on the positive half-plane $\R_{\ge 0} \times \R$ is uniquely determined by that of the Airy line ensemble (cf. \cite[Definition 8.1]{DOV}). In particular, we have the following coupling in terms of the Airy line ensemble Busemann profiles in \eqref{eq: Airy limit busemann}.

\begin{definition}(Airy sheet coupling)\label{coupling: Airy sheet}
The Airy sheet $\mathcal{S}(\cdot, \cdot)$ can be coupled with the (parabolic) Airy line ensemble $\mathcal{A}$ so that $\mathcal{S}(0,\cdot)=\mathcal{A}_1(\cdot)$ and almost surely for all $(x,y,z)\in (\Q\cap(0, \infty))\times \Q^2$\footnote{The above coupling can alternatively be established for all triplets $(x, y , z)\in K$, where $K$ is any fixed countable dense subset of $(0, \infty)\times \R^2$. Moreover, the above limits also exist simultaneously for all $(x, y, z)\in (0, \infty)\times \R^2$, see remark 8.1 in \cite{DOV}.}, there exists a random integer $K_{x,y,z}$ such that for all $k\ge K_{x,y,z}$ \begin{equation*}\label{eq: defS}
\mathcal{A}\bigg[\bigg(-\sqrt{\frac{k}{2x}}, k\bigg)\to (z,1)\bigg]-\mathcal{A}\bigg[\bigg(-\sqrt{\frac{k}{2x}}, k\bigg)\to (y,1)\bigg]=\mathcal{S}(x,z)-\mathcal{S}(x,y)\,.
\end{equation*}
\end{definition}
 The above coupling can be interpreted as saying the increments of the Airy sheet, themselves encoding lengths of geodesics in the directed landscape, see Section 12 in \cite{DOV} are precisely the relative lengths of coalescing semi-infinite geodesics in the Airy line ensemble, see Definition \eqref{def: semi-inf geo} and Lemma \ref{lemma: semi-inf geod properties}. This therefore suggests there is some sort of duality between the geometric information contained in the Airy sheet and Airy line ensemble. Over the course of Section \ref{sec: geod geom} and Section \ref{sec: exceptional geod geom}, this will hopefully become more apparent.

In the recent work of \cite{dauvergne2026downthemiddle}, the authors constructed a `down the middle' coupling between the Airy sheet and Airy line ensemble on the quarter-plane $\R_{\ge 0} \times \R_{\ge 0}$ using only finite depth last passage values over the Airy line ensemble. We will explore the connection between the semi-infinite Busemann functions in \eqref{eq: Airy limit busemann} and finite depth last passage values across the Airy line ensemble in Theorem \ref{thm: Busemann equality in law}.

\section{Semi-infinite geodesics: the basics}\label{sec: semi-inf geod}

In this section, we establish existence and uniqueness of semi-infinite geodesics in the Airy line ensemble, and establish some structural properties thereof, resembling geodesics in last passage percolation more generally. An analysis of the prelimiting Brownian melons in Theorem \ref{thm: melon scaling Airy} provides some basic structural results about semi-infinite geodesics in the Airy line ensemble.

We start with a definition.
\begin{definition}\label{def: semi-inf geo} For any fixed $x > 0$ and $y\in \R$, $\ell \ge 1$. A non-increasing cadlag path
\[
\pi[x\to (y, \ell)]: (-\infty, y] \longrightarrow \{\ell, \ell+1,\ldots\}
\]
is called an $x$-directed \emph{semi-infinite geodesic} with endpoint $(y, \ell)$ if all restrictions thereof to bounded intervals are geodesics, and 
\[
\lim_{t\to -\infty}\frac{\pi[x\to (y, \ell)](t)}{t^2} = 2x\,.
\]
\end{definition}

It is not clear a priori that these infinite paths exist. The fact that these semi-infinite paths exist is the content of Lemma \ref{lemma: semi-inf geod properties}.

\begin{lemma}\label{lemma: semi-inf geod properties}
	\begin{enumerate}
		\item[(1)] For any fixed $x > 0$, $y\in \R$, $\ell\ge 1$, almost surely there exists a unique semi-infinite geodesic $\pi[x\to (y, \ell)]$ in $\mathcal{A}$ from $x$ to $(y, \ell)$. 
        \item[(2)] There exists an event $\Omega_0$ of full probability such that on $\Omega_0$, for every $x\ge 0$, $y\in \R$, there exist semi-infinite geodesics $\pi^L[x \to (y, \ell)]$ and $\pi^R[x \to (y, \ell)]$ from $x$ to $(y, \ell)$ satisfying $\pi^L[x \to (y, \ell)](t) \le\pi (t) \le\pi^R[x \to (y, \ell)](t)$ for any geodesic $\pi$ from $x$ to $(y, \ell)$ and all $t \in (-\infty, y]$. We call $\pi^L[x \to (y, \ell)]$ and $\pi^R[x \to (y, \ell)]$ the \textbf{leftmost and rightmost geodesics from $x$ to $(y, \ell)$}.
		\item[(3)] (Coalescence) For a fixed $x > 0$ and $y, y' \in \R$, $\ell, \ell'\ge 1$ and any geodesics $\pi$ and $\pi'$ from $x$ to $(y, \ell)$ and $x$ to $(y', \ell')$, almost surely we have $\pi(z) = \pi'(z)$ for all sufficiently negative $z$.\\
        \item[(4)] (Shift invariance) For any $x > 0$, $y\in \R$ and $\ell \ge 1$, we have the equality in law between
        \[
        \pi[x\to (0, \ell)](\cdot) \text{ and } \pi[x\to (y, \ell)](y+\cdot)\,.
        \]
	\end{enumerate} 
\end{lemma}
\begin{proof}
We tacitly restrict ourselves to the event of the existence and uniqueness of geodesics in the Airy line ensemble simultaneously for any countable collection of endpoints, which we may augment as needed. Recall this is guaranteed by the remark after Lemma \ref{lemma: geod location brown env lpp}.

Items $(1), (2), (3)$ with $\ell = \ell' = 1$ are established in Lemma 4.2 of \cite{dauvergne2022disjointoptimizersdirectedlandscape}. Item $(4)$ with $\ell = 1$ is established in Lemma 3.4 of \cite{sarkar2021brownian}. We now focus on $\ell, \ell'\ge 2$. 

Fix $x> 0$, $y\in \R$, $\ell \ge 2$. By \cite[Lemma 3.4]{sarkar2021brownian}, one can construct an $x$-directed semi-infinite geodesic starting ending at $(y, 1)$ as the pointwise limit of geodesics $\pi[x_k\to (y, 1)]$ as $k\to \infty$ (recall from Section \ref{sec: semi-inf geod}, $x_k:=(-\sqrt{k}/\sqrt{2x}, k)$, $k\ge 1$). 

We now show that in the same fashion, for any $\ell \ge 2$ the paths $\pi[x_k\to (y, \ell)]$ also lead to consistent object, that is converge pointwise to a limiting path (recall from Definition \ref{def: length}) $\sigma$, satisfying $\lim_{t\to -\infty}\sigma(t)/t^2 = 2x$ whose restriction on any interval is a geodesic. 

By \cite[Lemma 3.4]{sarkar2021brownian}, that is the shift invariance (which incidentally is the statement of item $(4)$ with $\ell = 1$) and the parabolic shape of semi-infinite geodesics ending at the top line of the Airy line ensemble, there exists some random integer $Y > y$ such that $\pi[x\to (Y, 1)](y) \ge \ell + 1$. Now, by geodesic uniqueness in the Airy line ensemble, we have the ordering
\[
\pi[x_k\to (y, 1)](t)\le \pi[x_k\to (y, \ell)](t)\le \pi[x_k\to (Y, 1)](t)
\]
for all $k\ge \ell$ sufficiently large, $t\in [-\sqrt{k}/\sqrt{2x}, y]$. Taking $k\to \infty$, using the convergence of $\pi[x_k\to (y, 1)]$ and $\pi[x_k\to (Y, 1)]$ as $k\to \infty$ we see the jump times of $\pi[x_k\to (y, \ell)]$ remain bounded and thus have sub-sequential limits. We now show they have a unique limit. 

By \cite[Lemma 3.4]{sarkar2021brownian}, the geodesics $\pi[x\to (y, 1)]$ and $\pi[x\to (Y, 1)]$ eventually coalesce. Furthermore, again by Lemma 3.4 in \cite{sarkar2021brownian}, the jump times of $\pi[x\to (y, 1)]$ and $\pi[x\to (Y, 1)]$ eventually take distinct values and diverge to the left as the depth diverges. This means there exists some rational $T\le y$, $d(T)\ge \ell$ for which $(T, d(T))$ is on the zig-zags of both $\pi[x_k\to (y, 1)]$ and $\pi[x_k\to (Y, 1)]$ almost surely eventually in $k\ge 1$. Now, by geodesic uniqueness, one can obtain for all $k$ sufficiently large $\pi[x_k \to (y, \ell)]$ as a concatenation of $\pi[x_k\to (y, 1)]$ restricted to $[-\sqrt{k}/\sqrt{2x}, T]$ and the unique geodesic from $(T, d(T))$ to $(y, \ell)$. Thus, the limit of $\pi[x_k \to (y, \ell)]$ in the topology of jump times exists, which we call $\pi[x \to (y, \ell)]$ and is a semi-infinite geodesic with asymptotic direction $x> 0$. 

This establishes existence of semi-infinite geodesics ending at $(y, \ell)$, $y\in \R$, $\ell \ge 2$. We now show that the asymptotic shape in \eqref{def: semi-inf geo} and local geodesic property are sufficient to almost surely uniquely determine the above semi-infinite geodesics. Recall that item (1) with $\ell = 1$ has been established in Lemma 4.2 of \cite{dauvergne2022disjointoptimizersdirectedlandscape}. Now, suppose $\sigma$ is some semi-infinite geodesic with asymptotic direction $x> 0$ ending at $(y, \ell)$. Let $Y$ be a random integer as above. Then for all $0 < \varepsilon < x$ rational, by the asymptotic parabolic shape of $\sigma$ and geodesic uniqueness in the Airy line ensemble, one has that $\sigma$ will lie to the right of $\pi[x-\varepsilon \to (y, 1)]$ and to the left of $\pi[x+\varepsilon \to (Y, 1)]$. This ordering is then preserved taking the limit $\varepsilon \to 0$. By their aforementioned uniqueness, the semi infinite paths $\pi[x-\varepsilon \to (y, 1)]$ and $\pi[x+\varepsilon \to (Y, 1)]$ are monotone in $\varepsilon$. By the continuity of the Airy line ensemble and they thus converge to semi-infinite geodesics with asymptotic direction $x> 0$, being pointwise limits of $x$-directed semi-infinite geodesics. By Lemma 4.2 in \cite{dauvergne2022disjointoptimizersdirectedlandscape}, they converge almost surely to the semi-infinite geodesics $\pi[x \to (y, 1)]$ and $\pi[x \to (Y, 1)]$. Thus, due to their eventual coalescence, the path $\sigma$ also eventually agrees with $\pi[x \to (y, 1)]$ and $\pi[x \to (Y, 1)]$, which is then uniquely determined up to $y$ by geodesic uniqueness in the Airy line ensemble. This concludes the proof of item (1).

The same argument also proves the coalescence of semi-infinite geodesics with the same fixed asymptotic direction almost surely, item $(3)$ for $\ell \ge 2$. 

We now establish item (2), namely, the simultaneous existence for all $x\ge 0, y\in \R, \ell \ge 1$ of leftmost and rightmost geodesics with asymptotic directions $x$ ending at $(y, \ell)$ on an event of full probability. We will show that this is a deterministic fact on the event (of full probability) of the existence and uniqueness of semi-infinite geodesics with rational endpoints and asymptotic directions. We first consider $x> 0$. Indeed, we have from the previous discussion, we have almost surely for all $q\in \Q\cap (0, \infty), r\in\Q$ there exists a $q$-directed semi-infinite geodesic starting ending at $(r, \ell)$, $\pi[q\to (r, \ell)]$. Now to construction leftmost and rightmost geodesics for any $x> 0, y\in \R$, simply consider the monotone limits
 \[
     \pi^L[x\to (y, \ell)] = \sup_{\substack{n\in \N\,,\\\Q\ni z_n \nearrow x\,, \Q\ni y_n\nearrow y}}\pi[z_n, (y_n, \ell)]\text{ and } \pi^R[x\to (y_n, \ell)] = \inf_{\substack{n\in \N\,,\\\Q\ni w_n \searrow x\,, \Q\ni y_n\searrow y}}\pi[w_n, (y_n, \ell)]\,.
\]
The asymptotic parabolic shape can also be computed the limiting profiles of the approximating geodesics. Moreover, being limits of semi-infinite geodesics in the topology of jump times, the restrictions of limiting zig-zags on bounded intervals are geodesics by the continuity of the Airy line ensemble. The left-most/right-most properties are inherited by the uniqueness of geodesics with rational endpoints in the Airy line ensemble, which gives monotonicity of geodesics with rational endpoints. For a similar argument, see \cite[Lemma 3.5]{DOV} establishing the right left continuity of right/left most geodesics.

Now, to treat the case $x = 0$, note for any $y\in \R$ the path $\sigma \equiv \ell$ on $(-\infty, y]$ is locally a geodesic and has asymptotic shape $\lim_{t\to -\infty} \sigma(t) /(2t^2) = 0$. Thus, it is a $0$-directed semi-infinite geodesic. In particular, it is automatically the left-most such geodesic. Now, to construct the rightmost such geodesic, argue by approximation as above, taking the sequence of rational directions $w_n \searrow 0$.

To prove item $(4)$ when $\ell \ge 2$, it suffices to observe that for any fixed $c\in \R$, $\mathcal{A}(\cdot+c)\overset d = \mathcal{A}(\cdot) - 2ct - c^2$ on $\mathcal{C}^\N$, by the stationarity of $\mathcal{A}(\cdot) + (\cdot)^2$. This is because under this symmetry, translations of geodesics in the Airy line ensemble by a deterministic amount are also geodesics in the shifted Airy line ensemble.
\end{proof}

We focus on item $(2)$ from Lemma \ref{lemma: semi-inf geod properties} regarding leftmost and rightmost semi-infinite geodesics and record a very useful approximation lemma for them (which holds more generally for geodesics in last passage percolation over environments which give almost sure uniqueness thereof). 

\begin{lemma}[Approximation]\label{lemma: geodesic approx}
    Almost surely, for any $x\ge 0$, $y\in \R$, $\ell\ge 1$ the left-most and right most semi-infinite geodesics $\pi^L[x\to (y, \ell)], \pi^R[x\to (y, \ell)]$ from $x$ to $(y, \ell)$ can be expressed as the pointwise limits
    \[
     \pi^L[x\to (y, \ell)] = \sup_{\substack{n\in \N\,,\\\Q\ni z_n\nearrow x, \Q\ni y_n \nearrow y}}\pi[z_n, (y_n, \ell)]\,,\qquad x> 0\,,
     \]
     and $\pi^L[0\to (y, \ell)]\equiv 0$. Moreover,
     \[
     \pi^R[x\to (y, \ell)] = \inf_{\substack{n\in \N\,,\\\Q\ni w_n \searrow x, \Q\ni y_n \searrow y}}\pi[w_n, (y_n, \ell)]\,.
    \]
    Moreover, almost surely, for all $-\infty < t \le y$, with $\square \in \{L, R\}$, the restrictions $\pi^\square[z_n, (y, \ell)]|_{[t, y]}$ and $\pi^\square[w_n, (y, \ell)]|_{[t, y]}$ agree with $\pi^L[x\to (y, \ell)]|_{[t, y]}$ and $\pi^R[x\to (y, \ell)]|_{[t, y]}$ for all $n$ sufficiently large (depending on $t$).
\end{lemma}

\begin{proof}
The proof of the first part essentially follows from the ordering property of geodesics in the Airy line ensemble, see the proof of item $(1)$ in Lemma \ref{lemma: semi-inf geod properties}. 

 We now prove the second item. We show that on the event (of full probability, see the remark after Lemma \ref{lemma: geod location brown env lpp}, which we tacitly assume holds) of geodesic existence and uniqueness in the Airy line ensemble, this item is a deterministic fact.
 
In particular, we show that for all $x> 0$, $t\le y\in \R$, $\square \in \{L, R\}$ and rational approximating sequences $(z_n), (w_n)$ as in the statement, the restrictions $\pi^\square[z_n, (y, \ell)]|_{[t, y]}$ and $\pi^\square[w_n, (y, \ell)]|_{[t, y]}$ coincide with $\pi^L[x\to (y, \ell)]|_{[t, y]}$ and $\pi^R[x\to (y, \ell)]|_{[t, y]}$ for all $n$ sufficiently large (depending on $t$). Arguing as in the proof of item (2) of Lemma \ref{lemma: semi-inf geod properties}, the paths $\pi^\square[z_n, (y, \ell)]$ and $\pi^\square[w_n, (y, \ell)]$ converge in the topology of jump times to $\pi^L[x\to (y, \ell)]$ and $\pi^R[x\to (y, \ell)]$ respectively. The former paths will henceforth be referred to as \emph{approximating paths} and the latter the \emph{limiting paths}.
    
Indeed, as we had before, the pointwise convergence of geodesics implies convergence of jump times. Fix $x> 0$ and $y \in \R$, $\square\in \{L, R\}$. There are two cases to consider. 

Suppose the first jump of $\pi^\square[x\to (y, \ell)]$ is strictly less than the right endpoint $y$. Then as $n\to \infty$, the convergence of jump times and the uniqueness of geodesics in the Airy line ensemble will give that for any $t\le y$, for all $n$ sufficiently large, the approximating paths coincide with their limits on $[t, y]$. For otherwise, a `geodesic bubble' with rational endpoints between the approximating and limiting paths would form (see the illustration on the left of Figure \ref{fig: geod approx}).

Now, if the first or more jumps of the limiting paths occur at exactly $y$, we will still use a similar geodesic overlap argument, though this time on the approximating paths \emph{themselves}. If we are considering the rightmost path $\pi^R[x\to (y, \ell)]$, the argument is reduced to the previous case, as the jumps of the approximating paths are also forced to be $y$ if they are so for the limiting path. Now, in the left-most case, that is when the first or more jump times of the limiting left-most geodesic are exactly $y$, we claim eventually in $n$, the initial vertical segments of the approximating and limiting paths coincide. Indeed, if this was not the case, by the convergence of jump times of the approximating paths, a `geodesic bubble' between rational endpoints between two approximating paths would occur for sufficiently large $n$, depicted on the right of Figure \ref{fig: geod approx}, again violating geodesic uniqueness. Then, as for the rightmost case, one argues as before.

    \begin{figure}[ht]
        \centering
        \includegraphics[width=0.8\linewidth]{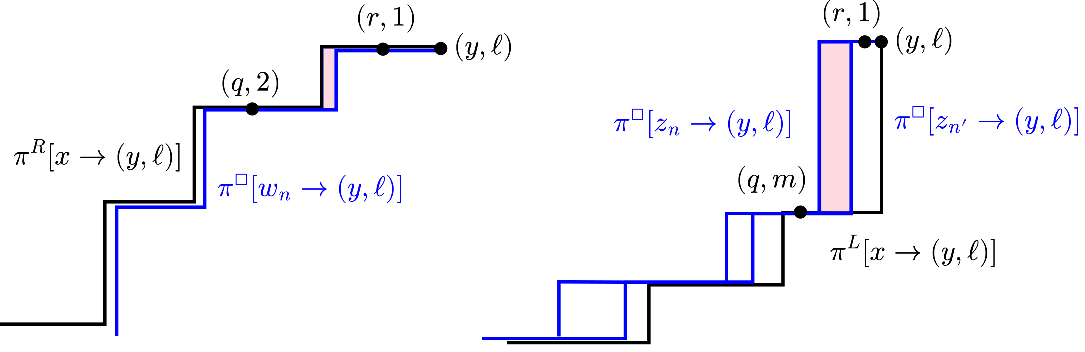}
        \caption{Illustration of possible obstructions to the second part of the statement in Lemma \ref{lemma: geodesic approx}, depending on whether there is at least one jump at $y$ or not. From the moment there is \textbf{left}: overlap between two horizontal segments of geodesics $\pi^R[x\to (y, \ell)]$ and $\pi^\square[w_n\to (y, \ell)]$ (alternatively $\pi^L[x\to (y, \ell)]$ and $\pi^\square[z_n\to (y, \ell)]$) if there is no jump at $y$, or \textbf{right}: $\pi^\square[z_n\to (y, \ell)]$ and $\pi^\square[z_{n'}\to (y, \ell)]$ (with $n'> n \ge 1$, $n, n'$ sufficiently large) on the event at least one jump of $\pi^L[x\to (y, \ell)]$ occurs at $y$, a `geodesic bubble' between rational endpoints forms (here in {\color{pink} pink}), violating geodesic uniqueness.}
        \label{fig: geod approx}
    \end{figure}
\end{proof}

For fixed $x> 0$, $y\in \R$, $\ell\ge 1$, the above pathwise constructions of semi-infinite geodesic $\pi[x\to (y, \ell)]$ as limits in the topology of jump times of the finite horizon geodesics $\pi[(-\sqrt{k}/\sqrt{2x}, k) \to (y, 1)]$, and their eventual coalescence for fixed asymptotic directions allow us to define their `relative lengths', of \emph{Busemann profiles} given by the limits
\begin{equation}\label{eq: Airy limit busemann}
    \Delta^{(t, m)}_{(t', m')}(x):= 
    \displaystyle\lim_{k\to\infty}(\mathcal{A}[(-\sqrt{k/2x},k)\to (t',m')]-\mathcal{A}[(-\sqrt{k/2x},k)\to (t,m)])\,,\quad (t,m) , (t', m')\in \R\times \N
\end{equation}
   In the definition of the Busemann function $\Delta^{(t, m)}_{(t', m')}$, the raised index $(t, m)$ should be thought of as `to the left' and `below' the lowered index $(t', m')$ should thus be thought of a `right' end point, or alternatively as $(t, m)\prec (t', m')$ where $\prec$ denotes the lexicographic ordering on $\R\times \N$ where $\R$ is equipped with its usual ordering and $\N$ with the reversed ordering $\ge$.  Define
   \begin{align}\label{eq: busemann cadlag/caglad}
       x\mapsto \Delta^{L, (t, m)}_{(t', m')}(x):= \lim_{y\searrow x}\Delta^{(t, m)}_{(t', m')}(y) \text{ and } x\mapsto \Delta^{R, (t, m)}_{(t', m')}(x):=\lim_{y\nearrow x}\Delta^{(t, m)}_{(t', m')}(y)
   \end{align}
   as the right and left continuous versions of $\Delta^{(t, m)}_{(t', m')}(x)$, respectively. When making statements that pertain to both left/right continuous versions of these Busemann processes, we will denote them by $\Delta^{\square, (t, m)}_{(t', m')}$, where $\square\in \{L, R\}$.

   It will turn out that these processes encode geometric information of the semi-infinite geodesics, as will become apparent from Section \ref{sec: geod geom} onwards.

\section{Disjointness estimates for semi-infinite geodesics}\label{sec: geod geom}
Having constructed and investigated some qualitative properties of semi-infinite geodesics in the Airy line ensemble, in this section, we will obtain more quantitative results on their geometric structure. In particular, we aim to obtain quantitative control on disjointness of geodesics in the Airy line ensemble. For this, we will need the coupling of the Airy line ensemble with the Airy sheet, \ref{coupling: Airy sheet}, to exploit the duality between geodesics in both models. Moreover, by Lemma \ref{lemma: crossing equiavlence}, we will analyse and obtain pathwise regularity of `Busemann functions' (see \eqref{eq: Airy limit busemann}) encoding the relative lengths of geodesics as they precisely encode this geometric information.

We start with some preliminary results for geodesics in the prelimiting Brownian melon environments in the statement of Theorem \ref{thm: melon scaling Airy}. 

We first recall the \emph{Minkowski} dimension of bounded subsets of Euclidean space.

\begin{definition}(Minkowski dimension)\label{def: Minkowski dimension}
Fix $n\in \N, r\ge 0$ and let $A$ be a bounded subset of $\mathbb{R}^n$, then $A$ is said to have \emph{(upper) Minkowski dimension $r$} if
\begin{align*}
 \displaystyle\limsup_{\varepsilon \rightarrow 0} \frac{\log N(\varepsilon)}{\log 1/\varepsilon} = r\,, 
\end{align*}
 where $N(\varepsilon)$ denotes the infimum of all cardinalities of $\varepsilon$-covers of $A$.
\end{definition}
\begin{figure}[ht]
\begin{center}
\begin{tikzpicture}[line cap=butt]
  % ---- tuning parameters ----
  \def\thick{8pt}            % thickness of bars
  \def\ysep{28pt}            % vertical separation
  \def\arrowoff{12pt}        % arrow distance above bars
  \def\timesscaleval{12}         % horizontal scale

  % smaller epsilon fonts
  \def\varepsilonsonefont{\large}
  \def\varepsilonstwofont{\normalsize}
  \def\varepsilonsthreefont{\small}

  % scope with horizontal scaling
  \begin{scope}[xscale=\timesscaleval]
    % draw Cantor levels (orders 0..5)
    \foreach \order in {0,...,5}{
      \pgfmathsetmacro{\step}{1/(3^\order)}
      \draw[line width=\thick, yshift={-\order*\ysep}]
        l-system[l-system={cantor set, axiom=F, order=\order, step=\step cm}];
    }

    % ---------- ARROWS & LABELS ----------
    % order 0 (top layer): [0,1]
    \begin{scope}[yshift=0pt]
      \draw[<->, >=Stealth, thick] (0,\arrowoff) -- (1,\arrowoff)
        node[midway, font=\varepsilonsonefont, fill=white, inner sep=1pt] {$\varepsilon_{1}$};
    \end{scope}

    % order 1: [0,1/3], [2/3,1]
    \begin{scope}[yshift={-\ysep}]
      \draw[<->, >=Stealth, thick] (0,\arrowoff) -- (1/3,\arrowoff)
        node[midway, font=\varepsilonstwofont, fill=white, inner sep=1pt] {$\varepsilon_{2}$};
      \draw[<->, >=Stealth, thick] (2/3,\arrowoff) -- (1,\arrowoff)
        node[midway, font=\varepsilonstwofont, fill=white, inner sep=1pt] {$\varepsilon_{2}$};
    \end{scope}

    % order 2: [0,1/9], [2/9,1/3], [2/3,7/9], [8/9,1]
    \begin{scope}[yshift={-2*\ysep}]
      \draw[<->, >=Stealth, thick] (0,\arrowoff) -- (1/9,\arrowoff)
        node[midway, font=\varepsilonsthreefont, fill=white, inner sep=1pt] {$\varepsilon_{3}$};
      \draw[<->, >=Stealth, thick] (2/9,\arrowoff) -- (1/3,\arrowoff)
        node[midway, font=\varepsilonsthreefont, fill=white, inner sep=1pt] {$\varepsilon_{3}$};
      \draw[<->, >=Stealth, thick] (2/3,\arrowoff) -- (7/9,\arrowoff)
        node[midway, font=\varepsilonsthreefont, fill=white, inner sep=1pt] {$\varepsilon_{3}$};
      \draw[<->, >=Stealth, thick] (8/9,\arrowoff) -- (1,\arrowoff)
        node[midway, font=\varepsilonsthreefont, fill=white, inner sep=1pt] {$\varepsilon_{3}$};
    \end{scope}
  \end{scope}
\end{tikzpicture}

\end{center}
\caption{Cartoon representation of construction of a Cantor set.}
\label{fig: cantor}
\end{figure}
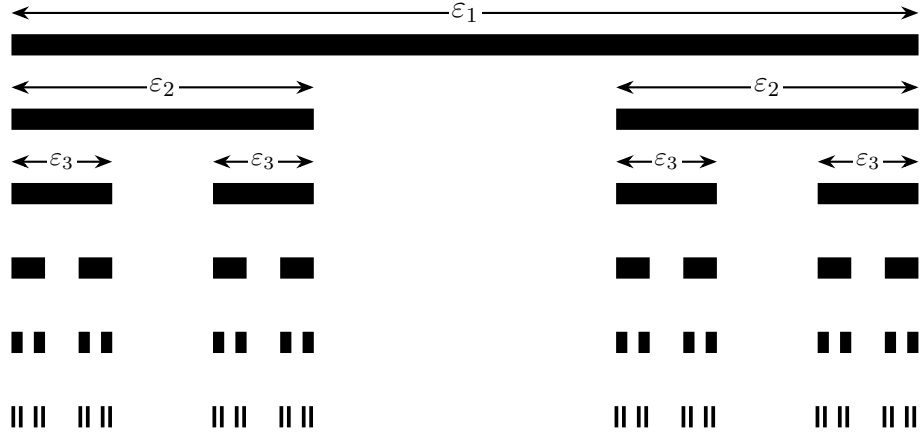

Examples of sets with Minkowski dimension at most $r\in (0, 1)$ include \textbf{compact, perfect and nowhere dense} (hence uncountable) `Cantor-type' sets (see Figure \ref{fig: cantor}) where at each stage from each interval, an interior subinterval is removed. Thus, at stage $n\ge 1$, the set is contained in $2^n$ many intervals of some finite length $\varepsilon_n$ (decreasing in $n\ge 1$ and converging to zero) small enough so that 
\[
\limsup_{n\to \infty} \frac{n\log 2}{\log(1/\varepsilon_{n-1})}<r\,.
\]
One can then estimate by the monotonicity of $N(\varepsilon)$ and $N(\varepsilon_n) \le 2^n$ for $n\ge 1$,
\begin{align*}
 \displaystyle\limsup_{\varepsilon \rightarrow 0} \frac{\log N(\varepsilon)}{\log (1/\varepsilon)}\le  \displaystyle\limsup_{n\to \infty} \frac{\log N(\varepsilon_{n})}{\log (1/\varepsilon_{n-1})} \le \displaystyle\limsup_{n \rightarrow \infty} \frac{n\log 2}{\log (1/\varepsilon_{n-1})}<r\,.
\end{align*} 

In the following lemma, we obtain uniform control (on countable sets of sufficiently small Minkowski dimension) on non-disjointness of geodesics with closely separated left endpoints in the prelimiting melon environments of Theorem \ref{thm: melon scaling Airy}. Recall from Definition \ref{def: disjointness} that two paths $\pi$ and $\tau$ are disjoint if their graphs, when viewed as `zigzags' have a positive separation on compact subsets on the intersection of their domains (which we take to be closed intervals).

For $x> 0, y\in \R$ and $n\in \N$, in accordance with the notation in \cite{DOV} we set
\begin{equation*}\overline{x}=2xn^{-1/3}\,,\qquad \mbox{ and } \qquad \hat{y}=1+2yn^{-1/3}\,.
\end{equation*}
Now let $n\ge 1$ be sufficiently large so that $\overline{x}< \hat{y}$. Consider a family of mutually independent standard Brownian motions $B_1, \ldots, B_n$. We will denote the geodesic from $(\overline{x}, n)$ to $(\hat{y}, 1)$ in the environment $WB^n$ (see Section \ref{subsec: melons}) by $\pi\{\underline{x}\to \hat{y}\}_n$.

\begin{lemma}\label{lemma: close-not-dj}
	Fix $y_1< y_2$ and let $K\subseteq (0, \infty)$ be bounded, bounded away from zero and countable with Minkowski dimension less than $r< 1/2$. Then for all $\varepsilon>0$,
	\begin{equation}\label{eq: close-not-dj}
	 \limsup_{n \to \infty} \PP \bigg(\bigcup_{x\in K} \pi\{\overline{x-\varepsilon}\to \widehat{y_1}\}_n \;\;\; \mathrm{and} \;\;\; \pi\{\overline{x +\varepsilon}\to \widehat{y_2}\}_n \text{ are disjoint} \bigg) \le C\varepsilon^{\frac{1}{2}-r}\cdot \mathrm{e}^{d|\log\varepsilon|^\frac{5}{6}},
	\end{equation}
    for some constants $d, C>0$. 
\end{lemma}

\begin{proof}
	By the remark in lemma \ref{lemma: geod location brown env lpp}, for all $n\in \N$, the geodesics $\pi\{\overline{x-\varepsilon}\to \widehat{y_1}\}$ and $\pi\{\overline{x+\varepsilon}\to \widehat{y_2}\}_n$ are up to a null event unique. Consider the function
    \[
    F^n_{x, \varepsilon}(u):= WB^n[(\overline{x+\varepsilon}, n)\to (\widehat{u}, 1)]- WB^n[(\overline{x-\varepsilon}, n)\to (\widehat{u}, 1)]\qquad u \in [y_1, y_2]\,.
    \]
    By the crossing lemma for geodesics in last passage percolation, Lemma \ref{lemma: crossing}, $F^n_{x, \varepsilon}$ is non-decreasing. Moreover, since the Brownian melons $WB^n$ are locally absolutely continuous with respect to independent Brownian motions (in fact the melons are equal in law to $n$-level Dyson Brownian motion and so satisfy the Brownian Gibbs property, see the remark after lemma \ref{lemma: geod location brown env lpp}), we have by Lemma \ref{lemma: crossing equiavlence}, almost surely for any fixed $u, v \in [y_1, y_2]$
    \begin{equation}\label{eq: melon monoto}
        F^n_{x, \varepsilon}(v) - F^n_{x, \varepsilon}(u) > 0\,, \text{ if and only if } \pi\{\overline{x-\varepsilon}\to \widehat{u}\}_n \;\;\; \mathrm{and} \;\;\; \pi\{\overline{x+\varepsilon}\to \widehat{v}\}_n \text{ are disjoint}\,.
    \end{equation}
    Now, consider the partition of $[y_1, y_2]$,
    \[
    a_0 = y_1, a_1 = \min\{y_1 + \varepsilon, y_2\}\,,\ldots a_k = \min\{y_1 + k\varepsilon, y_2\}\,,\ldots\,, a_N = \min\{ y_1 + N\varepsilon, y_2\}\,, N = \bigg\lceil \frac{y_2-y_1}{\varepsilon} \bigg\rceil\,.
    \]
    Define the events 
    \[
    E^{\varepsilon, x, u, v, n}:=\left\{\pi\{\overline{x-\varepsilon}\to \widehat{u}\}_n \;\;\; \mathrm{and} \;\;\; \pi\{\overline{x+\varepsilon}\to \widehat{v}\}_n \text{ are disjoint} \right\}\,, \quad x > 0\,, y_1 \le u < v \le y_1\,.
    \]
    Then by \eqref{eq: melon monoto} we have by the monotonicity of $F^n_{x, \varepsilon}$ (see Figure \ref{fig: geod config} for the case $N = 2$),
    \begin{align*}
        E^{\varepsilon, x, y_1, y_2, n} &= \{F^n_{x, \varepsilon}(y_2) - F^n_{x, \varepsilon}(y_1) > 0\} = \displaystyle\bigcup^{N-1}_{k = 0}\{F^n_{x, \varepsilon}(a_{k+1}) - F^n_{x, \varepsilon}(a_k) > 0\} = \displaystyle\bigcup^{N-1}_{k = 0}\{E^{\varepsilon, x, a_k, a_{k+1}, n}\}\,.
    \end{align*}
    By \cite[Lemma 4.5]{DOV}, we have up to null sets,for any fixed $y_1 \le u< v \le y_2$ the events $E^{\varepsilon, x, u, v, n}$ are included in the events 
    \begin{align*}
        E^{\varepsilon, x, u, u, n}_B :=\bigg\{\pi[(\overline{x-\varepsilon},n)\to (\widehat{u}, 1)]_n \;\mathrm{and} \; \pi[(\overline{x+\varepsilon},n)\to (\widehat{v}, 1)]_n \text{ are essentially disjoint}\bigg\}\,,
    \end{align*}
    where $B$ is a family of $n$-independent Brownian motions (suppressing the dependence on $n$) that is used in the pathwise construction of the $n$-melon $WB^n$, and $\pi[\cdot \to \cdot]_n$ the almost sure unique geodesics in this environment.
    
    By the invariance of Brownian motion under diffusive scaling and translation and \cite[Theorem 1.1]{Hammond_2019}, we have for all $1\le k \le N$,
\begin{align*}
    \limsup_{n\to \infty}\PP(E^{\varepsilon, x, u, u, n}_B)\le C\varepsilon^{\frac{3}{2}}\cdot\mathrm{e}^{d|\log \varepsilon|^{\frac{5}{6}}}
\end{align*}
and some universal constants $C, d > 0$. We will henceforth enlarge or reduce universal constants $C, d>0$ as needed. 

Informally, the above uniform bounds exist because these events capture the disjointness of rescaled geodesics in Brownian last passage percolation, which, under the (KPZ) $1:2:3$ scaling are tight and converge to geodesics in the Airy sheet. Moreover, the decay in $\varepsilon$ can be interpreted as being due to the rareness of disjoint geodesics in the Airy sheet (over a fixed time horizon) with starting and ending points within $\varepsilon$ of each other.

This gives 
\[
\PP(E^{\varepsilon, x, y_1, y_2, n})\le  C\lceil (y_2-y_1)/\varepsilon\rceil \cdot \varepsilon^{\frac{3}{2}}\cdot\mathrm{e}^{d|\log \varepsilon|^{\frac{5}{6}}}\le C\varepsilon^{\frac{1}{2}}\cdot\mathrm{e}^{d|\log \varepsilon|^{\frac{5}{6}}}\,.
\]

 Now, since $K$ is bounded, consider a family of $N(\varepsilon)$ neighbourhoods $[a_i, b_i]$ in $\R$ that cover $K$ where each neighbourhood has diameter at most $\varepsilon$. Then, we can estimate the above estimates and the monotonicity of geodesics
\begin{align*}
    &\limsup_{n\to \infty}\PP\left(\displaystyle\bigcup_{x\in K} E^{\varepsilon, x, y_1, y_2, n}\right)\\
    &\le \limsup_{n\to \infty}\sum^{N(\varepsilon)}_{i=1}\PP\left(\pi\{\overline{a_i-\varepsilon}\to \widehat{y_1}\}_n \;\;\; \mathrm{and} \;\;\; \pi\{\overline{b_i+\varepsilon}\to \widehat{y_2}\}_n \text{ are disjoint} \right)\\
    &\le C|N(\varepsilon)|\cdot (y_2-y_1) \cdot\varepsilon^{\frac{1}{2}}\mathrm{e}^{d|\log(\varepsilon)|^{\frac{5}{6}}}\,,
\end{align*}
concluding the proof, noting that $K$ having Minkowski dimension $r<1/2$, the covering numbers satisfy $N(\varepsilon)\lesssim \varepsilon^{-r}$.
\end{proof}

\begin{figure}[hb]
    \centering
    \includegraphics[width=\linewidth]{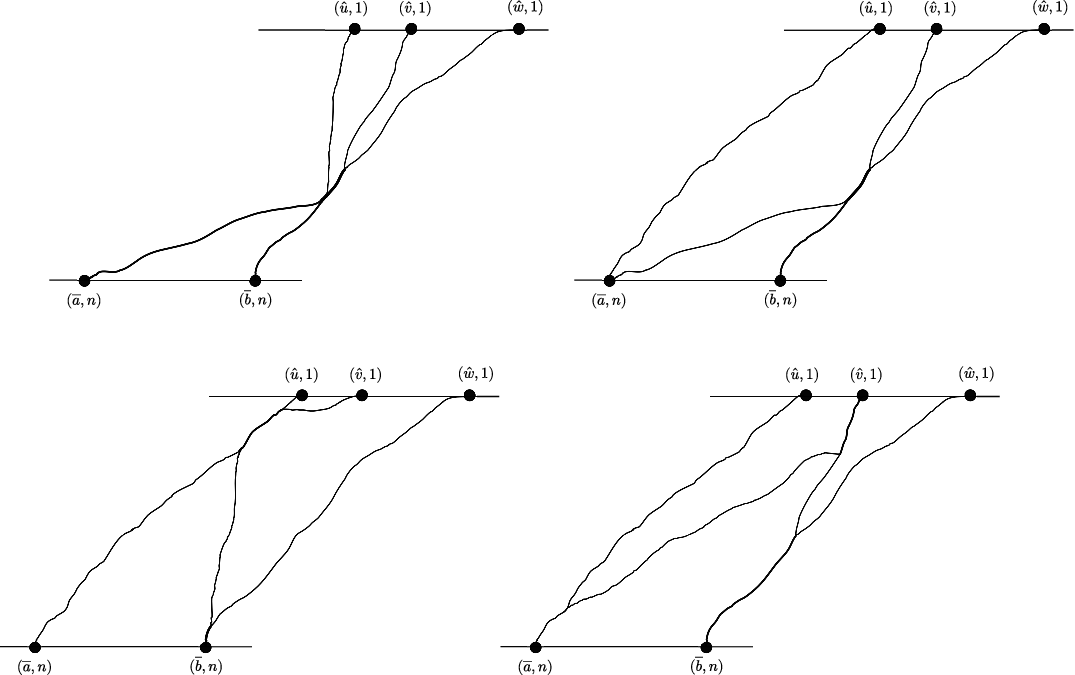}
    \caption{An illustration of the possible geometries of Brownian melon geodesics $\pi\{\overline{x}\to \hat{y}\}_n$, $x\in \{a, b\}$ and $y\in \{u, v, w\}$, owing to the tree-like structure and uniqueness of geodesics with fixed endpoints, Lemmas \ref{lemma: geod tree} and \ref{lemma: geod location brown env lpp}. Disjointness of the outermost geodesics $\pi\{\overline{a}\to \hat{u}\}_n$, $\pi\{\overline{b}\to\hat{w}\}_n$ is determined by checking the disjointness of geodesics $\pi\{\overline{a}\to \hat{u}\}_n$, $\pi\{\overline{b}\to \hat{v}\}_n$, and $\pi\{\overline{a}\to \hat{u}\}_n$, $\pi\{\overline{a}\to \hat{u}\}_n$. This information is precisely what is encoded by the monotonicity of $F^n_{x, \varepsilon}$. }
    \label{fig: geod config}
\end{figure}

\begin{remark}
    Such a rarity of disjointness of geodesics can also be extended in some sense to the case where $K$ is compact and has Hausdorff dimension $r < 1/2$. This need not be an event a priori, being an uncountable union of events. However, it is contained in a proper event by geodesic monotonicity (which can be obtained using the almost sure uniqueness of geodesics in the Brownian melon environments with rational endpoints, see the remark after Lemma \ref{lemma: geod location brown env lpp}) whose probability we can explicitly control.
    
    This is because for any fixed $\beta > r$ less than one half, for all $\varepsilon> 0$, one has a family of covers $(U^\varepsilon_i)$ with $\diam (U_i)\le \varepsilon$ satisfying
    \[
    \lim_{\varepsilon \to 0} \sum_i \mathrm{diam} (U^\varepsilon_i)^\beta = 0.
    \]
    By compactness (Lebesgue's number lemma), there exists some $\eta_\varepsilon > 0$ such that for $x\in K, y > 0$ with $|x-y|\le 2\eta_\varepsilon$, there exists $U^\varepsilon_i$ such that $x, y \in U^\varepsilon_i$. By monotonicity of geodesics, this means we have the inclusion
    \[
    \displaystyle\bigcup_{x\in K} E^{\eta_\varepsilon, x, y_1, y_2, n}\subseteq \displaystyle \bigcup_i \{\pi\{\overline{\inf U^\varepsilon_i}\to \widehat{y_1}\}_n \;\;\; \mathrm{and} \;\;\; \pi\{\overline{\sup U^\varepsilon_i}\to \widehat{y_2}\}_n \text{ are disjoint}\}\,.
    \]
    
    We can thus estimate arguing as in the proof of Lemma \ref{lemma: close-not-dj},
    \begin{align*}
        \limsup_{n\to \infty}\PP\left(\displaystyle\bigcup_{x\in K} E^{\eta_\varepsilon, x, y_1, y_2, n}\right) &\le C\sum_i \mathrm{diam} (U^\varepsilon_i)^{\frac{1}{2}}\mathrm{e}^{d|\log\mathrm{diam} (U^\varepsilon_i)|^{\frac{5}{6}}}\\
        &\le C \varepsilon^{\frac{1}{2}-\beta}\mathrm{e}^{d|\log \varepsilon|^{\frac{5}{6}}} \sup_{\varepsilon> 0}\sum_i \mathrm{diam} (U^\varepsilon_i)^\beta \\
        &\le C_\beta\varepsilon^{\frac{1}{2}-\beta}\mathrm{e}^{d|\log\varepsilon|^{\frac{5}{6}}}\,.
    \end{align*}

We will see later that this condition on the Hausdorff dimension is sharp for semi-infinite geodesics in the Airy line ensemble (which arise as scaling limits of these melon paths, see the coupling \pageref{Skorokhod coupling} below). This is due to exceptional sets of asymptotic directions and endpoints for which geodesic non-uniqueness occurs, see Theorem \ref{thm: airy sheet monotone haus dim}.
\end{remark}

For all $n \ge 1$, let $B^n = (B^n_1, \ldots, B^n_n)$ be families of independent, standard Brownian motions. Fix $n\in \N$ and $k\in \{2,\ldots, n\}$, $x > 0$, $y \in \R$. Let $\hat{Z}^n_k(x,y) = 1+2n^{-\frac{1}{3}}Z_k^n(x,y)$ be the jump time from level $k$ to level $k-1$ in the almost surely unique geodesic $\pi\{\overline{x}\to \hat{y}\}_n$ in the Brownian melon $WB^n$, constructed pathwise from $B^n$. 

We now state the coupling in \cite[p. 43]{DOV} between copies of Brownian melons rescaled $A^n$ in the statement of Theorem \ref{thm: melon scaling Airy}, the Airy line ensemble $\mathcal{A}$, prelimiting Airy sheets $\mathcal{S}^n$ defined on any compact set $K\subset\R^2$ for all $n$ sufficiently large in terms of Brownian last passage values
\[
\mathcal{S}^n(x,y) \overset d = n^{1/6}\bigg( B^n[(\overline{x}, n)\to (\hat{y}, 1)] + 2(x-y)n^{1/6}-2\sqrt{n}\bigg)\,,\qquad (x, y) \in K
\]
and the Airy sheet $\mathcal{S}$, along some deterministic subsequence of indices, which we do not relabel, where the following holds on an event of full probability.
\begin{itemize}\label{Skorokhod coupling}
\item[1.] The rescaled melon $A^n$ converge to the Airy line ensemble $\mathcal{A}$ uniformly on compact subsets of $\N\times \R$.
\item [2.] $\mathcal{S}^n$ converges locally uniformly to the Airy sheet $\mathcal{S}$.
\item[3.]We have for all $(x,y)\in (\Q\cap (0, \infty))\times \Q$,
\[
Z_k^n(x,y)\to Z_k(x,y) \qquad \mbox{ for all } k\ge 2\,.
\]
Moreover, as $k\to \infty$,
\begin{equation*}
Z_k(x,y)/\sqrt{k} \to -1/\sqrt{2x}.
\end{equation*}
\item[4.] For every rational triple $(x, y , z)\in (\Q\cap (0, \infty)) \times \Q^2$ with $y < z$, there exist random rational points $0< X_1 < x < X_2$ so that the geodesics $\pi\{\overline{X_1}, \hat{y}\}_n$ and  $\pi\{\overline{X_2}, \hat{z}\}_n$ are not disjoint for all $n$ large enough.
\end{itemize}
Observe that by the convergence of the melon environments to the Airy line ensemble and the jump times $Z^n_k$, the zig-zags associated to the monotone times $Z_k(x, y)$ for $x\in \Q\cap (0, \infty)$, $y\in \Q$ are locally geodesic. Thus, by item 4., they are semi-infinite geodesics in $\mathcal{A}$. Item (1) of Lemma \ref{lemma: semi-inf geod properties} gives that they must necessarily coincide almost surely with the jump time of the almost surely unique semi-infinite geodesic $\pi[x\to (y, 1)]$. Note for any countable dense subset of asymptotic directions and endpoints, $(0, \infty)\times \R$, one can construct an identical coupling in essentially the same way as is done in \cite{DOV}. 

Using this coupling, we can also pass to the limit as $n \to \infty$ to obtain the following sharp tail probabilities for semi-infinite geodesics in the Airy line ensemble. This is the content of the following theorem. 

\begin{theorem}\label{thm: semi-inf tail disj proba}
    Fix $\eta > 0, y_1 < y_2\in \R$ and $x> 0$. Then, for all $\eta>0$, there exists $\varepsilon_0(\eta)$ so that for all $0< \varepsilon < \min\{x, \varepsilon_0(\eta)\}$,
    \begin{equation*}
	C \cdot\varepsilon^{\frac{1}{2}+\eta}\le \PP \bigg(\pi[x-\varepsilon \to (y_1, 1)] \;\;\; \mathrm{and} \;\;\; \pi[x+\varepsilon \to (y_2, 1)] \text{ are disjoint} \bigg) \le C'\varepsilon^{\frac{1}{2}+o(1)}
	\end{equation*}
    for constants $C, C', d > 0$.
\end{theorem}
\begin{remark}
    Leveraging the coupling between the Airy line ensemble and the Airy sheet in \ref{coupling: Airy sheet}, using similar arguments one can show there exist some universal constants $c, d> 0$ such that for any $0< x_1< x_2$, $y_1, y_2\in \R$,
    \[
     \PP \bigg(\pi[x_1 \to (y_1, 1)] \;\;\; \mathrm{and} \;\;\; \pi[x_2 \to (y_2, 1)] \text{ are not disjoint} \bigg) \le c\mathrm{e}^{-d(|x_1-x_2|\land |y_1-y_2|)^3}.
    \]
\end{remark}

\begin{proof}
    Fix $\eta > 0, y_1<y_2$, $x> 0$ and $0 < \varepsilon < x$. We work with the coupling on page \pageref{Skorokhod coupling} and restrict ourselves to the event of full probability of semi-infinite geodesic uniqueness (with rational endpoints and asymptotic directions), guaranteed by Lemma \ref{lemma: semi-inf geod properties}. On this event, we have the monotonicity of disjointness events with respect to the geodesic endpoints and asymptotic directions. Thus, by the approximation lemma, \ref{lemma: geodesic approx}, we can without loss of generality take all $x, y_1, y_2, \varepsilon$ to be rational. 
    
    Let 
     \[
    \mathsf{Disj}_{x,y_1, y_2, \varepsilon} := \{\pi[x-\varepsilon \to (y_1, 1)] \;\;\; \mathrm{and} \;\;\; \pi[x+\varepsilon \to (y_2, 1)] \text{ are disjoint}\}.
    \]

\underline{Upper bound}:
    We first show using the coupling in page \pageref{Skorokhod coupling} between Brownian melons and the Airy line ensemble, that if (recall the notation for geodesics in Brownian melon environments from the discussion before Lemma \ref{lemma: close-not-dj})
    \[
    \pi\{\overline{x-\varepsilon}\to \widehat{y_1}\}_n \;\;\; \mathrm{and} \;\;\; \pi\{\overline{x +\varepsilon}\to \widehat{y_2}\}_n  \text{ are not disjoint infinitely often in }n\,,
    \]
    the geodesics $\pi[x-\varepsilon \to (y_1, 1)]$ and $\pi[x+\varepsilon \to (y_2, 1)]$ are not disjoint. This will then allow us to use the estimates from Lemma \ref{lemma: close-not-dj}. A defect of the coupling is that the statements holds for some subsequence of Brownian melons. We can harmlessly identify this with $\N$ and not do any further relabelling.

    By Property 2. of the coupling on page \pageref{Skorokhod coupling}, there exists a random $N\in \N$ such that
\begin{equation}\label{eq: geod parabol monoton 2}
    Z_{k}(x-\varepsilon, y_1) < Z_k(x, y_1)\,,
\end{equation}
for all $k\ge N$. This means the semi-infinite geodesics $\pi[x-\varepsilon \to (y_1, 1)]$ and $\pi[x \to (y_1, 1)]$ on $\mathcal{A}$ branch on or before level $N$, consistent with their asymptotic parabolic shapes. By Property 2. of the coupling \pageref{Skorokhod coupling}, \eqref{eq: geod parabol monoton 2}, we have the almost sure convergence of jump times
\[
Z^n_{\ell}(x-\varepsilon, y_1)\stackrel{n\to \infty}{\longrightarrow}Z_{\ell}(x-\varepsilon, y_1)\,,\quad Z^n_{\ell}(x+\varepsilon, y_2)\stackrel{n\to \infty}{\longrightarrow}Z_{\ell}(x+\varepsilon, y_2)\,,\qquad 1\le \ell \le k\,.
\]
Thus, on the event $\{N = k\}$, \eqref{eq: geod parabol monoton 2} means that for all $n$ sufficiently large, $Z^n_{k}(x-\varepsilon, y_1) < Z^n_{k}(x, y)$. Consequently, by the tree-like nature of geodesics, Lemma \ref{lemma: geod tree}, the geodesics $\pi\{\overline{x-\varepsilon}\to \widehat{y_1}\}_n$ and $\pi\{\overline{x}\to \widehat{y_1}\}_n$ branch before or at depth $k$ (since otherwise this would violate geodesic uniqueness on the melons $WB^n$, which we tacitly assume throughout).

Therefore, on the event 
\[
    \pi\{\overline{x-\varepsilon}\to \widehat{y_1}\}_n \;\;\; \mathrm{and} \;\;\; \pi\{\overline{x +\varepsilon}\to \widehat{y_2}\}_n  \text{ are not disjoint infinitely often in }n\,,
\]
we have by geodesic monotonicity that for all $n$ sufficiently large, the zig-zags of geodesics $\pi\{\overline{x-\varepsilon}\to \widehat{y_1}\}_n$ and $\pi\{\overline{x+\varepsilon}\to \widehat{y_2}\}_n$ cannot meet past level $k$, for this would create a `geodesic bubble', violating the tree-structure of the geodesics  $\pi\{\overline{x-\varepsilon}\to \widehat{y_1}\}_n$ and $\pi\{\overline{x}\to \widehat{y_1}\}_n$. Since this $k$ is uniform in $n\ge 1$, taking $n\to \infty$ and using the convergence of jump times, we obtain by compactness that the zig-zags of $\pi[x-\varepsilon \to (y_1, 1)]$ and $\pi[x+\varepsilon \to (y_2, 1)]$ share a common point. Since $k\ge 1$ was arbitrary, we conclude the event 
\[
\pi[x-\varepsilon \to (y_1, 1)] \;\;\; \mathrm{and} \;\;\; \pi[x+\varepsilon \to (y_2, 1)] \text{ are disjoint} 
\]
is contained in the event
\[
    \pi\{\overline{x-\varepsilon}\to \widehat{y_1}\}_n \;\;\; \mathrm{and} \;\;\; \pi\{\overline{x +\varepsilon}\to \widehat{y_2}\}_n  \text{ are disjoint eventually in }n\,.
\]
Thus, Fatou's lemma and the bounds in Lemma \ref{lemma: close-not-dj} give the desired result.
    
\underline{Lower bound}:
We proceed by way of contradiction. Suppose that there exists some $\eta > 0$ such that
\[
\liminf_{\varepsilon > 0}\frac{\PP \bigg(\pi[x-\varepsilon \to (y_1, 1)] \;\;\; \mathrm{and} \;\;\; \pi[x+\varepsilon \to (y_2, 1)] \text{ are disjoint} \bigg)}{\varepsilon^{\frac{1}{2}+\eta}} < \infty\,.
\]

We claim that by the planarity of geodesics, Lemma \ref{lemma: crossing equiavlence} and the eventual coalescence of semi-infinite geodesics in the Airy line ensemble, items $(1), (2), (3)$ in Lemma \ref{lemma: semi-inf geod properties}, we obtain (up to a null set)
\[
\pi[x-\varepsilon \to (y_1, 1)] \;\;\; \mathrm{and} \;\;\; \pi[x+\varepsilon \to (y_2, 1)] \text{ are disjoint if and only if } \Delta^{(y_1, 1)}_{(y_2, 1)}(x+\varepsilon) > \Delta^{(y_1, 1)}_{(y_2, 1)}(x-\varepsilon)
\]
eventually in $\varepsilon \to 0$ (rational).

Indeed, using the uniqueness and coalescing properties of semi-infinite geodesics with the same \emph{fixed} asymptotic direction $x> 0$, items $(1), (3)$ in Lemma \ref{lemma: semi-inf geod properties} and \cite[Lemma 3.4]{sarkar2021brownian} we have for any rational $0<\varepsilon < x$, $\pi[x-\varepsilon \to (y_1, 1)]$ and $\pi[x+\varepsilon\to (y_2, 1)]$ are pointwise limits of the geodesics $\pi[(x-\varepsilon)_k \to (y_1, 1)]$ and $\pi[(x+\varepsilon)_k \to (y_2, 1)]$ (recall the notation from Section \ref{sec: semi-inf geod}, $x_k := -\sqrt{k}/\sqrt{2x}$, $x> 0$, $k\ge 1$). Moreover, the convergence of jump times implies for a fixed, possible random rational $t\le y_1$, the paths $\pi[(x-\varepsilon)_k \to (y_1, 1)]$ and $\pi[(x+\varepsilon)_k \to (y_2, 1)]$ almost surely eventually stabilise to their limits on $[t, y_1]$ (see the illustration in Figure \ref{fig: geod approx}). Moreover, by the distinct asymptotic shapes of $\pi[x-\varepsilon \to (y_1, 1)]$ and $\pi[x+\varepsilon \to (y_2, 1)]$ (see \eqref{def: semi-inf geo}), there exists some rational $T < 0$ such that on $(-\infty, T]$, the zig-zags of $\pi[x-\varepsilon \to (y_1, 1)]$ and $\pi[x+\varepsilon \to (y_2, 1)]$ are disjoint. We thus obtain for all $k\ge 1$ by a direct application of Lemma \ref{lemma: crossing equiavlence},
 \[
\pi[(x-\varepsilon)_k \to (y_1, 1)] \;\;\; \mathrm{and} \;\;\; \pi[(x+\varepsilon)_k\to (y_2, 1)]\]
are disjoint if and only if
\[
\mathcal{A}[(x+\varepsilon)_k \to (y_2, 1)]-\mathcal{A}[(x+\varepsilon)_k \to (y_1, 1)] > \mathcal{A}[(x-\varepsilon)_k \to (y_2, 1)]-\mathcal{A}[(x-\varepsilon)_k \to (y_1, 1)].
\]
Thus, taking $k\to \infty$, we obtain the promised equivalence. 

By the coupling between the Airy sheet and the Airy line ensemble, see \ref{coupling: Airy sheet}, we have almost surely,
\[
\mathcal{S}(x, y_2)-\mathcal{S}(x, y_1) = \Delta^{(y_1, 1)}_{(y_2, 1)}(x)\,,\quad x\in (0, \infty)\cap \Q.
\] 
The quadrangle inequality satisfied by the Airy sheet, \eqref{eq: quadrangle ineq airy sheet} (entirely analogous to the crossing lemma, Lemma \ref{lemma: crossing} for geodesics in last passage percolation), yields that the disjointness of the geodesics $\pi[x-\varepsilon \to (y_1, 1)]$ and $\pi[x+\varepsilon \to (y_2, 1)]$ is equivalent to the `Airy difference profile' $x\mapsto \mathcal{S}(x, y_2)-\mathcal{S}(x, y_1)$ being non-constant on $[x-\varepsilon, x+ \varepsilon]$. Hence, we have for the $x$ above,
\begin{equation}\label{eq: lv airy diff proba bound}
    \liminf_{\varepsilon > 0}\frac{\PP \bigg(\mathsf{LV}(\mathcal{S}(x, y_2)-\mathcal{S}(x, y_1), [x-\varepsilon, x+\varepsilon])\neq \emptyset \bigg)}{\varepsilon^{\frac{1}{2}+\eta}} < \infty\,,
\end{equation}
where for $a< b$ and a function $f:[a, b] \to \R$, we denote the \emph{set of local variation}, 
\[
\mathsf{LV}(f, [a, b]) = \left\{\begin{aligned}
    &c\in [a, b]: f \text{ is non-constant there on every}\\
    &\text{neighbourhood of }c \text{ in } [a, b]
\end{aligned}\right\}\,.
\]
By translation symmetries \eqref{eq: Airy sheet symmetry} of the Airy sheet $\mathcal{S}$, for any $c \in \R$, the process $x\mapsto \mathcal{S}(x, y_2)-\mathcal{S}(x, y_1)$ has stationary increments. This means the above estimates hold for all $x\in \R$.

Now, for $a < b$ finite, consider the interval $[a, b]$. For any $\varepsilon > 0$ (here we will take the deterministic subsequence from \eqref{eq: lv airy diff proba bound} which we will not relabel), starting from $a$, partition $[a, b]$ into a family of intervals of diameter $2\varepsilon$, possibly overshooting $b$ in the last step. Let $N_\varepsilon\le \lceil(b-a)/(2\varepsilon)\rceil$ be the random number of intervals above for which the (non-decreasing by \eqref{eq: quadrangle ineq airy sheet}) difference profile $x\mapsto \mathcal{S}(x, y_2)-\mathcal{S}(x, y_1)$ experiences an increase. Using linearity and taking expectations, we have the estimate
\[
\mathbb{E}[N_\varepsilon] \le \bigg\lceil\frac{b-a}{2\varepsilon}\bigg\rceil\cdot\PP \bigg(\mathsf{LV}(\mathcal{S}(a, y_2)-\mathcal{S}(a, y_1), [a, a+2\varepsilon])\neq \emptyset \bigg)\le C_{a, b}\cdot\varepsilon^{-\frac{1}{2}+\eta}
\]
for some constant $C_{a, b} > 0$.

We can thus cover the set of local variation $\mathsf{LV}(\mathcal{S}(\cdot, y_2)-\mathcal{S}(\cdot, y_1), [a, b])$ by $N_\varepsilon$ many intervals $(U^i_\varepsilon)_i$ of diameter $2\varepsilon$ such that
\[
\sum_i \diam(U^i_\varepsilon)^{\frac{1}{2}-\frac{\eta}{2}}\le C_{a, b, \eta}  \cdot N_\varepsilon\cdot \varepsilon^{\frac{1}{2}-\frac{\eta}{2}}
\]
for some constant $C_{a, b, \eta} > 0$. Taking expectations, we have
\[
 \mathbb{E}[\sum_i \diam(U^i_\varepsilon)^{\frac{1}{2}-\frac{\eta}{2}}]\le C_{a, b, \eta}\cdot \varepsilon^{-\frac{1}{2}+\eta} \cdot \varepsilon^{\frac{1}{2}-\frac{\eta}{2}} \le C_{a, b, \eta} \cdot \varepsilon^{\frac{\eta}{2}}\,, 
\]
which converges to zero as $\varepsilon \to 0$. Thus, (after possibly passing to a further subsequence of $\varepsilon \searrow 0$) almost surely, the Hausdorff dimension of $\mathsf{LV}(\mathcal{S}(\cdot, y_2)-\mathcal{S}(\cdot, y_1), [a, b])$ is at most $1/2 - \eta /2 < 1/2$. Since $a< b$ were arbitrary, we obtain that the Hausdorff dimension of  $\mathsf{LV}(\mathcal{S}(\cdot, y_2)-\mathcal{S}(\cdot, y_1), \R)$ is at most $1/2 - \eta /2 < 1/2$. This is a contradiction as it is almost surely one half \cite[Theorem 1.1]{Basu2021Fractal}.
\end{proof}

\section{Busemann functions}\label{sec: Busemann}
As we saw in Theorem \ref{thm: semi-inf tail disj proba}, the pathwise regularity of the Airy difference profiles, given by the coupling \ref{coupling: Airy sheet} as `top-level' Busemann functions in the Airy line ensemble (recall the notation from \eqref{eq: Airy limit busemann}), 
\[
\mathcal{S}(x, y_2)-\mathcal{S}(x, y_1) = \Delta^{(y_1, 1)}_{(y_2, 1)}(x)\,,\quad x\in (0, \infty)\cap \Q\,,
\] 
plays a crucial role in controlling the likelihood of disjointness for semi-infinite geodesics in the Airy line ensemble ending at the top line thereof. In this light, we now investigate some structural properties of Busemann functions associated to semi-infinite geodesics with arbitrary endpoints $(t, m)$, $(t', m')$, $\R\times \N$, $x\mapsto \Delta^{(t, m)}_{(t', m')}(x)$ formally defined as
\[
\Delta^{(t, m)}_{(t', m')}(x):= 
    \displaystyle\lim_{k\to\infty}(\mathcal{A}[x_k\to (t',m')]-\mathcal{A}[x_k\to (t,m)])\,,\quad x \in \Q\cap (0, \infty), (t,m), (t', m')\in \Q\times \N\,.
\]
By the coalescence of semi-infinite geodesics in the Airy line ensemble with the same fixed asymptotic direction, item (3) of Lemma \ref{lemma: semi-inf geod properties}, we have that the limits above for all $x\in \Q\cap (0, \infty)$, $t, t; \in \Q, m ,m '$ almost surely eventually stabilise and so the Busemann profiles are well defined. In fact, again by Lemma \ref{lemma: semi-inf geod properties}, these limits exist and are well defined for any fixed countable dense subset of asymptotic directions $x>0$ and endpoints $(t, m), (t',m')$, $t, t'\in \R$.

This investigation will culminate in a proof of the non-degeneracy and H\"{o}lder $1/2^{-}-$continuity of these profiles, Theorems \ref{thm: busemann cont} and \ref{thm: Busemann equality in law} which will have implications on exceptional behaviour of semi-infinite paths; this will be explored in Section \ref{sec: exceptional geod geom}. The nature of this eventual argument will be geometric and will proceed by induction on the depths $m', m$, using the pathwise regularity of the Airy sheet as the base case. We will also give a separate proof using different ideas. In particular, we can also prove continuity of the Busemann profiles by leveraging certain melon identities in the prelimiting Airy line ensembles. This is in spirit similar to the newly found `down-the-middle' coupling between the Airy sheet and the Airy line ensemble in \cite{dauvergne2026downthemiddle}.

We start with a preliminary result which essentially states that the absolute increments of the functions $\Q\cap(0, \infty)\ni x\mapsto \Delta^{(t, m)}_{(t', m')}(x)$, are locally uniform in $t, t', m, m'$ and that in fact for fixed $m, m'\in \N, x\in \Q\cap(0, \infty)$ one obtains the profile $(t, t')\mapsto \Delta^{(t, m)}_{(t', m')}(x)$ has a jointly continuous modification.

First we record the observation that the functions $\Q\cap(0, \infty)\ni x\mapsto \Delta^{(t, m)}_{(t', m')}(x)$, are non-decreasing for $t\le t'$ and $m\le m'$. This follows from the limiting pointwise `relative-length' construction in \eqref{eq: Airy limit busemann} and the crossing lemma, Lemma \ref{lemma: crossing} applied to the finite depth approximating geodesics that appear in Lemma \ref{lemma: semi-inf geod properties}. 

\begin{proposition}\label{prop: Busemann regularity}
  Almost surely, for any fixed $x\in \Q\cap(0, \infty)$, $m, m' \in \N$, the map\\ \noindent $(t, t')\mapsto \Delta^{(t, m)}_{(t', m')}(x)$ has a modification in $\mathcal{C}(\R^2)$. Moreover, the process $\Q\cap(0, \infty)\ni x\mapsto \big(t, m, t', m'\mapsto \Delta^{(t, m)}_{(t', m')}(x)\big)$ has right and left limits, with convergence uniform in $(t, t', m, m')$ on bounded subsets of $\Q^2\times \N$.
\end{proposition}
\begin{proof}
    For fixed $x > 0$ rational, the local uniform continuity of the map $(t, m, t', m')\mapsto \Delta^{(t, m)}_{(t', m')}(x)$ on $(\Q\times \N)^2$ follows from the uniform coalescence of geodesics with fixed asymptotic direction and bounded right endpoints, item $(3)$ of Lemma \ref{lemma: semi-inf geod properties}. In particular, for any bounded subset $K\subset (\Q\times \N)^2$, all geodesics $\pi[x\to (t, m)]$ and $\pi[x\to (t', m')]$, for $(t, m , t', m ')\in K$ pass through a common point and coalesce thereafter. This means the Busemann limit in \eqref{eq: Airy limit busemann} almost surely eventually stabilises uniformly in $k\ge1$ for $(t, m , t', m ')\in K$ in any fixed bounded set, giving $\Delta^{(t, m)}_{(t', m')}(x)$ as a difference of two finite depth last passage values. 
    
    Now to establish the second part, we show that for any monotone sequence $(x_n)\subseteq \Q\cap(0, \infty)$ converging to some $x > 0$, on any bounded subset of $(\Q\times \N)^2$, the uniform norms $\norm{\Delta^{\cdot}_{\cdot}(x_n)-\Delta^{\cdot}_{\cdot}(x_m)}$ converge to zero as $m, n \to \infty$ by the monotonicity of the Busemann profiles and the planarity of geodesics, Lemma \ref{lemma: crossing}. Indeed, two application of Lemma \ref{lemma: crossing} gives for all $s\le t, t'\le r$ rational, $n\le m, m'\le N$ and $0< x< x'$ rational,
    \begin{align*}
        0&\le \sup_{\substack{s\le t, t'\le r\\n\le m ,m'\le N}}|\Delta^{(t, m)}_{(t', m')}(x')- \Delta^{(t, m)}_{(t', m')}(x)| \le \Delta^{(s, n)}_{(r, N)}(x')- \Delta^{(s, n)}_{(r, N)}(x)\to 0\,, \quad x'\to x\,,
    \end{align*}
    where the first inequality follows from Lemma \ref{lemma: crossing} applied to the approximating paths in Lemma \ref{lemma: semi-inf geod properties}, and the convergence to zero comes from the monotonicity of the Busemann processes $\Delta^{(s, n)}_{(r, N)}(\cdot)$.
\end{proof}

Equipped with this regularity for Busemann functions, we investigate the asymptotic non-degeneracy of these difference profiles, with estimates on their likelihood of non-constancy on a neighbourhood of a point. 

We recall the notation for the set of non-constancy, or local variation of some measurable function. Particularly, for $a< b$ and a function $f:[a, b] \to \R$, we denote the \emph{set of local variation}, 
    \[
\mathsf{LV}(f, [a, b]) = \left\{\begin{aligned}
    &c\in [a, b]: f \text{ is non-constant there on every}\\
    &\text{neighbourhood of }c \text{ in } [a, b]
\end{aligned}\right\}\,.
\]

Note for for $\ell\ge 2$ and fixed rational $0< a< b$, the set of local variation of the Busemann profiles (recall from \eqref{eq: busemann cadlag/caglad}),
$\mathsf{LV}\big(\Delta^{\square, (0,1)}_{(0,\ell)}(u)\,, u\in [a, b])$ only depends on the non-constancy of the values of the Busemann profiles $\Delta^{\square, (0,1)}_{(0,\ell)}$ at rational points in $[a,b]$, which equal the Busemann limit $\Delta^{(0,1)}_{(0,\ell)}$ in \eqref{eq: Airy limit busemann}. Thus, in the remainder, when referring to sets of local variation of $\Delta^{\square, (0,1)}_{(0,\ell)}$ in $[a, b]$ (by the same reasoning, we will also do so for $(0, \infty)$), we will simply drop the $\square$.

\begin{proposition}\label{prop: non-disj-close-endpts-inf-geo}
    Fix $\ell \ge 2$. Then, for any fixed rational $0<a< b$,
    \[
    \mathsf{LV}\big(\Delta^{(0,1)}_{(0,\ell)}(u)\,, u\in [a, b]\big)\neq\emptyset \text{ if and only if } \pi[a \to (0, 1)] \;\;\; \mathrm{and} \;\;\; \pi[b \to  (0, \ell)] \text{ are disjoint}\,.
    \]
    
    Moreover, for all rational $0< \varepsilon < x$,
    \begin{equation}\label{eq: LV estimate}
	\PP\bigg(\mathsf{LV}\big(\Delta^{(0,1)}_{(0,\ell)}(u)\,, u\in [x-\varepsilon, x+\varepsilon]\big)\neq\emptyset\bigg)\le C_x \cdot \varepsilon^{\frac{1}{2}+o_\varepsilon(1)}
	\end{equation}
    for some constant $C_x > 0$ that is uniformly bounded when $x$ is on a bounded subset of $(0, \infty)$, bounded away from zero. 
\end{proposition}
\begin{remark}
    Observe that by the stationarity of semi-infinite geodesics with respect to horizontal shifts, item $(3)$ of Lemma \ref{lemma: semi-inf geod properties}, the same results hold verbatim for all Busemann function $\Delta^{(y,1)}_{(y,\ell)}$ with $y\in \R$.
\end{remark}

\begin{proof}
Fix rational $0 < x< z$, and $\ell \ge 2$. We claim that
\[
    \pi[x \to (0, 1)] \;\;\; \mathrm{and} \;\;\; \pi[z \to  (0, \ell)] \text{ are disjoint if and only if } \Delta_{(0, \ell)}^{(0, 1)}(z) >  \Delta_{(0, \ell)}^{(0, 1)}(x)\,.
    \]
    We will proceed as in the proof of Theorem \ref{thm: semi-inf tail disj proba}, where this was established for $\ell = 1$.
    
    More precisely, using the uniqueness and coalescing properties of semi-infinite geodesics with the same \emph{fixed} asymptotic direction $x> 0$, items $(1), (3)$ in Lemma \ref{lemma: semi-inf geod properties} and \cite[Lemma 3.4]{sarkar2021brownian} we have $\pi[x \to (0, 1)]$ and $\pi[z\to (0, \ell)]$ are pointwise limits of the geodesics $\pi[x_k \to (0, 1)]$ and $\pi[z_k\to (0, \ell)]$ (recall the definition from Section \ref{sec: semi-inf geod}). By the distinct asymptotic shapes of $\pi[x \to (0, 1)]$ and $\pi[z\to (0, \ell)]$ (see \eqref{def: semi-inf geo}), there exists some rational $T < 0$ such that on $(-\infty, T]$, the zig-zags of $\pi[x \to (0, 1)]$ and $\pi[z\to (0, \ell)]$ are disjoint. Arguing as in the approximation lemma, Lemma \ref{lemma: geodesic approx}, the convergence of jump times implies for a fixed, possible random rational $t\le 0$, the paths $\pi[x_k \to (0, 1)]$ and $\pi[z_k \to (0, \ell)]$ almost surely eventually stabilise to their limits on $[t, 0]$ (see the illustration in Figure \ref{fig: geod approx}) Thus, the pointwise convergence of jump times and the approximation lemma, \ref{lemma: geodesic approx} implies the disjointness of $\pi[x \to (0, 1)]$ and $\pi[z \to (0, \ell)]$ is equivalent to the almost sure eventual disjointness of $\pi[x_k \to (0, 1)]$ and $\pi[z_k \to (0, \ell)]$ on $[T, 0]$ as $k\to \infty$.
    
    We also obtain by a direct application of Lemma \ref{lemma: crossing equiavlence}, that almost surely for all $k \ge \ell$
    \[
    \pi[x_k \to (0, 1)] \;\;\; \mathrm{and} \;\;\; \pi[z_k\to (0, \ell)]
    \]
    are disjoint (recall from Definition \ref{def: disjointness}) if and only if
     \[
    \mathcal{A}[z_k \to (0, \ell)]-\mathcal{A}[z_k \to (0, 1)] > \mathcal{A}[x_k \to (0, \ell)]-\mathcal{A}[x_k \to (0, 1)]\,.
    \]
    By item (3) in Lemma \ref{lemma: semi-inf geod properties} the above differences of last passage values eventually stabilise in $k\ge \ell$. 
    
    Hence, taking $k\to \infty$, using the definition of the Busemann processes in \eqref{eq: Airy limit busemann}, we conclude the proof of the claim. 
 
Now, fix $0 < \varepsilon < \min\{x, 1\}$, $y> 0$. On the event the jump time $Z_\ell(x+\varepsilon, y)$ of geodesic $\pi[x+\varepsilon, (y, 1)]$ from level $\ell$ to $\ell-1$ is greater than zero, if the geodesics $\pi[x-\varepsilon\to (0, 1)]$ and $\pi[x+\varepsilon \to (y, 1)]$ are not disjoint, then by monotonicity geodesics $\pi[x-\varepsilon\to (0, 1)]$ and $\pi[x+\varepsilon \to (0, \ell)]$ are also not disjoint. Thus, a union bound on the complement of the intersection of the above events gives
\begin{align*}
    &\PP\bigg(\mathsf{LV}\big(\Delta^{(0,1)}_{(0,\ell)}(u)\,, u\in (x-\varepsilon, x+\varepsilon)\big)\neq\emptyset\bigg)\\
    &= \PP\bigg(\pi[x-\varepsilon \to (0, 1)] \;\;\; \mathrm{and} \;\;\; \pi[x+\varepsilon \to  (0, \ell)] \text{ are disjoint} \bigg)\\
    &\le \PP\left( \pi[x-\varepsilon \to (0, 1)] \;\;\; \mathrm{and} \;\;\; \pi[x+\varepsilon  \to  (y, 1)] \text{ are disjoint} \right) + \PP\bigg(Z_{\ell}(x+\varepsilon, y)\le 0\bigg)\\
    &\le  \PP\left( \pi[x-\varepsilon \to (0, 1)] \;\;\; \mathrm{and} \;\;\; \pi[x+\varepsilon  \to  (y, 1)] \text{ are disjoint} \right)\\
    &+ \PP\bigg(Z_{\ell}(x+\varepsilon, 0)+y\le 0\bigg)\,,
\end{align*}
where the last line follows by item$(4)$ in Lemma \ref{lemma: semi-inf geod properties}, which gives the equality in law $Z_\ell(x+\varepsilon, y) \overset d = Z_\ell(x+\varepsilon , 0) + y$. We now estimate the lower tails of the geodesic jump times $Z_\ell (x+\varepsilon, 0)$. 

Fix $z> 0$, by geodesic coalescence, \cite[Lemma 3.4]{sarkar2021brownian}, we have there exists some random $T <0$ rational and depth $d(T) \in \N$ such that all the geodesics $\pi[z_k \to (-y, 1)]$ and $\pi[z_k\to (0, 1)]$ almost surely eventually coalesce as $k\to \infty$.

By the coupling between the Airy sheet and the Airy line ensemble, \ref{coupling: Airy sheet}, we have 
\[
\mathcal{S}(z,0)-\mathcal{S}(z,-y) = \mathcal{A}[(T, d(T))\to (0, 1)]- \mathcal{A}[(T, d(T))\to (-y, 1)]\,.
\]
Since last passage values become additive on the graphs of their geodesics (by the metric composition law, see \ref{lemma: Metric Composition}), on the event $\{Z_\ell(x, 0) \le -y\}$, we have the upper bounds
\begin{equation}\label{eq:comparison}
 \mathcal{S}(z,0)-\mathcal{S}(z,-y)
 \le \mathcal{A}[(-y, \ell)\to (0,1)] \le \mathcal{A}_1(0)-\mathcal{A}_\ell(-y)
 \quad\text{on }\{Z_\ell(x,0)\le-y\}\,,
\end{equation}
where the latter inequality comes from the monotonicity of the Airy lines, and the definition of the length functional, see Definition \ref{def: length}.
Rearranging \eqref{eq:comparison} shows that with $r = zy$,
\begin{equation}\label{eq:cost}
 \{Z_\ell(x,0)\le-y\}
 \subseteq
 \{2r\le\mathcal{A}_1(0)-(\mathcal{A}_\ell(-y)+y^2)-(\mathcal{S}(z,0)+z^2)+ \mathcal{S}(z,-y)+(z+y)^2\}\,.
\end{equation}
The random variables $\mathcal{S}(z,0)+z^2, \mathcal{S}(z,-y)+(z+y)^2$, and $\mathcal{A}_1(0)$ each have the GUE
Tracy--Widom distribution. Indeed, this follows from coupling in \ref{coupling: Airy sheet}, which gives $\mathcal{S}(0,\cdot)=\mathcal{A}_1$.
Also $\mathcal{A}_\ell(-y)+y^2\overset{d}=\mathcal{A}_\ell(0)$ by the stationarity of the
parabolic shifted line ensemble. 

A union bound and the above distributional identities give,
\begin{align*}
 \PP\bigl(Z_\ell(x,0)\le-y\bigr)
 &\le \PP\bigl(\mathcal{A}_1(0)\ge r/2\bigr)
      +\PP\bigl(\mathcal{A}_\ell(0)\le-r/2\bigr)\\
 &\quad+\PP\bigl(\mathcal{S}(x,0)+x^2\le-r/2\bigr)
      +\PP\bigl(\mathcal{S}(z,-y)+(z+y)^2\ge r/2\bigr).
\end{align*}
The Tracy--Widom upper and lower tails in
\cite[Theorem~2.3]{DOV}, and Lemma~\ref{lemma: airy counting 1}, imply
\[
 \PP\bigl(Z_\ell(x,0)\le-y\bigr)
 \le C \mathrm{e}^{-cr^{3/2}}+\PP\bigl(\mathcal{A}_\ell(0)\le-r/2\bigr)\,,
\]
for constants $c, C > 0$. Lemma \ref{lemma: airy counting 1} gives $\PP\bigl(\mathcal{A}_\ell(0)\le-r/2\bigr)\le C\mathrm{e}^{\ell-cr^{3/2}}$ for constants $c, C> 0$. Combining the above, we arrive at 
\[
\PP\bigl(Z_\ell(x,0)\le-y\bigr) \le C\cdot \exp (\ell - ca^{\frac{3}{2}})\,,
\]
for constants $c, C> 0$.

To estimate the first disjointness probability, we simply use Theorem \ref{thm: semi-inf tail disj proba}. Combining the above with $z = x+\varepsilon$ (noting $\varepsilon \in (0, 1)$) gives
\[
 \PP\bigg(\mathsf{LV}\big(\Delta^{(0,1)}_{(0,\ell)}(u)\,, u\in (x-\varepsilon, x+\varepsilon)\big)\neq\emptyset\bigg)\le C\cdot y \cdot \varepsilon^{\frac{1}{2}}\cdot \mathrm{e}^{|\log\varepsilon|^{\frac{5}{6}}} + C\cdot \exp\left(\ell-c y ^\frac{3}{2}\right)\,.
\]
Taking $y = O((\log(1/\varepsilon))^{\frac{2}{3}})$ gives the result.
\end{proof}

In fact, we can obtain a simultaneous description of the set of points where $\Delta^{(0, 1)}_{(0, \ell)}$ increases in all neighbourhood thereof. We will do so in terms of disjointness of semi-infinite geodesics. This is the content of the following proposition.

\begin{proposition}\label{prop: local variation geod disj}
    Almost surely, for all $\ell \ge 1$,
    \begin{align*}
       \bigg\{x>0:  \pi^L[x\to (0, 1)]\text{ and } \pi^R[x\to (0, \ell)] \text{ are disjoint} \bigg\} = \mathsf{LV}((0, \infty), \Delta^{(0, 1)}_{(0,  \ell)})\,.
    \end{align*}
    In particular, this set of geodesic disjointness is closed and nowhere dense.
\end{proposition}

\begin{proof}
We proceed by establishing the two inclusions separately.

\underline{$\subseteq$}: Suppose $x>0$ is such that $\pi^L[x\to (0, 1)]$ an $\pi^R[x\to (0, 1)]$ are disjoint (recall Definition \ref{def: disjointness}). Then by monotonicity of geodesics, item $(1)$ of Lemma \ref{lemma: semi-inf geod properties}, for all rational $q< x< r$, we have the \emph{unique} (by item $(2)$ of Lemma \ref{lemma: semi-inf geod properties}) geodesics $\pi[q\to (0, 1)]$ an $\pi[r\to (0, 1)]$ are disjoint.

Thus, we have arguing in exactly the same way as in the proof of \ref{prop: non-disj-close-endpts-inf-geo},
\[
\Delta^{(0, 1)}_{(0,  \ell)}(r) > \Delta^{(0, 1)}_{(0,  \ell)}(q)\,.
\]
Since $q, r\in \Q$ were arbitrary, we conclude that $x\in \mathsf{LV}(\Delta^{(0, 1)}_{(0, \ell)}, (0, \infty))$, concluding the proof of the first part.

\underline{$\supseteq$}: For this part suppose $x\in \mathsf{LV}(\Delta^{(0, 1)}_{(0, \ell)}, (0, \infty))$. We proceed by contradiction. 

If $\pi^L[x\to (0, 1)]$ an $\pi^R[x\to (0, \ell)]$ are \emph{not} disjoint, then their zigzags (recall \eqref{eq: zig-zag}) must share some common point $(z, m)$ with $m\ge \ell$, $z \le 0$. By approximation, Lemma \ref{lemma: geodesic approx}, for any fixed $t < 0$, there exist $q< x< r$ rational such that the restrictions of the zig-zags of $\pi[q\to (0, 1)]$ and $\pi^L[x\to (0, 1)]$ to $[t, 0]$ coincide and similarly for $\pi[r\to (0, \ell)]$ and $\pi^R[x\to (0, \ell)]$. By monotonicity this is true for all $q< q'< x< r'< r$ replacing $q$ with $q'$ and likewise $r$ with $r'$. Taking $t< z$ we see that $\pi[q\to (0, 1)]$ and $\pi[r\to (0, \ell)]$ share the point $(z, m)$. Now arguing as in the previous part and using Lemma \ref{lemma: crossing equiavlence}, instead of a strict inequality the non-disjointness forces an equality, whence we conclude 
\[
\Delta^{(0, 1)}_{(0,  \ell)}(r) = \Delta^{(0, 1)}_{(0,  \ell)}(q)\,.
\]
By monotonicity, this forces $\Delta^{(0, 1)}_{(0,  \ell)}$ to be constant on the interval $(q, r)$, a contradiction to $x\in \mathsf{LV}(\Delta^{(0, 1)}_{(0, \ell)}, (0, \infty))$ concluding the proof.
\end{proof}

We now prove that for any fixed $0< a < b$, and $y\in \R$, almost surely eventually in $\ell$, the sets of local variation $\mathsf{LV}\big(\Delta^{(y,1)}_{(y,\ell)}(u)\,, u\in [a, b]\big)$ are non-empty. By Proposition \ref{prop: local variation geod disj}, this gives in particular the existence of exceptional directions $x> 0$ for which geodesic non-uniqueness fails.

\begin{lemma}\label{lemma: asymptotic non degeneracy ell}
   Fix $\ell \ge 2$. Then for any fixed $\varepsilon > 0$ sufficiently small, 
   \[
     \PP \bigg(\pi[x-\varepsilon \to (0, 1)] \;\;\; \mathrm{and} \;\;\; \pi[x+\varepsilon \to (0, \ell)] \text{ are disjoint} \bigg) \to 1\,, \text{ as } \ell\to \infty\,.
    \]

    This means by Proposition \ref{prop: non-disj-close-endpts-inf-geo}, for any fixed $0< \varepsilon < x$, almost surely eventually in $\ell$, the sets of local variation $\mathsf{LV}\big(\Delta^{(0,1)}_{(0,\ell)}(u)\,, u\in (x-\varepsilon, x+\varepsilon)\big)$ are non-empty.
\end{lemma}
\begin{remark}
    By the stationarity of semi-infinite geodesics with respect to horizontal shifts, item $(3)$ of Lemma \ref{lemma: semi-inf geod properties}, for any fixed $y\in \R$, the same results hold verbatim for the Busemann profile $\Delta^{(y,1)}_{(y,\ell)}$.
\end{remark}

\begin{proof}
    Fix $0< \varepsilon < x$. Observe there is some random rational $T^{\varepsilon}<0$ such that the geodesics $\pi[x-\varepsilon \to (0, 1)],\pi[x+\varepsilon \to (0, 1)]$ are disjoint on $(-\infty, T^{\varepsilon}]$. This is due to their distinct asymptotic parabolic shapes, \eqref{def: semi-inf geo}. Then, for all $\ell \ge \lim_{t\nearrow T^\varepsilon}\pi[x+\varepsilon \to (0, 1)](t)+1$, by geodesic ordering, the zig-zag of the geodesic $\pi[x+\varepsilon \to (0, \ell)]$ is disjoint from $\pi[x+\varepsilon \to (0, 1)]$. Thus, by bounded convergence, we obtain the convergence statement for the disjointness probabilities. By Proposition \ref{prop: non-disj-close-endpts-inf-geo}, we have for all $\ell$ sufficiently large, the sets of local variation $\mathsf{LV}\big(\Delta^{(0,1)}_{(0,\ell)}(u)\,, u\in (x-\varepsilon, x+\varepsilon)\big)$ are non-empty, concluding the proof of the lemma.
\end{proof}

We now use the construction of semi-infinite geodesics $\pi[x\to (y, 1)]$ as the pointwise limit $\lim_{k\to \infty} \pi[x_k \to (y, 1)]$, see the proof of item $(1)$ of Lemma \ref{lemma: semi-inf geod properties}. This gives the jump times $Z_k(x, y)$ as the maximisers of a random continuous function involving Busemann profiles and finite depth last passage values over the Airy line ensemble. 

\begin{proposition}\label{prop: semi-inf geod jump}
    Fix $x> 0, y\in \R$ and $\ell \ge 2$. Then, with $Z_\ell(x, y)$ the jump time of the semi-infinite geodesic $\pi[x\to (y, 1)]$ from level $\ell$ to $\ell-1$, we have
    \[
    Z_\ell(x, y) \in  \displaystyle\operatorname*{argmax}_{s\in (-\infty, y]}\bigg(\Delta_{(s, \ell)}^{(0, 1)}(x)+\mathcal{A}[(s, \ell-1)\to (y,1)]\bigg)\,,\quad \ell \ge 2.
    \]
\end{proposition}

\begin{proof}
    Fix $-\infty< u < y$, $\ell\ge 2$ and consider the almost surely unique (by the remark after Lemma \ref{lemma: geod location brown env lpp}) maximiser
    \[
    Z^{*}_{\ell, u}(x,y):=\displaystyle\operatorname*{argmax}_{s\in [u, y]}\bigg(\mathcal{A}[x_k\to (s,\ell)]-\mathcal{A}[x_k\to (0,1)]+\mathcal{A}[(s, \ell-1)\to (y,1)]\bigg)\,,
    \]
    where we will use the notation $x_k := -\sqrt{k}/\sqrt{2x}$, $x> 0$, $k\ge 1$. It is precisely the jump time from level $\ell$ to level $\ell-1$ of the geodesic $\pi[x_k \to (y, \ell)]$ on the event that jump time is at least $u$. Now, by Lemma \ref{lemma: semi-inf geod properties}, the semi-infinite geodesic $\pi[x\to (y, \ell)]$ can be written as the limit of the geodesics $\pi[x_k \to (y, \ell)]$ in the topology of jump times. Moreover, the convergence of jump times and the almost sure uniqueness of finite horizon geodesics in the Airy line ensemble give that for any fixed possibly random $t \le y$, the geodesics $\pi[x_k\to (y, \ell)]$ and $\pi[x\to (y, \ell)]$ eventually coincide on $[t, y]$, for all $k$ sufficiently large (see Figure \ref{fig: geod approx}). This means for all $u \le Z_\ell(x, y)$ rational, we have almost surely eventually in $k \ge 1$,
    \[
    Z_\ell(x, y) \in \displaystyle\operatorname*{argmax}_{s\in [u, y]}\bigg(\mathcal{A}[x_k\to (s,\ell)]-\mathcal{A}[x_k\to (0,1)]+\mathcal{A}[(s, \ell-1)\to (y,1)]\bigg)\,,\quad \ell \ge 2.
    \]
    Equivalently, 
    \begin{align*}
         &\mathcal{A}[x_k\to (Z_\ell(x, y),\ell)]-\mathcal{A}[x_k\to (0,1)]+\mathcal{A}[(Z_\ell(x, y), \ell-1)\to (y,1)]\\
         &\ge \mathcal{A}[x_k\to (s,\ell)]-\mathcal{A}[x_k\to (0,1)]+\mathcal{A}[(s, \ell-1)\to (y,1)]\,,\quad s\in [u, y]\,.
    \end{align*}
    Taking $k\to \infty$, using the Busemann limit definition \eqref{eq: Airy limit busemann} and then sending $u\to - \infty$, concludes the proof.
\end{proof}

The above variational characterisation of geodesic jump times (an infinite horizon analogue of the metric composition law, see Lemma \ref{lemma: Metric Composition}) allows us to prove a simultaneous characterisation for semi-infinite geodesic jump times of geodesics with arbitrary asymptotic directions and endpoints.

\begin{proposition}\label{prop: semi-inf geod jump simultaneous}
    Almost surely, for all $x> 0, y\in \R$ and $\ell \ge 2$, $\square \in \{L, R\}$, the semi-infinite geodesic jump times $Z^\square_\ell(x, y)$ of the geodesics $\pi^\square [x\to (y, \ell)]$ from level $\ell$ to $\ell-1$ satisfy
    \[
    Z^\square_\ell(x, y) \in  \displaystyle\operatorname*{argmax}_{s\in (-\infty, y]}\bigg(\Delta^{\square, (0,1)}_{(s, \ell-1)}(x)+\mathcal{A}[(s, \ell-1)\to (y,1)]-\mathcal{A}_1(y)\bigg)\,,\quad \ell \ge 2.
    \]
    Moreover, almost surely for all $x> 0$, $A_x := \{y\in \R: Z^L_2(x, y) = Z^R_2(x, y)\}$ has Hausdorff dimension at least one half and is unbounded.
\end{proposition}

\begin{proof}
    By Proposition \ref{prop: semi-inf geod jump}, we have almost surely for all rational $q > 0, r$ the jump times satisfy
    \[
    Z_\ell(q, r) \in \displaystyle\operatorname*{argmax}_{s\in (-\infty, r]}\bigg(\Delta_{(s, \ell-1)}^{(0, 1)}(q)+\mathcal{A}[(s, \ell-1)\to (r,1)]\bigg)\,.
    \]
    In other words, for all $s\le r$,
    \[
    \Delta_{(Z_\ell(q, r) , \ell-1)}^{(0, 1)}(q)+\mathcal{A}[(Z_\ell(q, r) , \ell-1)\to (r,1)]\ge \Delta_{(s, \ell-1)}^{(0, 1)}(q)+\mathcal{A}[(s, \ell-1)\to (r,1)]\,.
    \]
    Now, we have using the approximation lemma, Lemma \ref{lemma: geodesic approx} that almost surely simultaneously for all $x > 0, y\in \R$, the convergence of jump times $Z^\square_\ell(x, y) = \lim_n Z_\ell(x_n, y_n)$ where $x_n>0, y_n$ are rational and increase/decrease respectively to $x, y$, depending on whether $\square = L, R$. 

    Thus, for fixed $s\le y$, we have using the limit definition \ref{eq: busemann cadlag/caglad} and the regularity of the Busemann processes, Proposition \ref{prop: Busemann regularity} for all $s\le y$,
    \[
    \Delta_{(Z^\square_\ell(x, y) , \ell-1)}^{\square, (0, 1)}(x)+\mathcal{A}[(Z^\square_\ell(x, y) , \ell-1)\to (y,1)]\ge \Delta_{(s, \ell-1)}^{\square, (0, 1)}(x)+\mathcal{A}[(s, \ell-1)\to (y,1)]\,.
    \]
    
    To bound the Hausdorff dimensions of $A_x$ from below simultaneously for all $x> 0$, observe that with the coupling between the Airy line ensemble and the Airy sheet \ref{coupling: Airy sheet}, we have
    \[
     F^\square_\ell(x, y): y\mapsto \displaystyle\operatorname*{max}_{s\in (-\infty, y]}\bigg(\Delta_{(s, \ell-1)}^{\square, (0, 1)}(x)+\mathcal{A}[(s, \ell-1)\to (y,1)]-\mathcal{A}_1(y)\bigg) = \mathcal{S}(x, y)-\mathcal{S}(x, 0)-\mathcal{S}(0, y)\,.
    \]
    In particular, this follows from an application of the approximation lemma, Lemma \ref{lemma: geodesic approx} and the continuity of the Airy sheet $\mathcal{S}$, using the fact that this equality is achieved for all rational $x, y$. We thus have almost surely for all $x> 0 , y\in \R$
    \begin{align*}
        &\Delta_{(Z^L_\ell(x, y), \ell-1)}^{L, (0, 1)}(x)+\mathcal{A}[(Z^L_\ell(x, y), \ell-1)\to (y,1)]-\mathcal{A}_1(y)\\
        = &\Delta_{(Z^R_\ell(x, y), \ell-1)}^{R, (0, 1)}(x)+\mathcal{A}[(Z^R_\ell(x, y), \ell-1)\to (y,1)]-\mathcal{A}_1(y)\\
        = &\mathcal{S}(x, y)-\mathcal{S}(x, 0)-\mathcal{S}(0, y)\,.
    \end{align*}

    Observe that by the quadrangle inequality for the Airy sheet \eqref{eq: quadrangle ineq airy sheet}, $F^\square(x, \cdot)$ is almost surely non-decreasing almost surely for all $x > 0$. Now, by the global shape estimates for the Airy sheet, \eqref{eq: airy shape bnds} we have almost surely, for all $x> 0$, $F^\square_\ell(x, y)\to \infty$ as $y\to \infty$. We now show that almost surely, simultaneously for all $x > 0$, the set of local variation of $F^\square_\ell(x, \cdot)$ has Hausdorff dimension at least $1/2$ and is unbounded. 

    Indeed, we proceed by a deterministic argument on the event of full probability of the H\"{o}lder $1/2^{-}$-continuity of the Airy sheet and the aforementioned divergence of $F^\square_\ell(x, \cdot)$ simultaneously for all $x > 0$. For $0< a< b$, suppose $\mathsf{LV}(\mathcal{S}(x, \cdot) - \mathcal{S}(0, \cdot), [a, b])$ is non-empty and has Hausdorff dimension $\delta < 1/2$. By the sub-additivity of Hausdorff measures, there exists some interval $[a, b]\subset \R$ so that $\mathsf{LV}(\mathcal{S}(x, \cdot) - \mathcal{S}(0, \cdot), [a, b])$ has Hausdorff dimension $\delta$. This means that for all $\varepsilon > 0$, there exists some cover of $\mathsf{LV}(\mathcal{S}(x, \cdot) - \mathcal{S}(0, \cdot), [a, b])$ by $N_\varepsilon$ many intervals $(I^\varepsilon_i)_i$ of diameter $\varepsilon$ such that
    \[
    \sum^{N_\varepsilon}_{i=1} \diam(I^\varepsilon_i)^{\delta}\to 0\,,\qquad \varepsilon \to 0\,.
    \]
    We thus estimate by the monotonicity of $\mathcal{S}(x, \cdot) - \mathcal{S}(0, \cdot)$, which follows from the quadrangle identity in \eqref{eq: quadrangle ineq airy sheet}, and the H\"{o}lder $1/2^-$-continuity of $\mathcal{S}$ (Section \ref{sec: preliminaries}),
    \begin{align*}
        0 &< \mathcal{S}(x, b) - \mathcal{S}(0, b) -  \mathcal{S}(x, a) - \mathcal{S}(0, a)\\
        &\le \sum^{N_\varepsilon}_{i=1}\mathcal{S}(x, \sup I^\varepsilon_i) - \mathcal{S}(0, \sup I^\varepsilon_i) -  \mathcal{S}(x, \inf I^\varepsilon_i) - \mathcal{S}(0, \inf I^\varepsilon_i)\\
        &\le C_\delta\cdot \sum^{N_\varepsilon}_{i=1} \diam(I^\varepsilon_i)^{\delta + \frac{1/2 - \delta}{2}}\le C_\delta\cdot \varepsilon^{\frac{1/2 - \delta}{2}}\cdot \sum^{N_\varepsilon}_{i=1} \diam(I^\varepsilon_i)^{\delta}\,,
    \end{align*}
    for some constant $C_\delta > 0$. By assumption the right hand side converges to zero as $\varepsilon \to 0$, a contradiction to the non-constancy of $\mathcal{S}(x, \cdot) - \mathcal{S}(0, \cdot)$ on $[a, b]$. To establish the lower bound on the Hausdorff dimension of $\mathsf{LV}(\mathcal{S}(x, \cdot) - \mathcal{S}(0, \cdot), (0, \infty))$ and show it is unbounded, repeat this argument on intervals $[a, b]$ of non-constancy of $F^\square_\ell(x, \cdot)$ for arbitrarily large $a, b$, guaranteed to exist by the above divergence of the values of $F^\square_\ell(x, y)$ as $y\to \infty$.
    
    For $\ell = 2$, we obtain using the coupling \ref{coupling: Airy sheet},
    \[
    Z^\square_2(x, y) \in \displaystyle\operatorname*{argmax}_{s\in (-\infty, y]}\bigg(\mathcal{S}(x, s)-\mathcal{S}(0, s)\bigg)\,, \qquad x \ge 0, y \in \R\,.
    \]
    Note due to the asymptotic parabolic shape of the Airy sheet, \eqref{eq: airy shape bnds} the expression above is well defined for all $y\in \R, x> 0$. As before, the set of local variation of 
    \[
    y\mapsto \max_{s\in (-\infty, y]}(\mathcal{S}(x, s)-\mathcal{S}(0, s)) = \mathcal{S}(x, y) -\mathcal{S}(0, y)
    \]
    has Hausdorff dimension at least one half and is unbounded. Now there is an at most countable set $N$ of boundary points of intervals of constancy of the above Airy difference profile. Thus, on an event of full probability, on $\mathsf{LV}(\mathcal{S}(x, \cdot) - \mathcal{S}(0, \cdot), (0, \infty))\setminus N$, which has Hausdorff dimension a least one half, for all $y\in K$, $\mathcal{S}(x, y) -\mathcal{S}(0, y) > \mathcal{S}(x, u) -\mathcal{S}(0, u)$ for all $u < y$, and so $Z^L_2(x, y) = Z^R_2(x, y)$. We thus obtain the inclusion $\mathsf{LV}(\mathcal{S}(x, \cdot) - \mathcal{S}(0, \cdot), (0, \infty))\setminus N \subseteq A_x$, concluding the proof.
\end{proof}

We prove that the relative lengths of semi-infinite geodesics, $\Delta^{(0,1)}_{(0,\ell)}(x)$ are almost surely non-degenerate on $(0, \infty)$. This is the content of the following proposition.

\begin{proposition}\label{prop: non-degen asymptotic}
    For all $\ell\ge 2$,
    \begin{equation*}
	\PP\bigg(\mathsf{LV}\big(\Delta^{(0,1)}_{(0,\ell)}(u)\,, u\in (0, \infty)\big)\neq\emptyset\bigg)=1\,.
	\end{equation*}
\end{proposition}

\begin{proof}
    First observe that almost surely, for all $x>0$ we have $Z^\square_2(x, 0)< 0$. By monotonicity of jump times it suffices to consider $x>0$ rational. To see this, by the asymptotic parabolic shape of $\pi[x\to (0,1)]$ there exists $q< 0$ rational and $k\ge 2$ such that $(q, k)$ lies on a horizontal segment of the zigzag (recall from \eqref{eq: zig-zag}) of $\pi[x\to (0,1)]$. Now, by the restriction property of semi-infinite geodesics, we have the path $\pi[x\to (0, 1)]|_{[q, 0]}$ is the unique (by the remark after Lemma \ref{lemma: geod location brown env lpp}) geodesic in the Airy line ensemble from $(q, k)$ to $(0, 1)$. Since the latter has an almost surely positive first jump time by Lemma 6.1 in \cite{tassopoulos2026kpzfixedpointbrownian}), the claim follows.
    
    \noindent \underline{Claim}:
    Fix $y\in \R$ and $\ell\ge 2$. As $x\to \infty$, we have almost surely that the geodesic jump times $Z^\square_\ell(x, y)\to 0$.

    By monotonicity of jump times, it suffices to consider $x> 0$ rational. By Proposition \ref{prop: semi-inf geod jump}, we have that 
     \[
    Z_\ell(x, y) \in  \displaystyle\operatorname*{argmax}_{s\in (-\infty, y]}\bigg(\Delta_{(s, \ell)}^{(y, 1)}(x)+\mathcal{A}[(s, \ell-1)\to (y,1)]\bigg)\,,\quad \ell \ge 2.
    \]
    Now for any fixed $\ell\ge 2$, $a< y$, we have by the metric composition law, Lemma \ref{lemma: Metric Composition}, 
    \[
    \max_{s\in (-\infty, a]}\bigg(\Delta_{(s, \ell)}^{(y, 1)}(x)+\mathcal{A}[(s, \ell-1)\to (y,1)]\bigg)\le \mathcal{S}(x, a)-\mathcal{S}(x,y)+ \mathcal{A}[(a, \ell)\to (y, 1)]\,,
    \]
    while evaluating at the right boundary, we have (for $x> 1$)
    \[
     \max_{s\in (-\infty, y]}\bigg(\Delta_{(s, \ell)}^{(y, 1)}(x)+\mathcal{A}[(s, \ell-1)\to (y,1)]\bigg)\ge \Delta_{(y, \ell)}^{(y, 1)}(x)\ge \Delta_{(y, \ell)}^{(y, 1)}(1) > -\infty
    \]
    where the last inequality comes from monotonicity of Busemann functions, Lemma \ref{lemma: crossing}. Thus, taking $x\to \infty$, using the stationarity of $\mathcal{S}(x, y) +(x-y)^2$, $x, y\in \R$ and the global bounds of the Airy sheet \eqref{eq: airy shape bnds}, we see as $x\to \infty$, $\mathcal{S}(x, a)-\mathcal{S}(x,y)\to -\infty$ and so eventually in $x > 0$, $Z_\ell(x, y) > a$. Taking $a\nearrow y$ concludes the proof of the claim.
    
    For any $x> 0$, taking any rational $x'> x$, by the asymptotic parabolic shape of semi-infinite geodesics, we have there exists some $k\ge 1$ such that for all $k'\ge k$, the jump times satisfy $Z_{k'}(x, 0) < Z_{k'}(x', 0)$. By the above claim, we can take some rational $x''> x'$ sufficiently large so that $Z_{k}(x'',0)>Z_2(x',0)$. Thus, the geodesics $\pi[x\to 0]$ and $\pi[x''\to (0, \ell)]$ are disjoint (as zigzags). Finally, by an application of Lemma \ref{lemma: crossing} we have that $\Delta^{(0,1)}_{(0,\ell)}(x'')>\Delta^{(0,1)}_{(0,\ell)}(x)$, completing the proof.
\end{proof}

Having established the connection between the non-uniqueness of semi-infinite geodesics and the sets of points of increase of Busemann profiles in Propositions \ref{prop: non-disj-close-endpts-inf-geo} and \ref{prop: local variation geod disj}, we obtain some control on the size of the former in terms of the latter.

\begin{theorem}\label{thm: haus dim ell}
    For all $\ell \ge 2$ and $\square \in \{L, R\}$, we have almost surely,
    \[
    0 \le \mathrm{dim}_\mathscr{H} (\mathsf{LV}(\Delta^{\square, (0,1)}_{(0, \ell)}, (0, \infty))) \le \frac{1}{2}\,.
    \] 
\end{theorem}

\begin{proof}
By Proposition \ref{prop: non-disj-close-endpts-inf-geo}, we have for all rational $0 < \varepsilon < x$ the estimates
    \begin{equation*}
	\PP\bigg(\mathsf{LV}\big(\Delta_{(0, \ell)}^{(0, 1)}\,, [x-\varepsilon, x+\varepsilon]\big)\neq\emptyset\bigg)\le C_x \cdot \varepsilon^{\frac{1}{2}+o_\varepsilon(1)}
	\end{equation*}
    for some constant $C_x > 0$ that is uniformly bounded when $x$ is on a bounded subset of $(0, \infty)$, bounded away from zero. 
    
    Now, for any $a < b$ finite, consider the interval $[a, b]$. For any $\varepsilon > 0$, starting from $a$, partition $[a, b]$ into a family of intervals of diameter $2\varepsilon$, possibly overshooting $b$ in the last step. Let $N_\varepsilon\le \lceil(b-a)/(2\varepsilon)\rceil$ be the random number of intervals above for which $\Delta_{(0, \ell)}^{(0, 1)}$ experiences an increase. Using linearity and taking expectations, we have using Markov's inequality the estimate
    \[
    \mathbb{E}[N_\varepsilon] \le \bigg\lceil\frac{b-a}{2\varepsilon}\bigg\rceil\cdot\PP \bigg(\mathsf{LV}(\Delta_{(0, \ell)}^{(0, 1)}, [x-\varepsilon, x+\varepsilon])\neq \emptyset \bigg)\le C_{a, b}\cdot\varepsilon^{-\frac{1}{2}+o_\varepsilon(1)}
    \]
    for some constant $C_{a, b} > 0$.

    We can thus cover the set of local variation $\mathsf{LV}(\Delta_{(0, \ell)}^{(0, 1)}, [a, b])$ by $N_\varepsilon$ many intervals $(U^i_\varepsilon)_i$ of diameter $2\varepsilon$ such that for any fixed $\delta > 0$,
    \[
    \sum_i \diam(U^i_\varepsilon)^{\frac{1}{2}+\delta}\le C_\delta  \cdot N_\varepsilon\cdot \varepsilon^{\frac{1}{2}+\delta}\,.
    \]
    for some constant $C_\delta > 0$. Taking expectations, we have
    \[
     \mathbb{E}[\sum_i \diam(U^i_\varepsilon)^{\frac{1}{2}+\delta}]\le C_{a, b, \delta}\cdot \varepsilon^{-\frac{1}{2}+o_\varepsilon(1)} \cdot \varepsilon^{\frac{1}{2}+\delta} \le C_{a, b, \delta} \cdot \varepsilon^{\delta+ o_\varepsilon(1)}\,, 
    \]
    which converges to zero as $\varepsilon \to 0$. Thus, almost surely, the Hausdorff dimension of $\Delta_{(0, \ell)}^{(0, 1)}$ is at most $1/2 +\delta$. Since $\delta$, $a< b$ were arbitrary, we obtain that the Hausdorff dimension of  $\mathsf{LV}(\Delta_{(0, \ell)}^{(0, 1)}, (0, \infty))$ is at most $1/2$, concluding the proof of the upper bound.
\end{proof}

We now prove the Busemann process comparing length of semi-infinite paths ending on the first and second lines of the Airy line ensemble, has \emph{no discontinuities.}

\begin{proposition}\label{prop: gap 2 cont}
    We in fact show that almost surely for all $s\in \R$ the Busemann profiles $(0, \infty) \ni x\mapsto \Delta^{\square, (s,1)}_{(s, 2)}(x)$ on $(0, \infty)$ are H\"{o}lder $1/2^-$-continuous on $(0, \infty)$.

    This means almost surely for all $s\in \R$ the Busemann profiles $\Delta^{L, (s,1)}_{(s, 2)}$ and $\Delta^{R, (s,1)}_{(s, 2)}$ almost surely agree on $(0, \infty)$.
\end{proposition}

\begin{proof}
The following argument is deterministic on the event of semi-infinite  geodesic existence and uniqueness (see Lemma \ref{lemma: semi-inf geod properties}). In the remainder, we set $s= 0$. 

      In the argument that follows, we will be considering general $\ell\ge 2$ and only use this fact in the end. Fix any $x_0 > 0$. By shift invariance of geodesics, item $(4)$ of Lemma \ref{lemma: semi-inf geod properties}, there exists some random $Y_0\in \N$ such that for all $x> x_0$, geodesics $\pi^\square[x\to (Y_0, 1)]$ satisfy $\pi^\square[x\to (Y_0, 1)](0)\ge \ell$ and so by monotonicity of geodesics, have for all $x> x_0$, $\pi^\square[x\to (Y_0, 1)]\ge \pi^\square[x\to (0, \ell)]$ on $(-\infty, 0]$. Moreover, by Proposition \ref{prop: semi-inf geod jump simultaneous} we can deterministically take $Y\ge Y_0$ (possibly depending on $x$) such that $Z^L_\ell(x, Y) = Z^R_\ell(x, Y)$, when $\ell = 2$ (we will later see this will always be the case for general $\ell\ge 2$).

     Now, take any $x>x_0$ and sequences of rational numbers $x_n < x <  y_n$ with $\inf_n x_n > x_0$ and $x_n - y_n \to 0$ as $n\to \infty$. The proof of Proposition \ref{prop: semi-inf geod jump} gives that almost surely for all $x\in  (0,\infty)\cap \Q , y\in \Q$ we have for the geodesic jump times $Z_k(x, y)$, $k\ge \ell$ (recall notation from Section \ref{sec: semi-inf geod})
     \begin{align*}
         &\lim_{m\to \infty}(\mathcal{A}[x_m\to (Z_k(x, y), k-1)]-\mathcal{A}[x_m\to (0, \ell-1)])+\mathcal{A}[x\to (Z_k(x, y), k-1)]\to (y, \ell-1)]\\
         = &\lim_{m\to \infty}(\mathcal{A}[x_m\to (y, \ell-1)]-\mathcal{A}[x_m\to (0, \ell-1)])\,.
     \end{align*}
     In terms of Busemann processes this reads for $y>0$ rational, 
     \[
     \Delta^{(0, 1)}_{(Z_k(x, y), k-1)}(x)+\mathcal{A}[x\to (Z_k(x, y), k-1)] = \Delta^{(0, 1)}_{(y, \ell-1)}(x)\,.
     \]
     Thus, by the approximation Lemma \ref{lemma: geodesic approx}, we have for any $u> 0 , v\in \R$ (noting the convergence of jump times and regularity of Busemann processes, Proposition \ref{prop: Busemann regularity}), 
      \[
     \Delta^{\square, (0, 1)}_{(Z^\square_k(u, v), k-1)}(x)+\mathcal{A}[(Z^\square_k(u, v), k-1)] = \Delta^{\square, (0, 1)}_{(v, \ell-1)}(x)\,.
     \]
     Thus, with $u = x_n, y_n, v = Y$, using that $Z^L_\ell(x, Y) = Z^R_\ell(x, Y)$ and for $\ell = 2$ $\Delta^{L, (0, 1)}_{(Y, \ell-1)} = \Delta^{R, (0, 1)}_{(Y, \ell-1)}$ $= \mathcal{S}(\cdot, Y)-\mathcal{S}(\cdot, 0)$, we have with $k = \ell+1$ (by Lemma \ref{lemma: geodesic approx} for all $n$ sufficiently large the jump times $Z^\square_\ell(x_n, Y) = Z^\square_\ell(y_n, Y) = Z_\ell(x, Y)$)    
    \[
     \Delta^{\square, (0, 1)}_{(Z_\ell(x, Y), \ell)}(y_n)-\Delta^{\square, (0, 1)}_{(Z_\ell(x, Y), \ell)}(x_n) = \mathcal{S}(y_n, Y)-\mathcal{S}(y_n, 0)-(\mathcal{S}(x_n, Y)-\mathcal{S}(x_n, 0))\,.
     \]
     For the displayed equality we used again that $\ell = 2$.
     
     Then, we estimate using Lemma \ref{lemma: crossing} and that $Z_\ell(x, Y)$ can be taken to be strictly greater than zero
     \[
     0\le \Delta^{(0, 1)}_{(0, \ell)}(y_n)-\Delta^{(0, 1)}_{(0, \ell)}(x_n)\le \Delta^{(0, 1)}_{(q, \ell)}(y_n)-\Delta^{(0, 1)}_{(q, \ell)}(x_n)\to 0 
     \]
     for all $0< q< Z_\ell(x, y)$ and so by continuity in the second variable (Proposition \ref{prop: Busemann regularity}), we have
     \[
     0\le \Delta^{(0, 1)}_{(0, \ell)}(y_n)-\Delta^{(0, 1)}_{(0, \ell)}(x_n)\le \Delta^{(0, 1)}_{(Z_\ell(x, Y), \ell)}(y_n)-\Delta^{(0, 1)}_{(Z_\ell(x, Y), \ell)}(x_n) = \Delta^{(0, 1)}_{(Y, \ell-1)}(y_n)-\Delta^{(0, 1)}_{(Y, \ell-1)}(x_n)\to 0
     \]
     as $n\to \infty$.
     
    Since $x_0> 0$ was arbitrary, we conclude the proof for $s = 0$. To obtain the result almost surely simultaneously for all other $s\in \R$, observe that by Lemma \ref{lemma: crossing} for all $0<x< y$ rational (two applications of the crossing lemma, Lemma \ref{lemma: crossing}, see also Proposition \ref{prop: Busemann regularity})
\[
\Delta^{\square, (s,1)}_{(s, 2)}(y)-\Delta^{\square, (s,1)}_{(s, 2)}(x) \le  \Delta^{\square, (\lfloor s\rfloor,1)}_{(\lceil s\rceil, 2)}(y)-\Delta^{\square, (\lfloor s\rfloor,1)}_{(\lceil s\rceil, 2)}(x)\,.
\]
Then the remaining argument with $\lfloor s\rfloor$, $\lceil s \rceil$ integers is essentially identical to that with setting $s = 0$ (see Figure \ref{fig: cont busemann}).
\end{proof}

\begin{figure}[ht]
    \centering
    \includegraphics[width=0.65\linewidth]{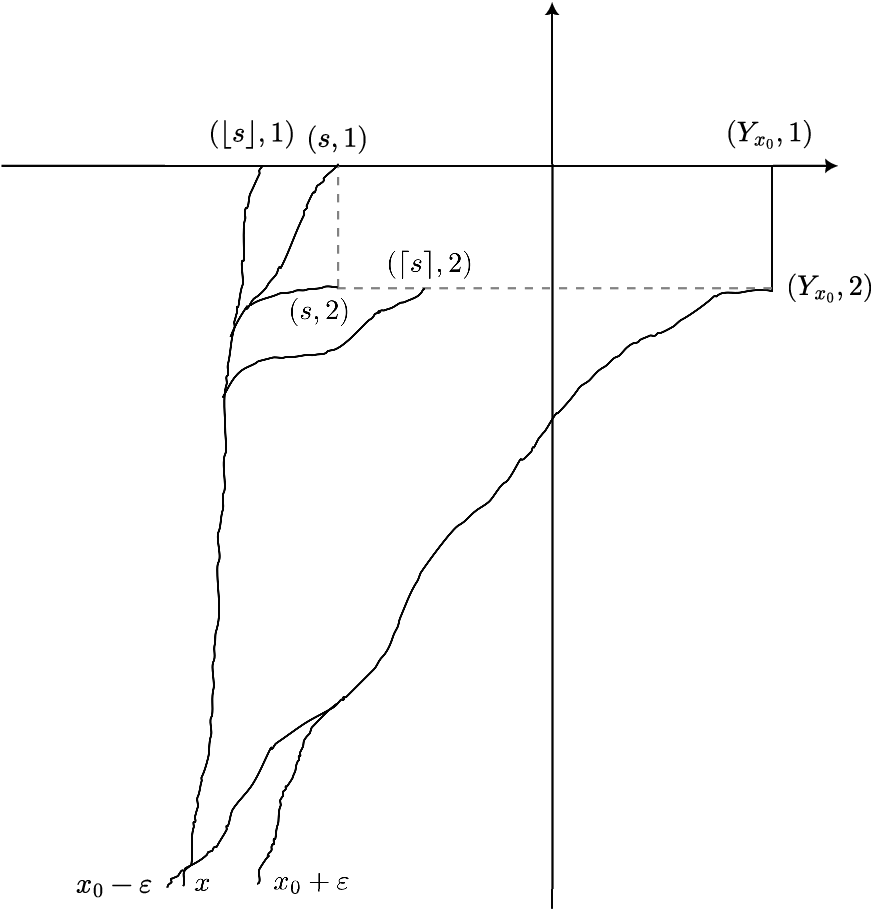}
    \caption{Geometric (cartoon) illustration of the argument in the proof of Proposition \ref{prop: gap 2 cont} for general $s\in \R$. The following construction if deterministic on the event of existence and uniqueness of geodesics (both finite and semi-infinite) on the Airy line ensemble on any fixed countable dense set of endpoints). For any fixed $x_0> 0$, by Proposition \ref{prop: semi-inf geod jump simultaneous} and Lemma \ref{lemma: geodesic approx}, there exists some $Y_{x_{0}}> \lceil s\rceil$ and some $\varepsilon > 0$ such that for all $x\in (x_0-\varepsilon, x_-+\varepsilon)$, all semi-infinite geodesics with asymptotic direction $x$ ending at $(Y_{x_0}, 1)$ necessarily reach $(Y_{x_0}, 2)$ first. Now, by the monotonicity of relative lengths of geodesics, Lemma \ref{lemma: crossing}, we have that the increments of Busemann functions  $\Delta^{\square, (\lfloor s\rfloor,1)}_{(\lceil s\rceil, 2)}$ in $(x_0-\varepsilon, x_0+\varepsilon)$ are at most the relative lengths of geodesics ending at  $(Y_{x_0}, 2)$ and $(\lfloor s \rfloor, 1)$, which can be expressed by the above observations in terms of differences of the Airy sheet.}
    \label{fig: cont busemann}
\end{figure}

As a result, we the following continuity of the Busemann processes `cascades' down levels of the Airy line ensemble, which allows us to conclude the proof of the following theorem.

\begin{theorem}\label{thm: busemann cont}
     Almost surely, for all $\ell \ge 2, s\in \R$, the Busemann processes $\Delta^{(s,1)}_{(s, \ell)}(x)$ on $(0, \infty)$ are H\"{o}lder $1/2^-$-continuous on $(0, \infty)$.
\end{theorem}
\begin{proof}
    The proof proceeds by induction. The base step $\ell=2$ was proved in Proposition \ref{prop: gap 2 cont}. To prove the inductive step for some $\ell$, repeat the argument in Proposition \ref{prop: semi-inf geod jump simultaneous} using the fact that by the induction hypothesis, $(x, s)\mapsto \Delta^{\square, (0, 1)}_{(s, \ell-1)}(x)$ are jointly continuous and so by planarity, Lemma \ref{lemma: crossing} (see also Proposition \ref{prop: Busemann regularity}) $\Delta^{\square, (0, \ell)}_{(s, \ell-1)}$ are identical on $(0, \infty)$ for all $s\in \R$. This follows since we can write $\Delta^{\square, (0, 1)}_{(s, \ell-1)} = \Delta^{\square, (s,1)}_{(s, \ell-1)}+ \Delta^{\square, (0, 1)}_{(s, 1)}$, the former by induction and the latter continuous using the coupling in \ref{coupling: Airy sheet}.
    Then argue as in Propositions \ref{prop: semi-inf geod jump} and \ref{prop: semi-inf geod jump simultaneous} to show almost surely for all $x > 0, y\in \R$, the jump times $Z^\square_\ell(x, y)$ of $\pi^\square[x\to (y, \ell-1)]$ from level $\ell$ to level $\ell-1$ satisfy
    \[
    Z^\square_\ell(x, y) \in  \displaystyle\operatorname*{argmax}_{s\in (-\infty, y]}\bigg(\Delta_{(s, \ell-1)}^{(0, \ell)}(x)-\mathcal{A}_{\ell-1}(s)\bigg)\,.
    \]
    The shift invariance of geodesics, item $(4)$ of Lemma \ref{lemma: semi-inf geod properties} gives that almost surely for all $x> 0 , y\in \R$ $Z^\square_\ell(x, y) \to \infty$ as $y\to \infty$. This means as $y\to \infty$, 
    \[
     \displaystyle\operatorname*{infargmax}_{s\in (-\infty, y]}\bigg(\Delta_{(s, \ell-1)}^{(0, \ell)}(x)-\mathcal{A}_{\ell-1}(s)\bigg)\to \infty
    \]
    as $y\to \infty$. This means
    \[
     y\mapsto \displaystyle\operatorname*{max}_{s\in (-\infty, y]}\bigg(\Delta_{(s, \ell-1)}^{(0, \ell)}(x)-\mathcal{A}_{\ell-1}(s)\bigg) = \Delta_{(y, \ell-1)}^{(0, \ell)}(x)-\mathcal{A}_{\ell-1}(y)
    \]
    cannot be eventually constant (the last equality follows from the metric composition law, Lemma \ref{lemma: Metric Composition}). Since then
    \[
     \sup_{y> 0}\displaystyle\operatorname*{infargmax}_{s\in (-\infty, y]}\bigg(\Delta_{(s, \ell-1)}^{(0, \ell)}(x)-\mathcal{A}_{\ell-1}(s)\bigg)< \infty\,.
    \]
    a contradiction.

    Now, arguing as in the last part of the proof of Proposition \ref{prop: semi-inf geod jump simultaneous}, we obtain that almost surely, for all $x> 0$ the sets $A_x = \{Z^L_\ell(x, y) = Z^R_\ell(x,y)\}$ has Hausdorff dimension at least $1/2$. We can thus finally proceed as in the proof of Proposition \ref{prop: gap 2 cont} (essentially as in Figure \ref{fig: cont busemann} replacing the depths $1, 2$ with $\ell-1, \ell$ respectively, and the Airy sheet with the Busemann functions $\Delta^{\square, (0, \ell)}_{ (s, \ell-1)}$, assumed continuous by the induction hypothesis), which concludes the proof of the inductive step.
\end{proof}

We end this section with a description of the law of the Busemann processes $\Delta^{(0, 1)}_{(0, \ell)}$ in terms of processes depending on \emph{finitely many} levels of the Airy line ensemble. This will depend on certain identities satisfied by the pathwise melon constructions of $WB^n$.

In particular, we have the following equality in law on paths.
\begin{theorem}\label{thm: Busemann equality in law}
    Fix $\ell\ge 1$. Then there exists a coupling of an Airy line ensemble $\mathcal{A}$ and a standard Airy sheet $\mathcal{S}$ where the Busemann functions $ \Delta^{(0, 1)}_{(0, \ell)}(x)$ can be characterised as follows on paths $\mathcal{C}(0, \infty)$,
    \begin{align*}
        \Delta^{(0, 1)}_{(0, \ell)}(x) = U_\ell(x) - \mathcal{S}(x, 0)+ \mathcal{A}_\ell(0) \,, \quad x > 0\,,
    \end{align*}
    where $U_\ell$ is a continuous stochastic process equal in law on $\mathcal{C}(0, \infty)$ to $\mathcal{A}[(0, \ell) \to (x, 1)]$.
    Moreover, for any fixed $z\in \R$ and $0< u< v$, we have the following characterisation of the increments of the Busemann functions $\Delta^{(z, 1)}_{(z, \ell)}$,
    \[
    \Delta^{(z, 1)}_{(z, \ell)}(v)-\Delta^{(z, 1)}_{(z, \ell)}(u) = U^{u, v}_\ell -\mathcal{S}(v, z)+\mathcal{S}(u, z)+2z(v-u)\,,
    \]
    with 
    \[
    U^{u, v}_\ell\overset d = \mathcal{A}[(0, \ell) \to (v, 1)] - \mathcal{A}[(0, \ell) \to (u, 1)]\,.
    \]
\end{theorem}

\begin{proof}
We work with the coupling of \pageref{Skorokhod coupling} prelimiting Brownian melons $WB^n$, Airy sheets $\mathcal{S}^n$ and copies of the Airy line ensemble $\mathcal{A}$ and the Airy sheet $\mathcal{S}$. For completeness, we restate the items in the coupling below:
\begin{itemize}
\item[1.] The rescaled melons $A^n$ converge to the Airy line ensemble $\mathcal{A}$ uniformly on compact subsets of $\N\times \R$.
\item [2.] $\mathcal{S}^n$ converges locally uniformly to the Airy sheet $\mathcal{S}$.
\item[3.]We have for all $(x,y)\in (\Q\cap (0, \infty))\times \Q$,
\[
Z_k^n(x,y)\to Z_k(x,y) \qquad \mbox{ for all } k\ge 2\,.
\]
Moreover, as $k\to \infty$,
\begin{equation*}
Z_k(x,y)/\sqrt{k} \to -1/\sqrt{2x}.
\end{equation*}
\item[4.] For every rational triple $(x, y , z)\in (\Q\cap (0, \infty)) \times \Q^2$ with $y < z$, there exist random rational points $0< X_1 < x < X_2$ so that the melon geodesics $\pi\{\overline{X_1}, \hat{y}\}_n$ and  $\pi\{\overline{X_2}, \hat{z}\}_n$ are not disjoint for all $n$ large enough.
\end{itemize}

Fix $n, m\ge 0$, $\ell \le n$ and a sequence $x_1, \ldots , x_m \in (0, \infty)$.
Consider the random variables
\[
\begin{aligned}
D_{n,\ell}(x_i)
&:= n^{1/6}\Bigl(
 WB^n[(2x_i n^{-1/3},n)\to(1,\ell)]
 -WB^n[(2x_i n^{-1/3},n)\to(1,1)]
 \Bigr)\,,\\
L_{n, k} &:= n^{1/6}(WB^n_k(1)-2\sqrt n) = A^n_k(0)\,,\quad 1\le k \le \ell\,,\\
\varepsilon_n(x_i) &:= \mathcal{S}^n(x, 0) - \mathcal{S}^n(0, 0) \,,\\
U_\ell(x_i)&:= \mathcal A[(0,\ell)\to(x_i,1)]\,.
\end{aligned}
\]
We aim to establish the following convergences in law,
\[
\begin{tikzcd}[
  column sep=11em,
  row sep=8em,
  every arrow/.append style={shorten <=4pt, shorten >=4pt}
]
\displaystyle
  \bigl(D_{n,\ell}(x_i)-L_{n, \ell}+\varepsilon_n(x_i)\bigr)_{i=1}^{m}
  \arrow[r, " d",
    "{\substack{\text{\cite[Lemma~6.2]{DOV}}\\ \text{ and Theorem \ref{thm: melon scaling Airy}}}}"']
  \arrow[d, " d"',
    "\text{Busemann limit}"]
&
\displaystyle
  \bigl(U_\ell(x_i)\bigr)_{i=1}^{m}
\\
\displaystyle
  \bigl(\Delta^{(0, 1)}_{(0, \ell)}(x_i)+\mathcal{S}(x_i, 0)-\mathcal{A}_\ell(0)\bigr)_{i=1}^{m}
  \arrow[ur, phantom,
    "\overset{ d}{=}" description, sloped]
&
\end{tikzcd}
\]
By the pathwise construction of melons, see (17) in \cite{DOV} and their preservation of last passage values, Proposition 4.1 in \cite{DOV}, we have $A^n_1(1) = \mathcal{S}^n(0, 0)$. By \cite[Lemma 6.2]{DOV}, one can simultaneously express for all $1\le i \le m$,
\[
D_{n,\ell}(x_i)-L_{n, \ell}+\varepsilon_n(x_i) =  2x_in^{1/3}-n^{1/6}W^*_1B^n[(1-2x_i n^{-1/3},1)\to_\mathfrak{f}(1,\ell)]\,,
\]
and $W^*_1B^n$ denotes the \emph{reversed melon opened up at 1} (see equation ($26$) in \cite{DOV} and $[\cdot \to_{\mathfrak{f}}\cdot]$ denotes a \emph{backwards first} passage value (see equation ($27$) in \cite{DOV}). Importantly, by the time reversal symmetry of Brownian motion, $W^*_1B^n$ has the same law as the Brownian melon $WB^n$. The joint convergence of the rescaled backwards melons (up to the edge scaling terms) to $U_\ell(x_i)$ as $n\to \infty$ is thus established by an application of Theorem \ref{thm: melon scaling Airy} to the melons $W^*_1B^n$, noting that the backwards first passage values (which are defined analogously to the usual last passage values) are continuous functionals of their environment with respect to the topology of local uniform convergence, see (35), (36) in \cite{DOV}.

Combining the above with Theorem \ref{thm: melon scaling Airy}, we have the convergence in distribution
\[
D_{n,\ell}(x_i)-L_{n, \ell}+\varepsilon_n(x_i) \overset{d}{\longrightarrow}\mathcal{A}[(0, \ell) \to (x, 1)]\,,\qquad n \to \infty\,.
\]

We now identify limits for individual terms on the left hand side of the above expression. Moreover, by the above coupling, $\varepsilon_n(x)$ converges to $\mathcal{A}(0)-\mathcal{S}(x, 0)$ in the topology of local uniform convergence on $\mathcal{C}$. By item 1. of the same coupling, the terms $L_{n, \ell}$ converge almost surely uniformly to $\mathcal{A}_\ell(0)$.

We now show that jointly in $1\le i \le m$, $D_{n, \ell}$ converges to $\Delta^{(0, 1)}_{(0, \ell)}(x_i)$ as $n\to \infty$. 

To achieve this, we use a strengthened version of the coupling between the melons $WB^n$ and the Airy line ensemble provided by Theorem \ref{thm: melon scaling Airy}. In particular, by another application of the Skorokhod representation theorem, we obtain the coupling on page \pageref{Skorokhod coupling}. In particular, from item $2.$ therein, we have the convergence of rescaled geodesic jump times on the melon to those of semi-infinite geodesics on the Airy line ensemble. 

Now, fix $1\le i\le m$. By the stationarity of geodesics under horizontal shifts, item $(4)$ in Lemma \ref{lemma: semi-inf geod properties}, there exists some random (possibly $i$-dependent) $y\in (\Q\cap(0, \infty))$ such that $Z_\ell(x, y) > 0$. This means almost surely eventually in $n$, the geodesics in $WB^n$ from $(2x_in^{-\frac{1}{3}}, n) \to (1+2yn^{-\frac{1}{3}}, 1)$ reach level $\ell$ before time $1$. 

By item 4. of the above coupling, we have the geodesics in $WB^n$ from $(2x_in^{-\frac{1}{3}}, n) \to (1, 1)$ and $(2x_in^{-\frac{1}{3}}, n) \to (1+2yn^{-\frac{1}{3}}, 1)$ meet at some point $(2Tn^{-\frac{1}{3}}, 1+2d(T)n^{-\frac{1}{3}})$ almost surely eventually in $n$. By the choice of $y\in \Q_{> 0}$ and the monotonicity properties of geodesics, we have the geodesics in $WB^n$ from $(2x_in^{-\frac{1}{3}}, n) \to (1, 1)$ and $(2x_in^{-\frac{1}{3}}, n) \to (1, \ell)$ meet at the same point $(2Tn^{-\frac{1}{3}}, 1+2d(T)n^{-\frac{1}{3}})$ almost surely eventually in $n$. Taking $n\to \infty$ and using the local uniform convergence of the melons the rescaled $A^n$ to the Airy line ensemble completes the proof, noting last passage values are continuous with respect to the environment in this local uniform sense.

Now, by the continuity of $\Delta^{(0, 1)}_{(0, \ell)}$, Theorem \ref{thm: busemann cont}, we obtain the desired joint convergence in law on paths on $\R_>0$ as the Borel sigma algebra of $\mathcal{C}((0, \infty))$ is generated by the evaluation functionals $\pi_x : \mathcal{C}((0, \infty))\to \R: f\mapsto f(x)$, $x> 0$. 

To prove the second part, one proceeds in an entirely analogous way, adjusting the definitions of $D_{n, \ell}$, $L_{n, \ell}$, $U_\ell(x_i)$ (using the edge scaling of terms in Theorems \ref{thm: melon scaling Airy} and Theorem 8.3 in \cite{DOV}) to accommodate the new endpoint $z$.
\end{proof}

\section{Quantitative estimates for semi-infinite geodesics: concentration and coalescence}\label{sec: quantitative geodesics}

In this section, we use tools from Section \ref{sec: preliminaries} to obtain concentration for semi-infinite geodesic jump times, Theorem \ref{thm: exponential intercept tails}, as well as quantitative control in the coalescence depths of semi-infinite geodesics, Theorem \ref{thm: coal depth}.

\subsection{An upper-tail estimate for geodesic intercepts}\label{sec: intercept concentration}

As applications of the framework developed in Sections \ref{sec: geod geom}, we obtain exponential tails for semi-infinite geodesic intercepts. This is the content of the following theorem. 
\begin{theorem}[Exponential tails for geodesic intercepts]
\label{thm: exponential intercept tails}
There exist universal constants $c,C>0$ such that, for every
$x,r>0$, $y\in\mathbb R$, and $k\ge2$,
\[
    \mathbb P\bigl(Z_k(x,y)>y-r\bigr)
    \le
    \min\left\{
        1,\,
        C\exp\bigl(Cx^3+Cr^3-ck\bigr)
    \right\}.
\]
In particular, for $x,y>0$,
\[
    \mathbb P\bigl(Z_k(x,y)>0\bigr)
    \le C\exp\bigl(Cx^3+Cy^3-ck\bigr).
\]
\end{theorem}

\begin{proof}
Fix $k\ge 2$ and $x > 0, y \in \R$. For the stationary Airy line ensemble, use $t\mapsto \mathcal R(t):=\mathcal A(t)+t^2$. As usual, constants denoted by $c,C$ below are positive and universal and may change from line to line.

To prove the desired concentration bounds for the semi-infinite geodesic jump times $Z_k(x, y)$, we will use the structure present in the prelimiting Brownian melon environments converging to the Airy line ensemble from Theorem \ref{thm: melon scaling Airy}.

First, we recall some notation from the discussion preceding the coupling on page \pageref{Skorokhod coupling}. 

Fix compact $I\subset (0, \infty), k \ge 3$ and $y\in \R$. By the stationarity of geodesics in the Airy line ensemble, Lemma \ref{lemma: semi-inf geod properties}, we have $Z_k(u, y) \overset d= Z_k(u, 0) + y$. We can thus without loss of generality set $y=0$. We will henceforth assume that there exist unique semi-infinite geodesics with rational endpoints (or on any fixed set of endpoints we use in the remainder of the argument), guaranteed to occur with probability one by Lemma \ref{lemma: semi-inf geod properties}. We will also assume the Airy line ensemble is coupled to a version of the Airy sheet $\mathcal{S}$, see \ref{coupling: Airy sheet}. 

Fix $q\in I$ rational. Then, by Proposition \ref{prop: semi-inf geod jump} and Lemma \ref{lemma: semi-inf geod properties}, with $Z_k(q, 0)$ the jump time of the semi-infinite geodesic $\pi[q\to (0, 1)]$ from level $k$ to $k-1$ satisfies
\[
Z_k(q, 0) = \displaystyle\operatorname*{argmax}_{s\in (-\infty, 0]}\bigg(\Delta_{(s, k)}^{(0, 1)}(q)+\mathcal{A}[(s, k)\to (0,1)]\bigg)\,,
\]
Now, for fixed $r>0$, by the metric composition law and the coupling between $\mathcal{A}, \mathcal{S}$, this gives the almost sure equality,
\[
0 = \Delta_{(Z_k(q, 0), k)}^{(0, 1)}(q)-\mathcal{A}_k(Z_k(q,0))+\mathcal{A}_k(Z_k(q,0)) + \mathcal{A}[(Z_k(q, 0), k) \to (0, 1)]\,.
\]
Since the Busemann functions above are evaluated at rational endpoints, we can express the above as the almost sure limit
\begin{align*}
    \Delta_{(Z_k(q, 0), k)}^{(0, 1)}(q) &= \lim_{m\to \infty} \mathcal{A}[(-\sqrt{m}/\sqrt{2q}, m)\to (Z_k(q, 0), k)]+\mathcal{A}_k(Z_k(q,0))\\
    &-\mathcal{A}[(-\sqrt{m}/\sqrt{2q}, m)\to (0, 1)]\,.
\end{align*}
On the event $\{Z_k(q, 0) > -r\}$, the crossing lemma, Lemma \ref{lemma: crossing} applied to the endpoints $(-r, k), (Z_k(q, 0), k)$ and $(0, 1), (0, k)$, we have
\[
\mathcal{A}_k(Z_k(q,0)) + \mathcal{A}[(Z_k(q, 0), k) \to (0, 1)]\le \mathcal{A}_k(-r) + \mathcal{A}[(-r, k) \to (0, 1)]\,.
\]
Similarly on $\{Z_k(q, 0) > -r\}$, Lemma \ref{lemma: crossing} applied to the points $(-\sqrt{m}/\sqrt{2q}, k), (-\sqrt{m}/\sqrt{2q}, m)$ and $(Z_k(q, 0) , k), (0,k)$ and taking $m\to \infty$ gives
\[
\Delta_{(Z_k(q, 0), k)}^{(0, 1)}(q)-\mathcal{A}_k(Z_k(q,0))\le \Delta_{(0, k)}^{(0, 1)}(q)-\mathcal{A}_k(0)\,.
\]
Combining the above gives on $\{Z_k(q, 0) > -r\}$,
\[
0\le \Delta_{(0, k)}^{(0, 1)}(q)-\mathcal{A}_k(0) + \mathcal{A}_k(-r) + \mathcal{A}[(-r, k) \to (0, 1)]\,.
\]
By Theorem \ref{thm: Busemann equality in law}, we have
\[
 \Delta^{(0, 1)}_{(0, k)}(q) = F^q_k -\mathcal{S}(q, 0)-\mathcal{A}_k(0)\,,
 \]
 where
 \[
 F^q_k \overset d = \mathcal{A}[(0, k) \to (q, 1)]\,,\qquad q \in (\Q\cap(0, \infty))\,.
 \]
We thus have the inclusion (see figure \ref{fig: geod length bound} for an illustration)
\[
\{Z_k(q, 0)> -r\} \subseteq \{\mathcal{A}_1(0)\le U^q_k + \mathcal{R}_k(-r) + \mathcal{R}[(-r, k) \to (0, 1)]\}\,,
\]
where
\[
\begin{aligned}
    \mathcal{R}(\cdot) := \mathcal{A}(\cdot) + (\cdot)^2\,, \quad U^q_k &\overset d = \mathcal{A}[(0, k) \to (x, 1)]+x^2\,.
\end{aligned}
\]
Now, for fixed $x> 0$, taking some rational sequence $q> 0$ approximating $x$, using the continuity of the Airy sheet and the almost sure continuity of the Busemann functions $\Delta^{(0, 1)}_{(0, k)}$ at $x$ (in fact, it is almost surely locally constant in a neighbourhood of $x$), we have almost surely the inclusion
\begin{equation}
\label{eq: intercept-inclusion}
    \{Z_k(x, 0)>-r\}
    \subseteq
    \{\chi\le
        U^x_k+V_k+\mathcal R_k(-r)\}\,,
\end{equation}
where $\mathcal R_k(-r)\overset d=\mathcal A_k(0)$ by stationarity and we set $V_k:=\mathcal R[(-r,k)\to(0,1)]$, $\chi:=\mathcal{S}(x,0)+x^2$.
By the proof of Lemma \ref{lemma: finite-depth-tail} (\eqref{eq: moment gen func airy lpp} specifically) and Cauchy--Schwarz we have for any $\lambda > 0$,
\begin{align*}
    \mathbb E \mathrm{e}^{\lambda(U^x_k+V_k)}
    &\le
    \left(
        \mathbb E \mathrm{e}^{2\lambda U^x_k}
        \mathbb E \mathrm{e}^{2\lambda V_k}
    \right)^{1/2}\le
    \exp\left[
        C\lambda\sqrt{k}(\sqrt{x}+\sqrt{r})
        +C\lambda^2 (x+r)
    \right].
\end{align*}
An application of Markov's inequality to the random variable $\mathrm{e}^{\lambda(U^x_k+V_k)}$, with $\lambda > 0$ chosen appropriately gives
\[
    \mathbb P(U^x_k+V_k\ge u)
    \le
    \exp\left[
        -\frac{
            c\bigl(u-C\sqrt{k}(\sqrt{x}+\sqrt{r})\bigr)_+^2
        }{x+r}
    \right]\,,\qquad u > 0\,.
\]

Fix $a>0$ sufficiently small (as required by Lemma \ref{lemma: airy counting 2}) and set $h=ak^{2/3}$. If $\mathcal R_k(-r)\le-4h$ and $\chi\ge-h$, then
\eqref{eq: intercept-inclusion} forces
$U^x_k+V_k\ge3h$. Hence, by a union bound we obtain the estimate
\begin{align*}
    \mathbb P(Z_k(x,0)>-r)
    &\le
    \mathbb P(\mathcal R_k(-r)>-4h)
    +\mathbb P(\chi<-h)+\mathbb P(U^x_k+V_k\ge3h)\\
    &\le
    \mathbb P(\mathcal A_k(0)>-4h)
    +\mathbb P(\chi<-h)+\mathbb P(U^x_k+V_k\ge3h)\,.
\end{align*}
The Airy counting estimate, Lemma \ref{lemma: airy counting 2} and the GUE Tracy--Widom lower
tail of $\chi$ give
\begin{equation}
\label{eq: explicit-intermediate-tail}
\begin{split}
    \mathbb P(Z_k(x,0)>-r)
    \le{}&
    Ce^{-ck\log(k+1)}\\
    &+
    C\exp\left[
        -\frac{
            c\bigl(ak^{2/3}
               -C\sqrt{k}(\sqrt{x}+\sqrt{r})\bigr)_+^2
        }{x+r}
    \right],
\end{split}
\end{equation}
where the Tracy--Widom term $Ce^{-ck^2}$ has been absorbed
into the first term. Since $\sqrt{x}+\sqrt{r}\le\sqrt{2(x+r)}$, there is a universal $C_0$ such that, whenever $k\ge C_0(1+(x+r)^3)$ the positive part in \eqref{eq: explicit-intermediate-tail}
is at least a fixed positive multiple of $k^{2/3}$.
Consequently,
\begin{equation}
\label{eq: four-thirds-intercept}
    \mathbb P(Z_k(x,0)>-r)
    \le
    Ce^{-ck\log(k+1)}
    +Ce^{-ck^{4/3}/(x+r)},
    \qquad
    k\ge C_0(1+(x+r)^3)\,,
\end{equation}
Under the same restriction on $k$, namely $k\ge C_0(1+(x+r)^3)$ we have $k^\frac{4}{3}/(x+r)\ge c_0 k$ for some constant $c_0$. Thus \eqref{eq: four-thirds-intercept} implies
\[
    \mathbb P(Z_k(x,0)>-r)\le Ce^{-ck}
    \qquad\text{for }k\ge C_0(1+(x+r)^3).
\]
For $k<C_0(1+(x+r)^3)$, the trivial bound by $1$ implies,
after enlarging the universal constants,
\[
    \mathbb P(Z_k(x,0)>-r)
    \le C\exp(Cs^3-ck).
\]
Combining the two regimes gives, for every $k\ge2$,
\[
    \mathbb P(Z_k(x,0)>-r)
    \le C\exp\bigl(C(x+r)^3-ck\bigr).
\]
Finally, the algebraic inequality $(x+r)^3\le4(x^3+r^3)$ and the endpoint translation identity yields
\[
    \mathbb P\bigl(Z_k(x,y)>y-r\bigr)
    \le C\exp(Cx^3+Cr^3-ck).
\]
Taking $r=y>0$ gives the claimed bound for
$\mathbb P(Z_k(x,y)>0)$. 
\end{proof}

\begin{figure}[ht]
    \centering
    \includegraphics[width=0.45\linewidth]{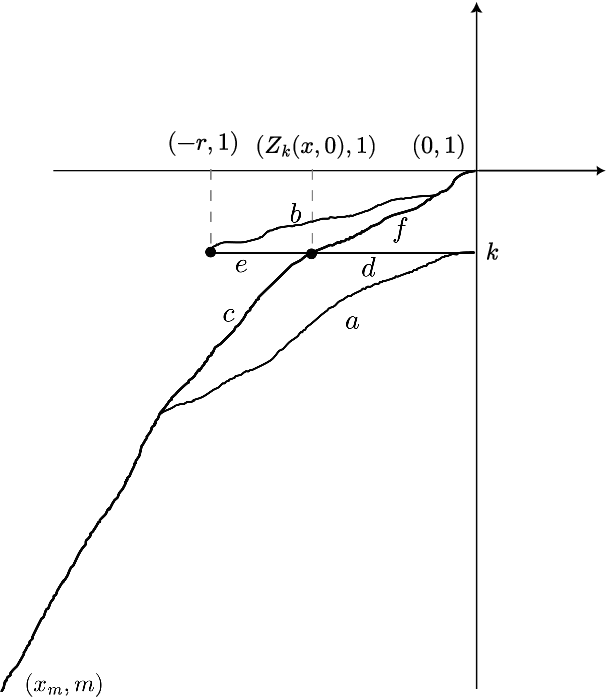}
    \caption{Illustration of a geometric argument for the inclusion \eqref{eq: intercept-inclusion} on the event $\{Z_k(x, 0) > -r\}$. Here $x_m = -\sqrt{m}/\sqrt{2x}$. By the metric composition law, we have the inequalities (the letters denoting last passage values of their respective geodesic segments) $c+d\le a$ and $f+e\le b$, giving an upper bound on the last passage value $e+f$ in terms of $a, b$ and the horizontal segment $e+d$.}
    \label{fig: geod length bound}
\end{figure}

\subsection{Geodesic coalescence depth concentration}\label{sec: coal depth}

Motivated by the tree-like nature of finite geodesics in last passage percolation, Lemma \ref{lemma: geod tree} and the qualitative coalescence results of Lemma \ref{lemma: semi-inf geod properties}, we will now obtain some quantitative bounds on the coalescence depth of semi-infinite geodesics in the Airy line ensemble.

Recall from the discussion after Lemma \ref{lemma: close-not-dj}, $Z_k(u,y)$ denotes the jump time from level $k$ to level
$k-1$ of the semi-infinite geodesic $\pi[u\to(y,1)]$, for $k\ge2$. More generally, for a path ending on level $\ell$, write $Z_j(u,(0, \ell))$ for its jump time from $j$ to $j-1$, $j\ge\ell+1$. For the $x$-directed semi-infinite geodesics $\pi[x\to (0, 1)]$ and $\pi[x\to (0, \ell)]$, we define the coalescence depth by
\[
 N_\ell(x):=\inf\left\{n\ge\ell:
 Z_j(x,0)=Z_j(x,(0, \ell))\text{ for every }j\ge n+1\right\}.
\]
Thus, after level $N_\ell(x)$, the two geodesics (also viewed as zig-zags) coincide.
Constants denoted by $c,C$ may change from line to line.

We prove polynomial concentration for the coalescence depths $N_\ell(x)$, which are almost surely finite for fixed $x, \ell$ by geodesic ordering considerations in the proof of Lemma \ref{lemma: semi-inf geod properties}. 
\begin{theorem}\label{thm: coal depth}
Fix $x>0$ and $\ell\ge2$. There exist constants $c_x,C_x>0$ such that,
for every integer $n\ge\max\{3,\ell\}$,
\begin{equation}\label{eq:main}
 \PP(N_\ell(x)>n)
 \le \min\left\{1,
 C_x(1+\ell+\log n)^{2/3}n^{-1/56}
 \exp\bigl(C_x(\log n)^{5/6}\bigr)\right\}.
\end{equation}
In particular, for fixed $x,\ell$ and every $\eta>0$,
\[
 \PP(N_\ell(x)>n)\le C_{x,\ell,\eta}n^{-1/56+\eta}\,.
\]
For $\ell=1$ the paths coincide and one can take $N_1(x)=1$.
\end{theorem}

Before we move to the proof of Theorem \ref{thm: coal depth}, we need some supporting results. First we establish upper tail concentration bounds for the depth two semi-infinite geodesics with asymptotic directions $x-\varepsilon, x$ eventually separate, again this is almost surely finite by asymptotic parabolic shape of geodesics, see the coupling in page \pageref{Skorokhod coupling} also.

\begin{proposition}[First strict separation]\label{prop: diff dir separation}
Fix a compact interval $I\subset(0,\infty)$. For $x\in I$, $y\in\R$
and $0<\varepsilon<\varepsilon_I$, set
\[
 K_\varepsilon(x,y):=
 \inf\left\{k\ge2:Z_k(x-\varepsilon,y)<Z_k(x,y)\right\}.
\]
There exist $c_I,C_I>0$ such that, for all $k\ge3$,
\begin{equation}\label{eq:Ktail}
 \PP(K_\varepsilon(x,y)>k)
 \le \min\left\{1,
 C_Ik^2\exp\left(-c_I\varepsilon^{4/3}k^{1/21}\right)
 +C_Ik\exp\left(-c_Ik^{1/63}\right)\right\}.
\end{equation}
In particular, for every $M>0$ there exist $B_{I,M},C_{I,M}>0$ such that $ \varepsilon\ge B_{I,M}k^{-1/28}(\log k)^{3/4}$ gives the bounds
\begin{equation}\label{eq:Kpoly}
 \PP(K_\varepsilon(x,y)>k)\le C_{I,M}k^{-M}.
\end{equation}
The constants do not depend on $y$.
\end{proposition}

\begin{lemma}[Quantitative jump-time localization]\label{lemma: jump parab}
Fix a compact $I\subset(0,\infty)$. For some $c_I,C_I,d_I>0$,
all $u\in I$, $y\in\R$, $k\ge3$ and $0<d<d_I$ satisfy
\begin{align}\label{eq:jump}
 &\PP\left(\left|Z_k(u,y)-y+\sqrt{\frac{k}{2u}}\right|
                 >d\sqrt{k}\right)\notag\\
 &\hspace{15mm}\le
 \min\left\{1,
 C_Ik^2e^{-c_Id^{4/3}k^{1/21}}
 +C_Ike^{-c_Ik^{1/63}}\right\}.
\end{align}
\end{lemma}

\begin{proof}
Fix compact $I\subset (0, \infty), k \ge 3$ and $y\in \R$. By the stationarity of geodesics in the Airy line ensemble, Lemma \ref{lemma: semi-inf geod properties}, we have $Z_k(u, y) \overset d= Z_k(u, 0) + y$. We can thus without loss of generality set $y=0$. We will henceforth assume that there exist unique semi-infinite geodesics with rational endpoints (or on any fixed set of endpoints we use in the remainder of the argument), guaranteed to occur with probability one by Lemma \ref{lemma: semi-inf geod properties}. We will also assume the Airy line ensemble is coupled to a version of the Airy sheet $\mathcal{S}$, see \ref{coupling: Airy sheet}. 

Fix $q, q'\in I$ rational Then, by Proposition \ref{prop: semi-inf geod jump} and Lemma \ref{lemma: semi-inf geod properties}, with $Z_k(q, 0)$ the jump time of the semi-infinite geodesic $\pi[q\to (0, 1)]$ from level $k$ to $k-1$ satisfies
\[
Z_k(q, 0) = \displaystyle\operatorname*{argmax}_{s\in (-\infty, 0]}\bigg(\Delta_{(s, k)}^{(-r, 1)}(q')+\mathcal{A}[(s, k)\to (y,1)]\bigg)\,,
\]
and
\[
Z_k(q', 0) = \displaystyle\operatorname*{argmax}_{s\in (-\infty, 0]}\bigg(\Delta_{(s, k)}^{(-r, 1)}(q')+\mathcal{A}[(s, k)\to (y,1)]\bigg)\,.
\]
Now, for fixed $r>0$, by the metric composition law and the coupling between $\mathcal{A}, \mathcal{S}$, this gives the almost sure inequality,
\[
\mathcal{S}(q', 0)-\mathcal{S}(q', -r) - \mathcal{S}(q, 0)+\mathcal{S}(q, -r)\ge \Delta_{(Z_k(q, 0), k)}^{(-r, 1)}(q') - \Delta_{(Z_k(q, 0), k)}^{(-r, 1)}(q)\,.
\]
Since the Busemann functions above are evaluated at rational endpoints, we can express the above as the almost sure limit
\begin{align*}
    \Delta_{(Z_k(q, 0), k)}^{(-r, 1)}(q') - \Delta_{(Z_k(q, 0), k)}^{(-r, 1)}(q) &= \lim_{m\to \infty} \mathcal{A}[(-\sqrt{m}/\sqrt{2q'}, m)\to (Z_k(q, 0), k)]\\
    &-\mathcal{A}[(-\sqrt{m}/\sqrt{2q}, m)\to (Z_k(q, 0), k)]\\
    &-\mathcal{A}[(-\sqrt{m}/\sqrt{2q'}, m)\to (-r, 1)]+\mathcal{A}[(-\sqrt{m}/\sqrt{2q}, m)\to (-r, 1)]\,.
\end{align*}
On the event $\{Z_k(q, 0) > -r\}$, we have by the crossing lemma for geodesics, Lemma \ref{lemma: crossing} the lower bound
\begin{align*}
    \Delta_{(Z_k(q, 0), k)}^{(-r, 1)}(q') - \Delta_{(Z_k(q, 0), k)}^{(-r, 1)}(q) &\ge \lim_{m\to \infty} \mathcal{A}[(-\sqrt{m}/\sqrt{2q'}, m)\to (-r, k)]\\
    &-\mathcal{A}[(-\sqrt{m}/\sqrt{2q}, m)\to (-r, k)]\\
    &-\mathcal{A}[(-\sqrt{m}/\sqrt{2q'}, m)\to (-r, 1)]+\mathcal{A}[(-\sqrt{m}/\sqrt{2q}, m)\to (-r, 1)]\\
    &=  \Delta_{(-r, k)}^{(-r, 1)}(q') - \Delta_{(-r, k)}^{(-r, 1)}(q)\,.
\end{align*}
By Theorem \ref{thm: Busemann equality in law}, we have
\[
 \Delta^{(-r, 1)}_{(-r, k)}(q')-\Delta^{(-r, 1)}_{(-r, k)}(q)+\mathcal{S}(q', -r)-\mathcal{S}(q, -r)+2z(u-v) \overset d = \mathcal{A}[(0, k) \to (q', 1)] - \mathcal{A}[(0, k) \to (q, 1)]\,.
\]
We thus have the inclusion (for a geometric proof, see Figure \ref{fig: geod inc parab loc})
\[
\{Z_k(q, 0)> -r\} \subseteq \{H^{q, q'} \ge 2r(q-q') + U^{q, q'}_k\}\,,
\]
where
\[
\begin{aligned}
    H^{q, q'} &:= \mathcal{S}(q', 0) - \mathcal{S}(q, 0)\\
    U^{q, q'}_k & \overset d = \mathcal{A}[(0, k) \to (q', 1)] - \mathcal{A}[(0, k) \to (q, 1)] \,.
\end{aligned}
\]
Arguing in an entirely analogous way, we obtain for $0< q'< q$, on the event $\{Z_k(q, 0)< -r\}$ (flipping $q, q'$ in the upper tail event inclusion),
\[
\{Z_k(q, 0) < -r\} \subseteq \{H^{q', q} \ge 2r(q'-q) + U^{q', q}_k\}\,.
\]
Now, for fixed $u, v\in I$, taking rational $q, q'$ approximating $u, v$ respectively and using the continuity of the Airy sheet and the almost sure continuity of the Buseman functions $\Delta^{(-r, 1)}_{(-r, k)}$ at $u, v$ (in fact, it is locally constant in a neighbourhood of $\{u, v\}$), we have almost surely,
\[
\begin{aligned}
    &\{Z_k(u,0)>-r\}\subseteq \{ H^{u, v}  \ge  2r(u-v) + U^{u, v}_k\}\,,\quad 0< u < v\\
    & \{Z_k(u,0)<-r\}\subseteq \{ H^{v, u} \ge 2r(v-u) + U^{u, v}_k\}\,,\quad 0< v < u
\end{aligned}
\]

Fix $d> 0$ and $u \in I$. By a union bound, we can thus estimate 
\begin{align*}
    &\PP\left(\left|Z_k(u,0)+\sqrt{\frac{k}{2u}}\right|
                 >d\sqrt{k}\right)\le \PP\left(Z_k(u, 0) < -\sqrt{\frac{k}{2u}}+d\sqrt{k}\right) + \PP\left(Z_k(u, 0) > -\sqrt{\frac{k}{2u}}-d\sqrt{k}\right)\,.
\end{align*}
To estimate the first term, taking in the above couplings $r = \sqrt{k}/\sqrt{2u}-d\sqrt{k}$ and $v = u + \theta d$, for some $\theta>0$ sufficiently small to be chosen later, we estimate,
\begin{align*}
    &\PP\left(Z_k(u, 0) < -\sqrt{\frac{k}{2u}}+d\sqrt{k}\right)\le \PP\bigg( H^{u, v} -U^{u, v}_k\ge -2d(\sqrt{k}/\sqrt{2u}-d\sqrt{k})\theta\bigg)\\
    &\le \PP\bigg( |H^{u, v}+v^2-u^2| + |U^{u, v}_k+v^2-u^2-2\sqrt{2k}(\sqrt{v}-\sqrt{u})|\ge D\bigg)\,,
\end{align*}
where for $u\in I$, $\theta$ sufficiently small (possibly depending on $I$), possibly depending on $I$,
\[
D = 2\sqrt{k}\left(\sqrt{2}(\sqrt{v}-\sqrt{u}) - d(1/\sqrt{2u}-d)\theta\right) =  2\sqrt{k}\left(\sqrt{2}(\sqrt{u+\theta d}-\sqrt{u}) - \theta d/\sqrt{2u}+ d^2\theta\right)\,.
\]
for some $I$ dependent constant $c_I> 0$.
By Taylor's theorem, we have $\sqrt{u+\theta d}- \sqrt{u}$ $= \theta d /(2\sqrt{u}) + O(d^2\theta^2)$ for all $\theta, d > 0$ sufficiently small. This gives the lower bound
\[
D = 2\sqrt{k}\left(d^2\theta + O(d^2\theta^2)\right) \ge c_I d^2\sqrt{k}\,,
\]
again for $\theta > 0$ sufficiently small.

A union bound, standard GUE Tracy-Widom bounds for the stationary Airy$_2$ process, see Lemma 9.5 in \cite{DOV} give 
\[
\PP(|H^{u, v}+v^2-u^2|\ge c_I d^2\sqrt{k})\le Ce^{-c(d^2\sqrt{k})^{3/2}}\,.
\]
There exists some $d_I$ so that for all $0< d< d_I$, Lemma~\ref{lemma: shape} applied to the `last passage' term $U^{u, v}_k$, with
\[
 1\le s:=c_Id^{4/3}k^{1/21}\le k^{1/14},\qquad
 k^{3/7}s^{3/2}=c_Id^2\sqrt{k}\,,
\]
gives for all $k$ sufficiently large
\[
\PP(|U^{u, v}_k+v^2-u^2-2\sqrt{2k}(\sqrt{v}-\sqrt{u})|\ge c_Id^2\sqrt{k})\le C_Ik^2e^{-c_Id^{4/3}k^{1/21}}
 +C_Ike^{-c_Ik^{1/63}}\,.
\]
By enlarging $C_I$, we can trivially extend the above bounds to all $k\ge 3$.

For the lower tail of $Z_k(u, 0)$, use the above coupling with $H^{v, u}$ and $U^{u, v}_k$, taking $r=-\sqrt{k}/\sqrt{2u}-d\sqrt{k}$ and $v=u-\theta d$ and proceed in an entirely analogous way, concluding the proof.
\end{proof}

\begin{figure}[hb]
    \centering
    \includegraphics[width=0.95\linewidth]{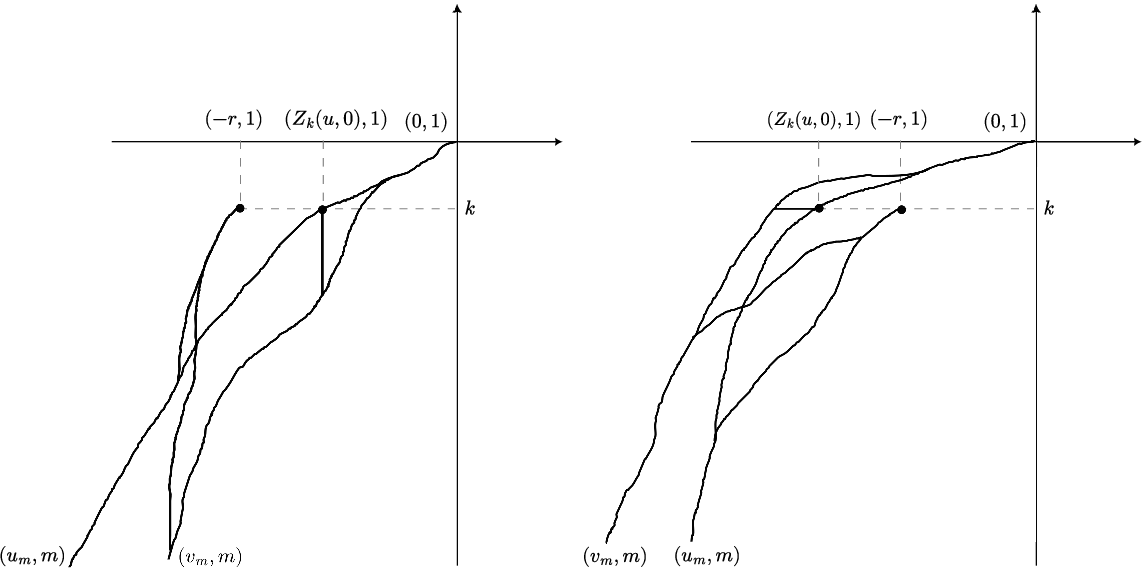}
    \caption{Geometric representation of the estimates in the proof of Lemma \ref{lemma: jump parab}, on the events $\{Z_k(u, 0) > -r\}$ (left) and $\{Z_k(u, 0) < -r\}$ (right). In both cases, the metric composition law is used to estimate from below the relative lengths of the paths starting at $(u_m, m)$, $(v_m ,m)$ anding at $(0, 1)$ in terms of the relative lengths of paths from  $(u_m, m)$, $(v_m ,m)$ to $(Z_k(u,0), k)$ (here  $x_m = -\sqrt{m}/\sqrt{2x}$, $x> 0$). Then, the monotonicity of the aforementioned relative lengths, due to Lemma \ref{lemma: crossing} compares them to the relative lengths of geodesics from  $(u_m, m)$, $(v_m ,m)$ to $(-r, k)$.}
    \label{fig: geod inc parab loc}
\end{figure}

\begin{proof}[Proof of Proposition~\ref{prop: diff dir separation}]
By monotonicity of geodesics, $Z_k(x-\varepsilon,y)\le Z_k(x,y)$.
Moreover, we have the uniform lower bound on some neighbourhood $U_I$ of $I$ on the positive reals separated away from zero
neighbourhood of $I$,
\[
 \sqrt{\frac{k}{2(x-\varepsilon)}}-
 \sqrt{\frac{k}{2x}}= \sqrt{\frac{k}{2x(x-\varepsilon)}}\cdot\frac{\varepsilon}{\sqrt{x}+\sqrt{x-\varepsilon}}\ge c_I\varepsilon\sqrt{k}\,,\qquad x\in I, x-\varepsilon \in U_I\,.
\]
Now, if both jump times $Z_k(x-\varepsilon,y), Z_k(x,y)$ lie within $c_I\varepsilon\sqrt{k}/3$ of $-\sqrt{k}/\sqrt{2(x-\varepsilon)}$ and $-\sqrt{k}/\sqrt{2x}$ respectively, then they must be positively separated.
Thus on event $\{K_\varepsilon>k\}$, at least one jump time must have separation at least $c_I\varepsilon\sqrt{k}/3$ from its `parabolic' counterpart. Lemma~\ref{lemma: jump parab} with $d=c_I\varepsilon/3$ gives the estimate in \eqref{eq:Ktail}. 

Fix $M>0$. Observe that for positive $B>0$,
\[
 \left(B k^{-1/28}(\log k)^{3/4}\right)^{4/3}k^{1/21}
   =B^{4/3}\log k.
\]
Choosing $B$ large enough gives the first term in \eqref{eq:Ktail}
is at most $Ck^{-M}$; the second has the same bound after adjusting
$C$.
\end{proof}

\begin{lemma}[The deterministic coalescence comparison]\label{lemma:geometry}
Fix $x>0$, $0<\varepsilon<x$, $Y>0$ and $n\ge\ell$.
Then, outside a null event depending on these fixed parameters, with
\[
E_\varepsilon(0, Y): = \{\mathsf{ZigZag}(\pi[x-\varepsilon\to(0,1)])\text{ and }\mathsf{ZigZag}(\pi[x+\varepsilon\to(Y,1)]) \text{ touch}\}\,,
\]
the following inclusion holds,
\begin{equation}\label{eq:inclusion}
 \{N_\ell(x)>n\}\subseteq
 E_\varepsilon(0,Y)^c
 \cup\{Z_\ell(x+\varepsilon,Y)\le0\}
 \cup\{K_\varepsilon(x,0)>n\}.
\end{equation}
\end{lemma}

\begin{proof}
Suppose the three events on the right do not occur. Write
\[
 \alpha=\pi[x-\varepsilon\to(0,1)],\quad
 \beta=\pi[x\to(0,1)],\quad
 \gamma=\pi[x\to(0,\ell)],\quad
 \delta=\pi[x+\varepsilon\to(Y,1)]\,.
\]
Since $Z_\ell(x+\varepsilon,Y)>0$, $\delta(0)\ge \ell+1$. Geodesic monotonicity gives the pointwise ordering
\[
 \alpha\le\beta\le\gamma\le\delta\qquad\text{on }(-\infty,0].
\]
Moreover, there exists some point $(z, j)$ for $z < 0$, $j\ge \ell+1$ at which the zig-zags of $\alpha$ and $\delta$ touch (see Figure \ref{fig: geod coal depth}). By monotonicity this means all of the zig-zags of $\beta, \gamma$ also meet at $(z, j)$. By Lemma \ref{lemma: geod tree}, $j\le K_\varepsilon(x, 0) \le n$ since otherwise this would produce two distinct semi-infinite geodesics with asymptotic direction $x$ and right endpoint $(0, 1)$, violating the uniqueness of semi-infinite geodesics in the Airy line ensemble, Lemma \ref{lemma: semi-inf geod properties}. Now, since by Lemma \ref{lemma: semi-inf geod properties}, the geodesics $\pi[x\to (0, \ell)]$ and $\pi[x\to (0, 1)]$ eventually coalesce, the same geodesic uniqueness argument gives that they must have already coalesced by level $j \le n$, giving $N_\ell(x)\le n$, concluding the proof.
\begin{figure}[ht]
    \centering
    \includegraphics[width=0.6\linewidth]{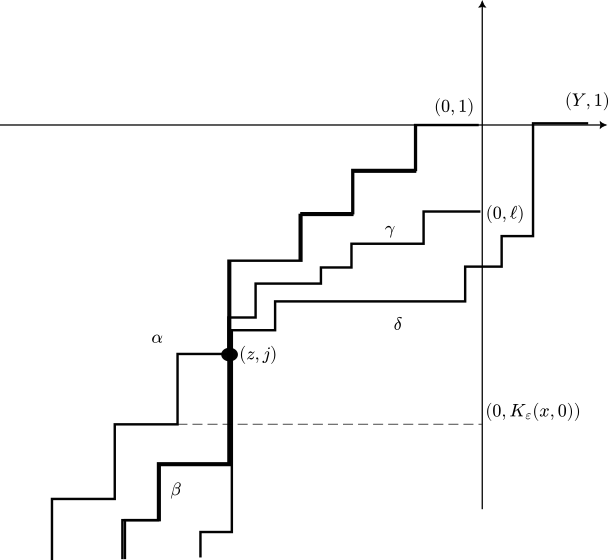}
    \caption{Illustration of a typical realisation of geodesics $\alpha, \beta, \gamma, \delta$ on the event $E_\varepsilon(0, Y)$ where $\alpha$ and $\delta$ touch and the path $\delta$ descends below level $\ell + 1$ on the vertical axis.}
    \label{fig: geod coal depth}
\end{figure}
\end{proof}

\begin{proof}[Proof of Theorem~\ref{thm: coal depth}]

Arguing as in the proof of Theorem \ref{prop: non-disj-close-endpts-inf-geo}, we have for any $x> 0, y > 0, \ell \ge 2$,
\[
\PP\bigl(Z_\ell(x,0)\le-y\bigr) \le C\cdot \exp (\ell - ca^{\frac{3}{2}})\,,
\]
for constants $c, C> 0$. This gives the uniform bounds,
\begin{equation}\label{eq:shield}
 \PP(Z_\ell(u,Y)\le0)
 =\PP(Z_\ell(u,0)\le-Y)
 \le C_x\exp\bigl(\ell-c_xY^{3/2}\bigr)\,,\qquad u\in[x/2,3x/2]\,.
\end{equation}

Combining Lemmas~\ref{lemma: close-not-dj} and \ref{lemma:geometry} with
Proposition~\ref{prop: diff dir separation} and a union bound give the estimates
\begin{align}\label{eq: coal bound 1}
 \PP(N_\ell(x)>n)
 &\le C_x(1+Y)\varepsilon^{1/2}
             \mathrm{e}^{C_x|\log\varepsilon|^{5/6}}
       +C_xe^{\ell-c_xY^{3/2}}\notag\\
 &\quad+C_xn^2e^{-c_x\varepsilon^{4/3}n^{1/21}}
       +C_xn \mathrm{e}^{-c_xn^{1/63}}\,.
\end{align}
Fix $M=3$. For all sufficiently large $n$, choose
\[
 \varepsilon_n=B_xn^{-1/28}(\log n)^{3/4},\qquad
 Y_n=D_x(1+\ell+\log n)^{2/3},
\]
where $B_x,D_x$ are sufficiently large. Then the last three terms
of \eqref{eq: coal bound 1} are at most $C_xn^{-3}$. The first is at most
\[
 C_x(1+\ell+\log n)^{2/3}n^{-1/56}(\log n)^{3/8}
           \mathrm{e}^{C_x(\log n)^{5/6}} =  C_x(1+\ell+\log n)^{2/3}n^{-1/56}
           \mathrm{e}^{C_x(\log n)^{5/6}}
\]
after enlarging $C_x$. After increasing $C_x$ to cover the finite range where
$\varepsilon_n$ is not sufficiently small, this proves
\eqref{eq:main}. The second expression follows from the fact
$\exp(C(\log n)^{5/6})=n^{o(1)}$. 
\end{proof}

\section{Exceptional behaviour of semi-infinite geodesics}\label{sec: exceptional geod geom}

In this section, we investigate many exceptional sets of geodesic non-uniqueness or coalescence, in Corollaries \ref{corollary: geod two end points disjoint}, \ref{cor: geod disj same endpt}, Theorem \ref{thm: fixed dir coal} and Theorem \ref{thm: geod no coal h dim}, Proposition \ref{prop: multiple geod fixed endpt}, respectively.

As a direct application of Theorems \ref{thm: haus dim ell} and \ref{thm: busemann cont} and Proposition \ref{prop: local variation geod disj} thus obtain the following corollary (see Theorems 8.2, 2.11 in \href{https://arxiv.org/pdf/2203.13242}{this paper} for an analogous statement for Busemann functions in the directed landscape).

\begin{corollary}\label{corollary: geod two end points disjoint}
 For any  $\ell \ge 1$ and $y_1< y_2$, almost surely
    \[
    \mathrm{dim}_{\mathscr{H}}\bigg(\bigg\{x>0:  \pi^L[x\to (y_1, 1)]\text{ and } \pi^R[x\to (y_2, \ell)] \text{ are disjoint} \bigg\}\bigg) = \frac{1}{2}.
    \]
\end{corollary}

Additional geometric information about semi-infinite geodesics in the Airy line ensemble can be revealed by considering the Airy sheet with its `additive part' subtracted. Consider the random measure $\nu$ on $(0, \infty)^2$ defined by first setting
\[
\nu((0, x]\times (0, y]) := \mathcal{S}(x, y) - \mathcal{S}(x,0) - \mathcal{S}(0, y) + \mathcal{S}(0, 0)
\]
for all $x> 0, y> 0$. By the quadrangle inequality for the Airy sheet, \eqref{eq: quadrangle ineq airy sheet} , $\nu((0, x]\times (0, y])$ is monotone increasing in $x, y$ and also jointly continuous (by the continuity of the Airy sheet). Thus, it is the cumulative distribution function of a Borel measure, which we also denote by $\nu$. This can be extended to a Borel measure on $\R^2$ by translation symmetry of the Airy sheet, \eqref{eq: Airy sheet symmetry}. 

The topological support of $\nu$ is related to the disjointness of geodesics in the directed landscape, see \cite{Gangulyetal2022}. For an alternative proof of this result (and characterisation of the support of $\nu$) in terms of geodesics in the directed landscape, see the discussion preceding \cite[Theorem 1.10]{BatesGangulyHammond2022}. In the following theorem, we characterise the support of $\nu$ in terms of disjointness of semi-infinite geodesics in the Airy line ensemble and compute its Hausdorff dimension, in parallel to \cite{BatesGangulyHammond2022}.

\begin{theorem}\label{thm: airy sheet monotone haus dim}
    The Hausdorff dimension of the topological support of the random measure $\nu$ defined above satisfies almost surely
    \[
    \mathrm{dim}_{\mathscr{H}}(\mathrm{supp}(\nu)) = \frac{1}{2}.
    \]
    Moreover, $\mathrm{supp}(\nu)$ is a closed nowhere dense subset of $\R^2$.
\end{theorem}
\begin{proof}

By translation invariance of the Airy sheet $\mathcal{S}(\cdot+ t, \cdot + t) \stackrel{d}{=}\mathcal{S}(\cdot+ t, \cdot + t)$ for any fixed $t\in \R$, \eqref{eq: Airy sheet symmetry} we can without loss of generality take
 \[
    \nu((0, x]\times (0, y]) := \mathcal{S}(x+1, y+1) - \mathcal{S}(x+1,1) - \mathcal{S}(1, y+1) + \mathcal{S}(1, 1)\,.
\]

   \underline{Upper bound}: By Theorem 1.1 of \cite{Hammond_2019}, arguing as in the proof of Lemma \ref{lemma: close-not-dj} using the construction of semi-infinite geodesics with rational endpoints as limits of last passage paths on the prelimiting Brownian melon, that is the coupling on page \pageref{Skorokhod coupling} we obtain for all $\varepsilon>0$, $y\in \R$
    \begin{equation}\label{eq: close-not-dj-near-endpts}
	\PP \bigg(\pi[x-\varepsilon \to (y-\varepsilon, 1)] \;\;\; \mathrm{and} \;\;\; \pi[x+\varepsilon \to (y+\varepsilon, 1)] \text{ are disjoint} \bigg) \le C\cdot\varepsilon^{\frac{3}{2}+o_\varepsilon(1)}\,,
	\end{equation}
    for some universal constant $C> 0$.

    By the coupling in \ref{coupling: Airy sheet} between the Airy sheet and the Airy line ensemble, we can express for all $(x, y) \in (0, \infty)\times (0, \infty)$,
    \[
    \nu((0, x]\times (0, y]) = \Delta_{(y+1, 1)}^{(1, 1)}(x+1) - \Delta_{(y, 1)}^{(1, 1)}(1)\,.
    \]
    Notice a quick computation shows for any rectangle $(a, b] \times (a', b']\subset (0, \infty)\times(0, \infty)$, 
     \[
    \nu((a, b]\times (a', b']) = \Delta_{(b'+1, 1)}^{(a'+1, 1)}(b+1) - \Delta_{(b'+1, 1)}^{(a'+1, 1)}(a+1)\,.
    \]
    
    Now, for any $a < b$ finite, consider the rectangle $[a, b]\times [a, b]\subset\R^2$. For any $\varepsilon > 0$, starting from $a$, partition $[a, b]\times [a, b]$ into a family of intervals of diameter $2\varepsilon$, possibly overshooting $(b, b)$ in the last step. Let $N_\varepsilon\le \lceil(b-a)/(2\varepsilon)\rceil^2$ be the random number of intervals $(I^\varepsilon_i)_i$ above for which $\nu(I^\varepsilon_i)>0$. Arguing as in Proposition \ref{prop: local variation geod disj} we have the equivalence
    \[
    \nu((a, b]\times (a', b'])> 0 \text{ if and only if } \pi[a , a'] \;\;\; \mathrm{and} \;\;\; \pi[b , b'] \text{ are disjoint}\,.
    \]
    Thus, we have by \eqref{eq: close-not-dj-near-endpts} for all $i\le N_\varepsilon$
    \[
    \PP(\nu(I^\varepsilon_i)> 0)\le C\cdot\varepsilon^{\frac{3}{2}+o_\varepsilon(1)}\,.
    \]
    Using linearity and taking expectations, we have using Markov's inequality the estimate
    \[
    \mathbb{E}[N_\varepsilon] \le \bigg\lceil\frac{b-a}{2\varepsilon}\bigg\rceil^2\cdot\PP \le C_{a, b}\cdot\varepsilon^{-\frac{1}{2}+o_\varepsilon(1)}
    \]
    for some constant $C_{a, b} > 0$.

    We can thus cover the support of $\nu$ in $[a, b]\times [a, b]$, $\mathrm{supp}(\nu, [a, b]\times [a, b])$ by $N_\varepsilon$ many intervals $(I^\varepsilon_i)_i$ of diameter $2\varepsilon$ such that for any fixed $\delta > 0$,
    \[
    \sum_i \diam(I^\varepsilon_i)^{\frac{1}{2}+\delta}\le C_\delta  \cdot N_\varepsilon\cdot \varepsilon^{\frac{1}{2}+\delta}\,.
    \]
    for some constant $C_\delta > 0$. Taking expectations, we have
    \[
     \mathbb{E}[\sum_i \diam(U^i_\varepsilon)^{\frac{1}{2}+\delta}]\le C_{a, b, \delta}\cdot \varepsilon^{-\frac{1}{2}+o_\varepsilon(1)} \cdot \varepsilon^{\frac{1}{2}+\delta} \le C_{a, b, \delta} \cdot \varepsilon^{\delta+ o_\varepsilon(1)}\,, 
    \]
    which converges to zero as $\varepsilon \to 0$. Thus, almost surely, the Hausdorff dimension of $\mathrm{supp}(\nu, [a, b]\times [a, b])$ is at most $1/2 +\delta$. Since $\delta$, $a< b$ were arbitrary, we obtain that the Hausdorff dimension of  $\mathrm{supp}(\nu, (0, \infty) \times (0, \infty))$ is at most $1/2$. To extend the upper bound to $\mathrm{supp}(\nu, \R^2)$, use the stationarity of the Airy sheet under uniform shifts in both coordinates, \eqref{eq: Airy sheet symmetry}.
    
\underline{Lower bound}: The lower bound follows from the pathwise regularity of the Airy sheet, in particular that it is H\"{o}lder $1/2^-$ continuous on $\R^2$ and that the differences $\mathcal{S}(x+1, y+1) - \mathcal{S}(x+1,1) - \mathcal{S}(1, y+1) + \mathcal{S}(1, 1)$ are asymptotically non-degenerate. In fact, by the global shape bounds of the Airy sheet, \eqref{eq: airy shape bnds}, almost surely as $x, y \to \infty$, $\nu((0, x]\times(0, y]) \to \infty$.

Suppose for a contradiction the Hausdorff dimension of $\mathrm{supp}(\nu, \R^2)$ is less than $1/2$. Then, we claim on any interval $[a, b]$ $\nu((0, x]\times(0, y])$ must be constant; in particular, it must vanish. Indeed, we estimate for all $\varepsilon > 0$ by the Holder continuity of $\mathcal{S}$
    \begin{equation}\label{eq: mass distribution busemann-two-dim}
         0 < \nu((a, b] \times (a, b]) \le \sum_{i}\nu(I^\varepsilon_i)\le C_{a, b}\cdot\sum_{i}\mathrm{diam}(I^\varepsilon_i)^{\frac{1}{2}+o_\varepsilon(1)}\,,
    \end{equation}
    for some possibly random $C_{a, b}>0$ where $(I^\varepsilon_i)$ cover $(a, b]\times (a, b]$ and have diameter $c\varepsilon$ for some universal $c> 0$. Since we assumed $\mathrm{supp}(\nu, [a, b]\times [a, b])$ has Hausdorff dimension less than $1/2$, say $\delta < 1/2$ we could have chosen the $(I^\varepsilon_i)_i$ so that
    \[
    \lim_{\varepsilon \to 0} \sum_{i}\mathrm{diam}(I^\varepsilon_i)^{\delta + (1/2-\delta)/2} = 0\,.
    \]
    This however, is a contradiction to \eqref{eq: mass distribution busemann-two-dim} since $\frac{1}{2}+o_\varepsilon(1) > \delta + (1/2-\delta)/2$ eventually as $\varepsilon \to 0$. 
    
    Since $a< b$ was arbitrary, we see that either the Hausdorff dimension of the support of $\nu$ on $(0, \infty)\times (0, \infty)$ is $1/2$, or $\nu((0, \infty) \times (0, \infty))=0$, a contradiction to its asymptotic non-degeneracy as discussed above. This concludes the proof of the lower bound.

    Since the Lebesgue measure of $\supp(\nu)$ necessarily vanishes by the above computation of the Hausdorff dimension, $\supp(\nu)$ must have empty interior, and is so nowhere dense.
\end{proof}

We thus obtain the following Corollary. 

\begin{corollary}\label{cor: geod disj same endpt}
    Almost surely, the set of directions
    \[
    \mathsf{Disj}_{\verticalbond{}}: =  \{(x, y)\in (0, \infty)\times \R: \pi^L[x\to (y, 1)] \text{ and } \pi^R[x\to (y, 1)] \text{ are disjoint strictly before }y\}\subseteq (0, \infty) \times \R\,.
    \]
    Moreover, its closure, $\overline{\mathsf{Disj}_{\verticalbond{}}}$ is a closed, nowhere dense set of Hausdorff dimension equal to $\frac{1}{2}$.
\end{corollary}

\begin{proof}
To compute the Hausdorff dimension of $\mathsf{Disj}_{\verticalbond{}}$, it suffices to observe that almost surely, 
\begin{align*}
    \mathsf{Disj}_{\verticalbond{}} = \mathrm{supp}(\nu, (0,\infty)\times \R)
\end{align*}
with $\nu$ as in Theorem \ref{thm: airy sheet monotone haus dim}.

\underline{$\subseteq$}: Indeed, by the monotonicity of left most/rightmost geodesics with respect to their endpoints and asymptotic directions, we have that if $(x, y)\in \mathsf{Disj}_{\verticalbond{}}$, then for any rational $0<u< x< v$ and $w< y< z$, we have the unique (item $(2)$ in Lemma \ref{lemma: semi-inf geod properties}) geodesics $\pi[u\to (w, 1)]$ and $\pi[v\to (z, 1)]$ are disjoint before $y$ (any possibly common vertical interval is taken care of by taking $w< y$). Then arguing exactly as in Proposition \ref{prop: local variation geod disj}, we obtain that $x\in \mathrm{supp}(\nu, (0, \infty)\times \R)$. This gives the upper bound of $1/2$ by Theorem \ref{thm: airy sheet monotone haus dim}.

\underline{$\supseteq$}: Now if $x\in  \mathrm{supp}(\nu, (0, \infty)\times \R)$ then for any rational $0<u< x< v$ and $w< y< z$ we have the unique (item $(2)$ in Lemma \ref{lemma: semi-inf geod properties}) $\pi[u\to (w, 1)]$ and $\pi[v\to (z, 1)]$ are disjoint.

Suppose $(x, y) \notin \mathsf{Disj}_{\verticalbond{}}$. Then by geodesic uniqueness in the Airy line ensemble, see the remark after Lemma \ref{lemma: geod location brown env lpp} and the approximation Lemma \ref{lemma: geodesic approx} (see Figure \ref{fig: geod approx}), it cannot be the case that the point of intersection $(z, \ell)$ of the zigzags (recall \eqref{eq: zig-zag}) of $\pi^L[x\to (y, 1)]$ and $\pi^R[x\to (y, 1)]$ occurs after a horizontal segment of both paths. Since if that were the case, there would exist $0<u< x< v$ and $w< y< z$ such that the unique geodesics (item $(2)$ in Lemma \ref{lemma: semi-inf geod properties}) $\pi[u\to (w, 1)]$ and $\pi[v\to (z, 1)]$ are not disjoint; in particular, they would share $(z, \ell)$. Arguing as in the proof of Proposition \ref{prop: local variation geod disj} would thus contradict $x\in \mathrm{supp}(\nu, (0, \infty)\times \R)$. This means the geodesics must be disjoint before $y$, which concludes the proof of this part.

To prove $ \mathsf{Disj}_{\verticalbond{}}$ is nowhere dense, observe that by the approximation lemma, Lemma \ref{lemma: geodesic approx}, for all rational pairs $(q, r) \in (0, \infty)\times \R$ the uniqueness of geodesics with such endpoints (item $(2)$ in Lemma \ref{lemma: semi-inf geod properties}) guarantees that there exists a an open set $(q,r)\ni U_{q, r} \not \in \mathrm{supp}(\nu, (0,\infty)\times \R)$. Closedness follows immediately by the equality previously established, which means $ \mathsf{Disj}_{\verticalbond{}}$ is the topological support of a random measure.
\end{proof}
\begin{figure}[ht]
    \centering
    \includegraphics[width=0.65\linewidth]{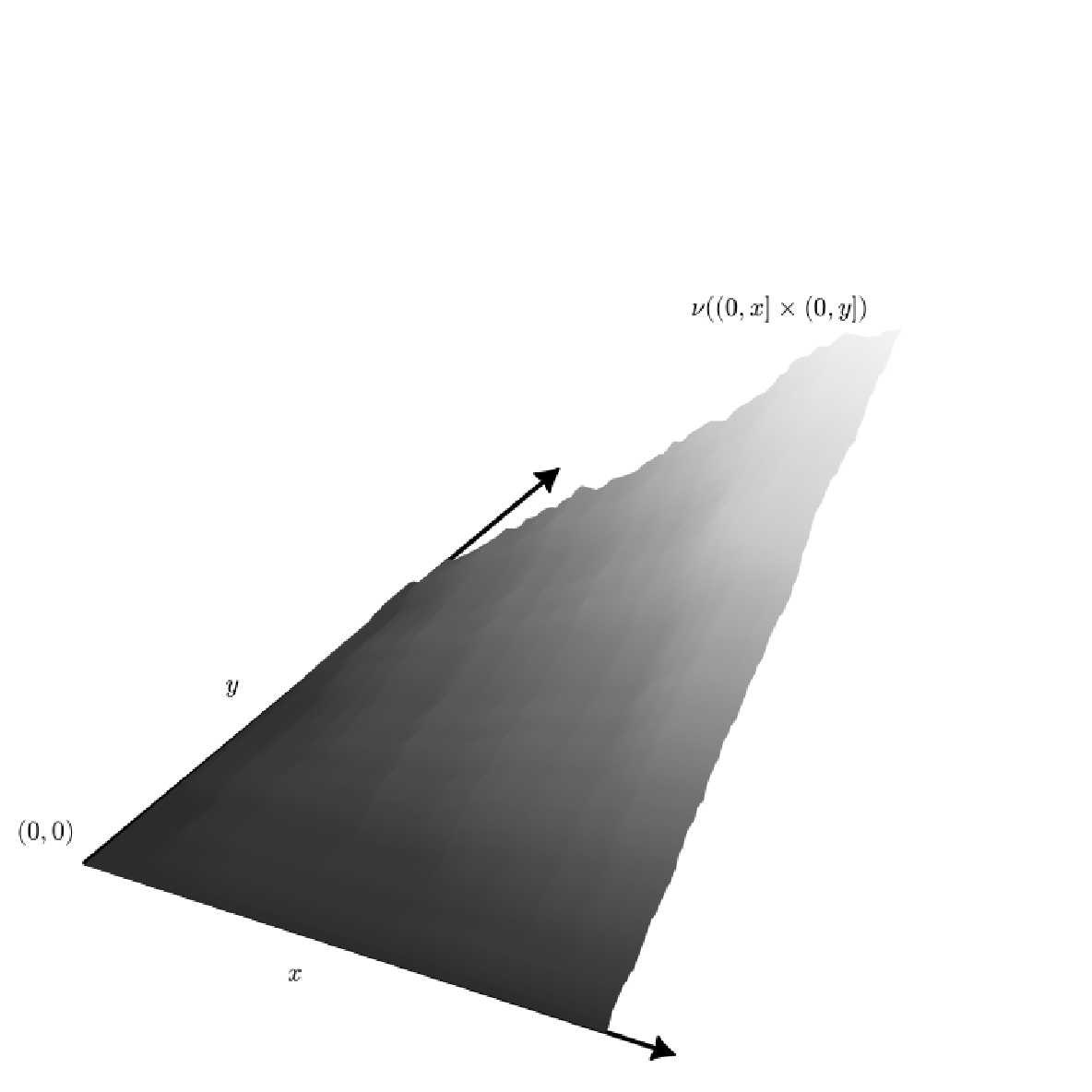}
    \caption{Illustration of the map $(x, y)\mapsto  \nu((0, x]\times (0, y])$ on some bounded rectangle of $[0, \infty)\times [0, \infty)$.}
    \label{fig: airy diff cdf}
\end{figure}

We now consider for a fixed asymptotic direction, exceptional right endpoints where semi-infinite geodesic non-uniqueness occurs. This is the content of the following theorem.

\begin{theorem}\label{thm: fixed dir coal}
    Fix $x>0$. Almost surely, the set
    \[
        \mathsf{NonUnique}^x_{\uparrow}
        :=
        \left\{y\in\R:
        \begin{array}{l}
            \text{there are at least two semi-infinite geodesics}\\
            \text{with direction $x$ ending at $(y,\ell)$ for some $\ell\ge1$}
        \end{array}
        \right\}
    \]
    is at most countably infinite.
\end{theorem}

\begin{proof}
    For $n,\ell\in\N$, let
    \[
        \mathcal N^x_{n,\ell}
        :=
        \left\{y\in[-n,n]:
        \begin{array}{l}
            \text{there are at least two semi-infinite geodesics}\\
            \text{with direction $x$ ending at $(y,\ell)$}
        \end{array}
        \right\}.
    \]
    Observe Since
    \[
        \mathsf{NonUnique}^x_{\uparrow}
        =
        \bigcup_{n,\ell\in\N}\mathcal N^x_{n,\ell}\,,
    \]
    it therefore suffices to prove that $\mathcal N^x_{n,\ell}$ is
    almost surely countable for every fixed $n,\ell\in\N$.

    By the coalescence of geodesics with fixed asymptotic direction,
    item $(3)$ of Lemma~\ref{lemma: semi-inf geod properties}, and
    the ordering of the extremal geodesics, all semi-infinite
    geodesics with direction $x$ ending at $(y,\ell)$,
    $y\in[-n,n]$, eventually coalesce. Indeed, it suffices to
    consider the leftmost geodesic ending at $(-n,\ell)$ and the
    rightmost geodesic ending at $(n,\ell)$.
    We may therefore choose a point $(z,m)$ on this common tail
    with $z<-n$ rational and $m\ge\ell+1$.
    Restricting the geodesics to the portion after $(z,m)$ gives
    \[
        \mathcal N^x_{n,\ell}
        \subseteq
        \left\{y\in[-n,n]:
        \begin{array}{l}
            \text{there are at least two geodesics}\\
            \text{from $(z,m)$ to $(y,\ell)$}
        \end{array}
        \right\}.
    \]

    Since each finite collection of Airy lines is locally
    absolutely continuous with respect to independent rate two
    Brownian motions see Section \ref{sec: preliminaries}, it suffices to prove that, for every
    $n,\ell,m\in\N$ with $m\ge\ell+1$ and every rational $z<-n$,
    the set
    \[
        \mathsf{NonUnique}^{z,n,m,\ell}_{\uparrow}
        :=
        \left\{y\in[-n,n]:
        \begin{array}{l}
            \text{there are at least two Brownian last passage}\\
            \text{geodesics from $(z,m)$ to $(y,\ell)$}
        \end{array}
        \right\}
    \]
    is almost surely countable.

    Let $B$ be an environment of independent rate two Brownian
    motions. By geodesic uniqueness for deterministic endpoints (see the remark after Lemma \ref{lemma: geod location brown env lpp}),
    we may work on the probability-one event on which, for all
    $m\ge k\ge1$ and all rational $z<q$, the geodesic from
    $(z,m)$ to $(q,k)$ is unique.
    We will proceed by induction on the relative depth
    $m-\ell\ge1$.

    \underline{Base case, $m=\ell+1$}:
    Write
    \[
        W(s):=B_{\ell+1}(s)-B_\ell(s),\qquad
        M(t):=\max_{z\le s\le t}W(s),\qquad
        \widetilde W(t):=M(t)-W(t).
    \]
    By independence, $W$ is a rate four Brownian motion.
    By the metric composition law,
    Lemma~\ref{lemma: Metric Composition}, the possible jump
    times from level $\ell+1$ to level $\ell$ are precisely in the set
    \[
        A(y):=\operatorname*{argmax}_{z\le s\le y}W(s),
    \]
    whose minimum and maximum are
    $\inf A(y):= Z^L_m((z,m),(y,\ell))$ and
    $\sup A(y):= Z^R_m((z,m),(y,\ell))$, respectively.
    Incidentally, by L\'evy's theorem,
    see \cite[Theorem~2.34]{morters_peres_2010},
    $\widetilde W(z+\cdot)$ has the law of the absolute value
    of a rate four Brownian motion started at zero.

    Observe that the set of maximisers $A(y)$ cannot contain two points $a<b<y$ as this would violate geodesic
    uniqueness from $(z,m)$ to $(q,\ell)$.
    Hence, if $A(y)$ is not a singleton, then $A(y)=\{a,y\}$ for some $a < y$. In this case,
    \[
        \widetilde W(a)=\widetilde W(y)=0,
        \qquad
        \widetilde W(t)>0\quad\text{for }a<t<y,
    \]
    so $(a,y)$ is an excursion of $\widetilde W$ away from zero.
    Conversely, if $(a,b)$ is such an excursion with $b<\infty$,
    then $A(b)=\{a,b\}$.
    Thus $\mathsf{NonUnique}^{z,n,m,\ell}_{\uparrow}$
    consists precisely of the finite right endpoints in
    $[-n,n]$ of excursions of $\widetilde W$ away from zero.
    There are at most countably many such excursions, since
    $\{t>z:\widetilde W(t)>0\}$ is open.
    This concludes the proof of the base case.

    \underline{Induction step, $m\ge\ell+2$}:
    For $k\in\{\ell,\ell+1\}$ and $s\ge z$, write
    \[
        G_k(s):=B[(z,m)\to(s,k)],
        \qquad
        H(s):=G_{\ell+1}(s)-B_\ell(s).
    \]
    By the metric composition law, Lemma \ref{lemma: Metric Composition}
    \[
        G_\ell(y)
        =
        \max_{z\le s\le y}
        G_{\ell+1}(s)+B_\ell(y)-B_\ell(s)
        =
        B_\ell(y)+\max_{z\le s\le y}H(s).
    \]
    Thus the possible jump times from level $\ell+1$ to
    level $\ell$ are precisely in the set
    \[
        B(y):= \operatorname*{argmax}_{z\le s\le y}H(s).
    \]
    Observe that with the gap process
    \[
        \widetilde H(y)
        :=
        \max_{z\le s\le y}H(s)-H(y)
        =
        G_\ell(y)-G_{\ell+1}(y)\,,
    \]
    the excursion argument from the base case applies to $H, \tilde{H}$. Observe for any $y\in [-n, n]$, the set $B(y)$ cannot contain two $ z \le a  < b< y$, as this would violate uniqueness. Hence, whenever the jump time at $y$ is not unique, the maximizers in $B(y)$ are exactly $\{a,y\}$, where $(a,y)$ is an
    excursion of $\widetilde H$ away from zero.
    In particular, the set
    \[
        \mathcal E
        :=
        \left\{y\in[-n,n]:
        Z^L_{\ell+1}((z,m),(y,\ell))
        <
        Z^R_{\ell+1}((z,m),(y,\ell))
        \right\}
    \]
    is countable, by continuity.

    Now fix $y\in[-n,n]\setminus\mathcal E$ and let
    \[
        s
        :=
        Z^L_{\ell+1}((z,m),(y,\ell))
        =
        Z^R_{\ell+1}((z,m),(y,\ell)).
    \]
    Suppose, $s<y$, then all geodesics from $(z,m)$ to $(y,\ell)$
    share the horizontal segment from $(s,\ell)$ to $(y,\ell)$.
    Choose a rational $q\in(s,y)$.
    Their restrictions to $(q,\ell)$ are geodesics from
    $(z,m)$ to $(q,\ell)$, and so coincide.
    Since their remaining portions are the same horizontal
    segment, the original geodesics are identical as well.

    If $s=y$, any two distinct geodesics
    from $(z,m)$ to $(y,\ell)$ restrict to two distinct
    geodesics from $(z,m)$ to $(y,\ell+1)$.
    We therefore have
    \[
        \mathsf{NonUnique}^{z,n,m,\ell}_{\uparrow}
        \setminus\mathcal E
        \subseteq
        \mathsf{NonUnique}^{z,n,m,\ell+1}_{\uparrow}.
    \]
    The set on the right is countable by the induction
    hypothesis, and $\mathcal E$ is countable.
    This concludes the inductive step and the proof.
\end{proof}

We now obtain a more refined description of the set of directions of geodesic coalescence.

\begin{theorem}\label{thm: geod no coal h dim}
    Set
    \[
    \mathsf{NoCoal}_{\nocoal}
    :=
    \bigcup_{\substack{t<t'\in\R\\m\le m'\in\N}}
    \left\{x>0:
    \begin{aligned}
        &\pi^L[x\to(t,m)] \text{ and } \pi^R[x\to(t',m')]\\
        &\text{do not coalesce}
    \end{aligned}
    \right\}.
    \]
    Then, there exists an event $\Omega_0$ of full probability
    on which $\mathsf{NoCoal}_{\nocoal}$ has Hausdorff dimension
    equal to $\frac12$. However,
    \[
    \mathsf{NoCoalTouch}_{\nocoaltouch}
    :=
    \bigcup_{\substack{t<t'\in\R\\m\le m'\in\N}}
    \left\{x>0:
    \begin{aligned}
        &\pi^L[x\to(t,m)] \text{ and } \pi^R[x\to(t',m')]\\
        &\text{do not coalesce and are not disjoint}
    \end{aligned}
    \right\}
    \]
    is at most countable on $\Omega_0$.
\end{theorem}
\begin{proof}
We start by computing the Hausdorff dimension of
$\mathsf{NoCoal}_{\nocoal}$.

\underline{Lower bound}: To establish the lower bound in the
Hausdorff dimension of $\mathsf{NoCoal}_{\nocoal}$, observe that
by Proposition \ref{prop: local variation geod disj}, almost
surely, for all $\ell \ge 1$,
\[
    \left\{x>0:
    \begin{aligned}
        &\pi^L[x\to(0,1)] \text{ and } \pi^R[x\to(0,\ell)]\\
        &\text{are disjoint}
    \end{aligned}
    \right\}
    =
    \mathsf{LV}\bigl((0,\infty),
    \Delta^{(0,1)}_{(0,\ell)}\bigr).
\]

This means $\mathsf{NoCoal}_{\nocoal}$ contains the set of
local variation of $\Delta^{(y_2,1)}_{(y_2,1)}$, $y_1<y_2$,
which by Corollary \ref{corollary: geod two end points disjoint}
has Hausdorff dimension equal to $1/2$.

\underline{Upper bound}: Observe that by geodesic monotonicity
and uniqueness (item $(2)$ in
Lemma \ref{lemma: semi-inf geod properties}), we have
\[
    \mathsf{NoCoal}_{\nocoal}
    \subseteq
    \bigcup_{n\in\N}
    \left\{x>0:
    \begin{aligned}
        &\pi^L[x\to(-n,1)] \text{ and } \pi^R[x\to(n,1)]\\
        &\text{do not coalesce}
    \end{aligned}
    \right\}.
\]
By Proposition \ref{prop: local variation geod disj},
$\mathsf{NoCoal}_{\nocoal}\setminus\mathsf{NoCoal}_{\nocoal}$ is thus contained in the union
of the sets of local variation of
$\Delta^{(-n,1)}_{(n,1)}$, $n\ge1$, which by
Corollary \ref{corollary: geod two end points disjoint}
have Hausdorff dimension equal to $1/2$. 

We now prove that the set of geodesic non-disjointness and no coalescence $\mathsf{NoCoalTouch}_{\nocoaltouch}$
is at most countable. This would then yield the upper bound on the Hausdorff dimension of $\mathsf{NoCoal}_{\nocoal}$.

By the approximation lemma,
Lemma \ref{lemma: geodesic approx} and geodesic monotonicity, we have
\[
    \mathsf{NoCoalTouch}_{\nocoaltouch}
    =
    \bigcup_{\substack{t<t'\in\Q\\m\le m'\in\N}}
    \left\{x>0:
    \begin{aligned}
        &\pi^L[x\to(t,m)] \text{ and } \pi^R[x\to(t',m')]\\
        &\text{do not coalesce and are not disjoint}
    \end{aligned}
    \right\}.
\]
We thus consider fixed $t<t'\in\Q$, $m\le m'\in\N$ in
the remainder of the proof and show that, almost surely,
\[
    \left\{x>0:
    \begin{aligned}
        &\pi^L[x\to(t,m)] \text{ and } \pi^R[x\to(t',m')]\\
        &\text{do not coalesce and are not disjoint}
    \end{aligned}
    \right\}
\]
is at most countable.

We start with the following observation. If
$\pi^L[x\to(t,m)]$ and $\pi^R[x\to(t',m')]$ do not
coalesce, then by geodesic uniqueness in the Airy line
ensemble for geodesics with rational endpoints (see the
remark after Lemma \ref{lemma: geod location brown env lpp}),
eventually the geodesics must become disjoint. Now, if they
were not entirely disjoint, then for every asymptotic direction $x$
in the set above, we can associate an open subset
$\mathcal{U}_x\subset(0,\infty)\times\R$ that is the space
between the geodesics $\pi^L[x\to(t,m)]$ and
$\pi^R[x\to(t',m')]$ after their last point of contact
(see Figure \ref{fig: geod coal not disj} for an illustration).

Moreover, we claim that for all $x\neq y$,
$\mathcal{U}_x\cap\mathcal{U}_y=\emptyset$. This would
give an injection of the set above into $\Q\cap (0, \infty)$,
concluding the proof of countability.

\begin{figure}[ht]
    \centering
    \includegraphics[width=0.65\linewidth]{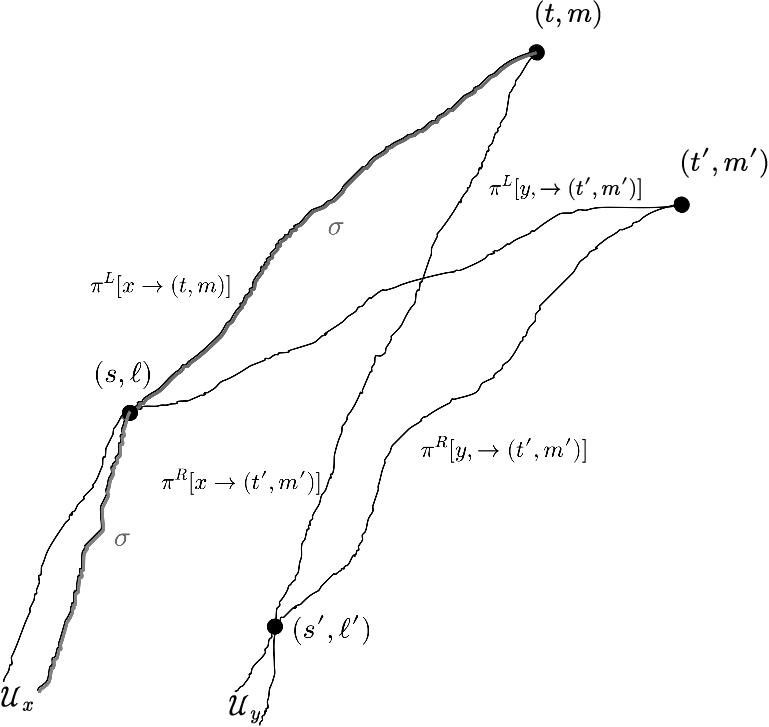}
    \caption{Illustration of non-coalescing geodesics
        which are not disjoint and the space between
        them given by open sets $\mathcal{U}_x$.}
    \label{fig: geod coal not disj}
\end{figure}

Indeed, let $x<y$. Then, observe that $\mathcal{U}_x$ is
bounded to the right by $\sigma$, which is a path given
by the concatenation of
$\pi^L[x\to(t,m)]|_{[s,t]}$ with
$\pi^R[x\to(t',m')]|_{(-\infty,s]}$, where $(s,\ell)$
is the last point of contact of the zig-zags of $\pi^L[x\to(t,m)]$ and
$\pi^R[x\to(t',m')]$
(see Figure \ref{fig: geod coal not disj}). Now, consider
the path $\pi^L[y\to(t,m)]$, which bounds $\mathcal{U}_y$
from the left. We claim that we have the ordering
$\sigma(z)\le\pi^L[y\to(t,m)](z)$ for all $z\le t$.

Now, by geodesic monotonicity, we have
\[
    \sigma|_{[s,t]}
    =
    \pi^L[x\to(t,m)]|_{[s,t]}
    \le
    \pi^L[y\to(t,m)]|_{[s,t]}
\]
pointwise. In particular,
$\sigma(s)=\pi^L[x\to(t,m)](s)
\le\pi^L[y\to(t,m)](s)$.
Moreover, by geodesic monotonicity
(by geodesic uniqueness, see the remark after Lemma \ref{lemma: geod location brown env lpp},
see also Figure \ref{fig: geod approx}), we have the
ordering (see Figure \ref{fig: geod coal not disj})
\begin{align*}
    \sigma|_{(-\infty,s]}
    &=
    \pi^R[x\to(t',m')]|_{(-\infty,s]}\\
    &\le
    \pi^R[x\to(s,\ell)]|_{(-\infty,s]}\\
    &\le
    \pi^L[y\to(s,\ell)]|_{(-\infty,s]}\\
    &\le
    \pi^L[y\to(t,m)]|_{(-\infty,s]}.
\end{align*}
\end{proof}

We consider for a fixed starting point, the set of directions for which there is geodesic non-uniqueness.

\begin{proposition}\label{prop: multiple geod fixed endpt}
    Fix $y\in \R$. There exists an event $\Omega_0$ of full probability such that on $\Omega_0$, the set of directions
    \[
    \mathsf{NonUnique}^y_{\downarrow}: = \{x> 0: \pi^L[x\to (y, 1)] \neq  \pi^R[x\to (y, 1)]\}
    \]
    is a dense countably infinite subset of $(0, \infty)$.
\end{proposition}
\begin{remark}
    The same holds replacing $(y, 1)$ with $(y, \ell)$, $\ell \ge 1$.
\end{remark}

\begin{proof}
    As usual work with event on which all geodesics with rational directions are unique and for all rational asymptotic directions any two semi-infinite geodesics sharing the former eventually coalesce. 
    
       We now show $\mathsf{NonUnique}^y_{\downarrow}$ is dense in $(0, \infty)$, hence at least countably infinite.

       Suppose for a contradiction that it is not dense. Let $0< u< v$ be such that for all $w\in [u, v]$ there is a unique $w$-directed semi-infinite geodesic, that is $\pi^L[w\to (y, 1)] = \pi^R[w\to (y, 1)]$. By monotonicity of geodesics, for $\ell \ge 2$ the geodesic jump times $Z^\square_\ell (w, y)$, $w\in [u, v]$ are monotone and by the second part of the approximation lemma, piecewise constant. By the supposed geodesic uniqueness on $[u, v]$ this means the jump times $Z^\square\ell(\cdot, y)$ coincide and so are continuous. This means that they must all coincide. Since $\ell$ was arbitrary, this means in fact that all geodesics $\pi[w\to (y, 1)]$ is a contradiction to the asymptotic parabolic trajectories of semi-infinite geodesics, in Definition \ref{def: semi-inf geo}.

       The countability of $\mathsf{NonUnique}^y_{\downarrow}$ follows by a similar argument using the monotonicity of the geodesic jump times above. 
\end{proof}

\section{Geometric properties of collections of geodesics}\label{sec: geod collections}

In this section, we study geometric and ergodic properties of collections of semi-infinite geodesics, Theorems \ref{thm: geodesic tree}, \ref{thm: geodesic gasket}, \ref{thm: ergodic geodesic tree} and \ref{thm: geodesic network haus dim}.

In particular, we consider \emph{collections} of geodesics as random objects in their own right. Fix $y \in \R$ and consider the union of all of the semi-infinite geodesics passing through $(y, 1)$. That is, we consider the `downward geodesic tree' (see Figure \ref{fig: geod down} for an illustration)
\begin{align}\label{eq: geod down tree}
    \mathcal{T}^y_\downarrow&:= \bigcup_{x> 0}\bigcup_{\square \in \{L, R\}}\mathsf{ZigZag}(\pi^{\square}[x\to (y, 1)])\\
    &\cup\{y\}\times [1, \infty)\cup (-\infty, y]\times \{1\}\subseteq (-\infty, y]\times[1, \infty)\,,
\end{align}
where we also include the `$0$' and `$\infty$'-directed geodesics with zig-zags $(-\infty, y]\times \{1\}$ and $\{0\}\times [1, \infty)$, respectively. 

We show that these downward geodesic collections admit a proper tree structure, compute their Hausdorff dimension and by their asymptotic parabolic shape, show that under a parabolic scaling, they fill the corner $(-\infty, y]\times[1, \infty)$.
\begin{figure}[ht]
    \centering
    \includegraphics[width=0.8\linewidth]{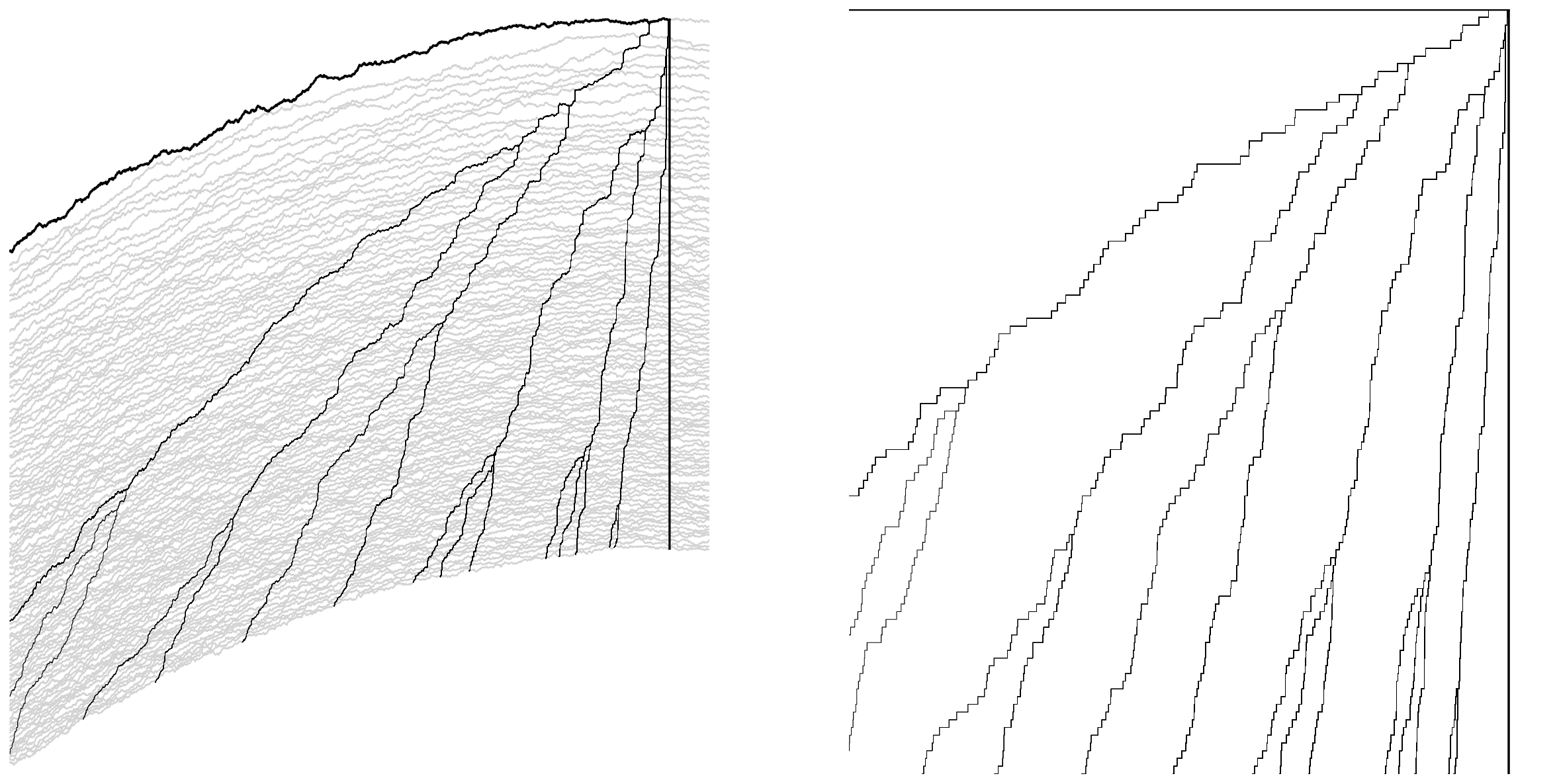}
    \caption{Illustration of a portion of $\mathcal{T}^y_\downarrow$ (excluding `$0$' and `$\infty$-directed geodesics) \textbf{left}: on the Airy line ensemble (in {\color{gray}gray}) and \textbf{right}: as zig-zag paths, for some fixed $y\in \R$ and $\theta$-directed geodesics with $\theta\ge 0$ in some bounded interval (bounded away from zero.}
    \label{fig: geod down}
\end{figure}

\begin{theorem}\label{thm: geodesic tree}
    There exists an event of full probability $\Omega_0$ such that on $\Omega_0$, the Hausdorff dimension of $\mathcal{T}^y_\downarrow$ is $1$ simultaneously for all $y\in \R$. 
    
    Moreover, for any fixed $y\in \R$, almost surely $\mathcal{T}^y_\downarrow$ contains no loops and there exists an at most countably infinite family $\Delta_n$, $n\ge 1$ of pairwise disjoint `wedge-like' open sets such that $\mathcal{T}^y_\downarrow\cup_n \Delta_n =  (-\infty, y]\times[1, \infty)$.

    Finally, under the scaling
    $\Psi_k: (x, y) \mapsto (1/k x, 1/k^2 y)$ $k\ge 1$, we have for any fixed $y\in \R$ almost surely that
    \[
    \bigcup_{k\ge 1}\Psi_k\circ \mathcal{T}^y_\downarrow
    \]
    is dense in $(-\infty, 0] \times [1, \infty)$. In particular, for any bounded and open subset $U\subset (-\infty, 0]\times [1, \infty)$ (with respect to the subspace topology), the Hausdorff distance of the closures 
    \[
    d_{\mathrm{Hausdorff}}(\overline{\cup_{1\le k \le M}\Psi_k\circ \mathcal{T}^y_\downarrow\cap U}, \overline{U}) \to 0
    \]
    as $M\to \infty$.
\end{theorem}

\begin{proof}
    Since horizontal segments of geodesics lie in $\cup_n(-\infty, y]\times \{n\}$, and so have Hausdorff dimension $1$, it suffices to compute the Hausdorff dimension of the vertical portions of the zig-zags of geodesics.
     
    Almost surely, for all $x\ge 0$, $y\in \R$ by Lemma \ref{lemma: geodesic approx}, the left-most and right most semi-infinite geodesics $\pi^L[x\to (y, \ell)], \pi^R[x\to (y, \ell)]$ from $x$ to $(y, \ell)$ can be expressed as the pointwise limits
    \[
     \pi^L[x\to (y, \ell)] = \sup_{\substack{n\in \N\,,\\\Q\ni z_n \nearrow x}}\pi^\square[z_n, (y, \ell)]\text{ and } \pi^R[x\to (y, \ell)] = \inf_{\substack{n\in \N\,,\\\Q\ni w_n \searrow x}}\pi^{\square}[w_n, (y, \ell)].
    \]
    Moreover, almost surely, for any $t\le y$, the restrictions $\pi^\square[z_n, (y, 1)]|_{[t, y]}$ and $\pi^\square[w_n, (y, 1)]|_{[t, y]}$ agree with $\pi^L[x\to (y, 1)]|_{[t, y]}$ and $\pi^R[x\to (y, 1)]|_{[t, y]}$ for all $n$ sufficiently large (depending on $t$). 

    This means almost surely, for all $y\in \R$ and any $x>0, \ell\ge 2$, there exists an open interval $\mathcal{U}_x$ in $[0, \infty)$ (in the subspace topology) such that the geodesic jump times in $\mathcal{T}^y_{\downarrow}\cap \mathcal{U}_x\times [\ell, \ell-1]$ coincide with those of $\pi^\square[x\to (y, 1)]$. This means, using a compactness argument and the fact that $\ell\ge 2$ is arbitrary that any fixed compact subset of the plane away from $\{y\}\times[1, \infty)$ contains finitely many vertical geodesic segments of $\mathcal{T}^y_{\downarrow}$.
    
     We are now in a position to compute the full Hausdorff dimension of $\mathcal{T}^y_\downarrow$. To do this we have by monotonicity of the Hausdorff dimension
    \[
    \mathrm{dim}_{\mathscr{H}}(\mathcal{T}^y_\downarrow) = \sup_{\ell\ge 2}\mathrm{dim}_{\mathscr{H}}(\mathcal{T}^y_\downarrow\cap \R\times [\ell, \ell-1])\,.
    \]
    We can drop the endpoints $\ell, \ell-1$ as they do not contribute to the Hausdorff dimension (they only account for horizontal segments of geodesics which take discrete values). 
    \[
    \mathrm{dim}_{\mathscr{H}}(\mathcal{T}^y_\downarrow) = \sup_{\ell\ge 2}\mathrm{dim}_{\mathscr{H}}(\mathcal{T}^y_\downarrow\cap \R\times (\ell, \ell-1))\,.
    \]
    We can now express $\mathcal{T}^y_\downarrow\cap \R\times (\ell, \ell-1) = A_\ell\times (\ell, \ell-1)$, where $A_\ell$ is the set of jump times $Z^\square_\ell(x, y)$, $x\ge 0$. Note by the above we have computed 
    $\mathrm{dim}_{\mathscr{H}}(A_\ell)=1$ for all $\ell \ge 2$ (we showed in fact $A_\ell$, $\ell\ge 2$ are all countable). 
    
    Thus, by standard arguments, see \cite{Hatano1971Notes} we can estimate
    \begin{align*}
    1=&\mathrm{dim}_{\mathscr{H}}(A_\ell)+ \mathrm{dim}_{\mathscr{H}}((\ell, \ell-1)) \le \mathrm{dim}_{\mathscr{H}}(\mathcal{T}^y_\downarrow\cap \R\times (\ell, \ell-1))\\
    \le &\min\{1+ \mathrm{dim}_{\mathscr{H}}(A_\ell), 1+ \mathrm{dim}_{\mathscr{H}}((\ell, \ell-1))\} =1
    \end{align*}
    concluding that the Hausdorff dimension of $\mathcal{T}^y_\uparrow$ is $1$.

    For the reminder, we fix $y\in \R$. 
    
    Observe that by geodesic uniqueness in the Airy line ensemble for finite horizon geodesics with fixed endpoints, almost surely $T^y_\uparrow$ cannot contain any loops (see Figure \ref{fig: geod approx}).
    
    Now, consider the set of directions $x> 0$ for which geodesics $\pi^L[x\to (y, 1)]$ and $\pi^R[x\to (y, 1)]$ are not identical. On $\Omega_0$, there are at most countably many such asymptotic directions $x> 0$ (by Proposition \ref{prop: multiple geod fixed endpt}). Now, let $(x_n)_n$ be such a sequence (if it exists). Let $(\Delta_n)_n$ denote the interiors between the zig-zags of $\pi^L[x_n\to (y, 1)]$ and $\pi^R[x_n\to (y, 1)]$ after their last point of contact. Moreover, observe that taking $x\to 0$, we have by monotonicity of jump times that $\pi^R[x\to (y, 1)]$ converges to a parabolic path with asymptotic direction $0$, namely $\pi^R[0
    \to (y, 1)]$. If this path is not identically equal to $\pi^L[0\to (y, 1)]\equiv 1$, then add $\Delta_0$ defined as above. 

    As $x\to \infty$, for all $\ell \ge 2$, we have that the geodesic jump times $Z_\ell(x, y)\to 0$. For a proof, see Proposition \ref{prop: non-degen asymptotic}. Moreover, by monotonicity of jump times and the above, it follows that any point $(x, z)\in (-\infty, y]\times [1, \infty)$ lies either on the zig-zag of some parabolic path or is in some $\Delta_n$, concluding the proof of this part.

    For the final part, fix any bounded and open set $U$ of $(-\infty, 0]\times [0, \infty)$ (with respect to to the subspace topology). By the inclusion for all $M\ge 1$, $\cup_{1\le k \le M}\Psi_k\circ \mathcal{T}^y_\downarrow\cap U\subset U$ it suffices to show for any $\varepsilon> 0$, the $\varepsilon$-neighbourhood of $\Psi_k\circ \mathcal{T}^y_\downarrow\cap U$ contains $U$. By compactness, for all $\varepsilon'>0$, the closure of $U$, $\overline{U}$ can be covered by the $\varepsilon'$-neighbourhoods of $(t_i, 2x_it_i^2)$ for finitely many $x_i, t_i$.

    Observe that the asymptotic parabolic shape of geodesics $\lim_{t\to -\infty} \pi^\square[x\to (y, 1)](t)/t^2 = 2x$ simultaneously for all $(x, y)\in [0, \infty) \times \R$ means that for any fixed $x\ge 0, y\in\R$, the rescaled graph of the geodesic $\pi^\square[x\to (y, 1)]$, $\Psi_k\circ \pi^\square[x\to (y, 1)](\cdot) = 1/k^2\cdot \pi^\square[x\to (y, 1)](k \cdot(\cdot))$ converges pointwise to $2xt^2$ as $k\to \infty$. Thus, taking $k\to \infty$ sufficiently large so that for all $x_i, t_i$, $\Psi_k\circ \pi^\square[x_i\to (y, 1)](t_i)$ are within $\varepsilon'/2$ of $(t_i, 2x_it_i^2)$, the $3/2\varepsilon'$-neighbourhood of $\Psi_k\circ \mathcal{T}^y_{\uparrow}$ contains $U$ for all $k$ sufficiently large. Since $\varepsilon'$ was arbitrary, this concludes the proof.
\end{proof}

As a direct consequence of the arguments in the proof of Theorem \ref{thm: geodesic tree}, we obtain the following structural result for semi-infinite geodesic jump times as a function of their asymptotic direction.

\begin{corollary}\label{cor: geod jump piecewise constant}
    As a consequence, for fixed $y\in \R$, we obtain that for all $\ell \ge 2$ the jump times $(0, \infty)\ni x\mapsto Z^{\square}_\ell(x, y)$ are piecewise constant. Since they cannot be globally constant as they converge to zero as $x\to \infty$, this means that there must exist asymptotic directions $x> 0$ such that geodesics $\pi^{\square}[x\to (y, 1)]$ are not identical. Moreover, since we can take $Z^{\square}_2(x, y)$ to be a.s. positive for all $x> 0$ on $\Omega_0$, we have the set of asymptotic directions for which geodesic non-uniqueness occurs is \emph{countably infinite}.
\end{corollary}

Theorem \ref{thm: geodesic tree} allows us to interpret the downward geodesic trees $\mathcal{T}^y_\downarrow$ as bona fide random objects taking values in some Polish space. In particular, one can view them as random closed sets of the plane, $\mathcal{F}(\R^2)$ equipped with the \emph{Fell} topology, see the Appendix, \ref{sec: polish meas} for more details on the topology and proof of measurability. We now show these geodesic trees emanating from a point are ergodic.

First recall some basic definitions in ergodic theory pertaining to continuous time continuous stochastic processes taking values in some Polish space $(X, \rho)$. Here, we identify continuous $\N\times \R $-indexed line ensembles as continuous stochastic processes on $\R^\N$ endowed with the topology of convergence on $\R^D\subset \R^\N$, $D\subset \N$ finite. In the remainder of this section, by a process $L$ we mean a continuous time process (i.e. measurable map from some sample space $\Omega \to (\mathcal{C}(\R; X), \mathcal{B})$, where  $\mathcal{B}$ is the Borel sigma algebra generated by the uniform-on-compacts topology) taking values in $(X, \rho)$ as above.

For $s\in \R$ define the \textbf{shift map} on paths
\[
    \theta_s: \sigma(\pi_t: t\in [s,\infty))\to \sigma(\pi_t: t\in [0,\infty))\,, \quad\mathcal{C}(\R; X)\ni A\mapsto \theta_s\circ \{\omega(s+\cdot): \omega \in A\} \,,
\]
where $\pi_t: \mathcal{C}([0,\infty);X)\to X$, $t\in \R$ are the coordinate projections. Note these shift maps also extend to measurable maps from path space in the natural way, i.e. for a measurable map $Z: \mathcal{C}(\R; X)\to \R$,
\[
\omega \mapsto \theta_s \circ Z (\omega) \equiv Z(\theta _s\circ \omega)\,, \quad s\in \R\,.
\]

\begin{definition}\label{def: ergodicity}
Suppose $L$ is a horizontal shift-invariant process on the probability space $ (\Omega, C , \PP ) $. We say $ A\in C$ is \textbf{shift-invariant} if $ \theta_s\circ A:=\{\theta_s \circ\omega\,| \omega \in A  \}  =A $ for all $ s \in\R $. Then $ L $ is \textbf{ergodic} if for all shift-invariant $ A $, $ \PP[A] =0 $ or 1.
\end{definition}

We record the following theorem which shows that the stationary Airy line ensemble is \textit{ergodic}.

\begin{theorem}(\cite[Theorem 1.6]{corwin2014ergodicityairylineensemble})\label{thm: erg airy line ensemble}
    The stationary Airy line ensemble is ergodic with respect to horizontal shifts. 
\end{theorem}

The ergodicity of the Airy line ensemble also translates to an ergodic statement on the semi-infinite geodesics therein, which is the content of the following theorem.

\begin{theorem}\label{thm: ergodic geodesic tree}
    For $n\ge 1$, consider the map $T_n :(x, y) \mapsto (x+n, y)$. Then, for any open subset $U$ with respect to the Fell topology of closed subsets of $\R^2$ containing an element of the topological support of $\mathcal{T}^0_\downarrow$, almost surely the shifted trees $T_{n}\circ\mathcal{T}^{-n}_\downarrow$, are infinitely often in $U$. 

    Moreover, for any $f: \mathcal{F}(\R^2) \to \R$ bounded measurable almost surely,
    \[
    \lim_{n\to \infty}\frac{1}{n}\sum_{i=1}^n f(\theta_{-n}\circ\mathcal{T}^{0}_{\downarrow}) = \int_{\mathcal{K}}f \diff\nu\,,
    \]
    where $\nu$ denotes the common law on $\mathcal{F}(\R^2)$ (on the associated Borel sigma algebra) of $\theta_{-n}\circ\mathcal{T}^{0}_{\downarrow}$, $n\ge 1$.
\end{theorem}

\begin{proof}
    Observe that by the approximation lemma, Lemma \ref{lemma: geodesic approx} all geodesics in the tree $\mathcal{T}^0_\downarrow$ can be obtained as pointwise limits of geodesics with rational asymptotic directions $\pi[q \to (0, 1)]$ for $q\ge 0$. By item (4) in Lemma \ref{lemma: semi-inf geod properties}, the laws of $\theta_{-n}\circ \mathcal{T}^0_\downarrow$ are therefore identical for all $n\ge 1$. In particular, again by item (4) of Lemma \ref{lemma: semi-inf geod properties}, we have almost surely for all $n\ge 1$, $\theta_{-n}\circ \mathcal{T}^0_\downarrow = \theta_{-n}\circ \mathcal{T}^0_\downarrow$ as elements of $\mathcal{F}(\R^2)$. Also observe the trees $\theta_{-n}\circ \mathcal{T}^0_\downarrow$ depend on the sigma algebra
    \[
    \sigma\bigg(\{\mathcal{A}_i(x): i\ge 1\,, x\le -n \}\bigg)\,.
    \]
    Thus, for $U$ as above, we have $\nu(U)> 0$. This means for all $n\ge 1$, 
    \[
    \PP(\theta_{-n}\circ\mathcal{T}^{0}_{\downarrow}\in U) = \nu(U)> 0
    \]
    and so
    \[
    \PP(\theta_{-n}\circ\mathcal{T}^{0}_{\downarrow}\in U\text{ infinitely often}) \ge \limsup_n \PP(\theta_{-n}\circ\mathcal{T}^{0}_{\downarrow}\in U) = \nu(U)> 0\,.
    \]
    Thus, the event $\theta_{-n}\circ\mathcal{T}^{0}_{\downarrow}\in U\text{ infinitely often}$ is shift invariant (see Definition \ref{def: ergodicity}). Using Theorem \ref{thm: erg airy line ensemble} we see that it must occur almost surely, being positive.
    
    The second part follows from an application of Birkhoff's ergodic theorem, treating for $f$ bounded measurable, $f(\mathcal{T}^{0}_{\downarrow})$ as measurable functionals of the Airy line ensemble (which can be done by the above discussion).
\end{proof}

We now also consider the `upward' tree starting from a fixed asymptotic direction $x\ge 0$. More precisely, for $x> 0$ fixed consider the union of all of the semi-infinite $x$-directed geodesics. That is consider the `upward geodesic collection' (see Figure \ref{fig: geod up} for an illustration)
\[
\mathcal{T}^x_\uparrow: = \bigcup_{y\in \R}\bigcup_{\square \in \{L, R\}}\mathsf{ZigZag}(\pi^{\square}[x\to (y, 1)])\subseteq (-\infty, y]\times[1, \infty)\,.
\]

\begin{figure}[ht]
    \centering
    \includegraphics[width=0.8\linewidth]{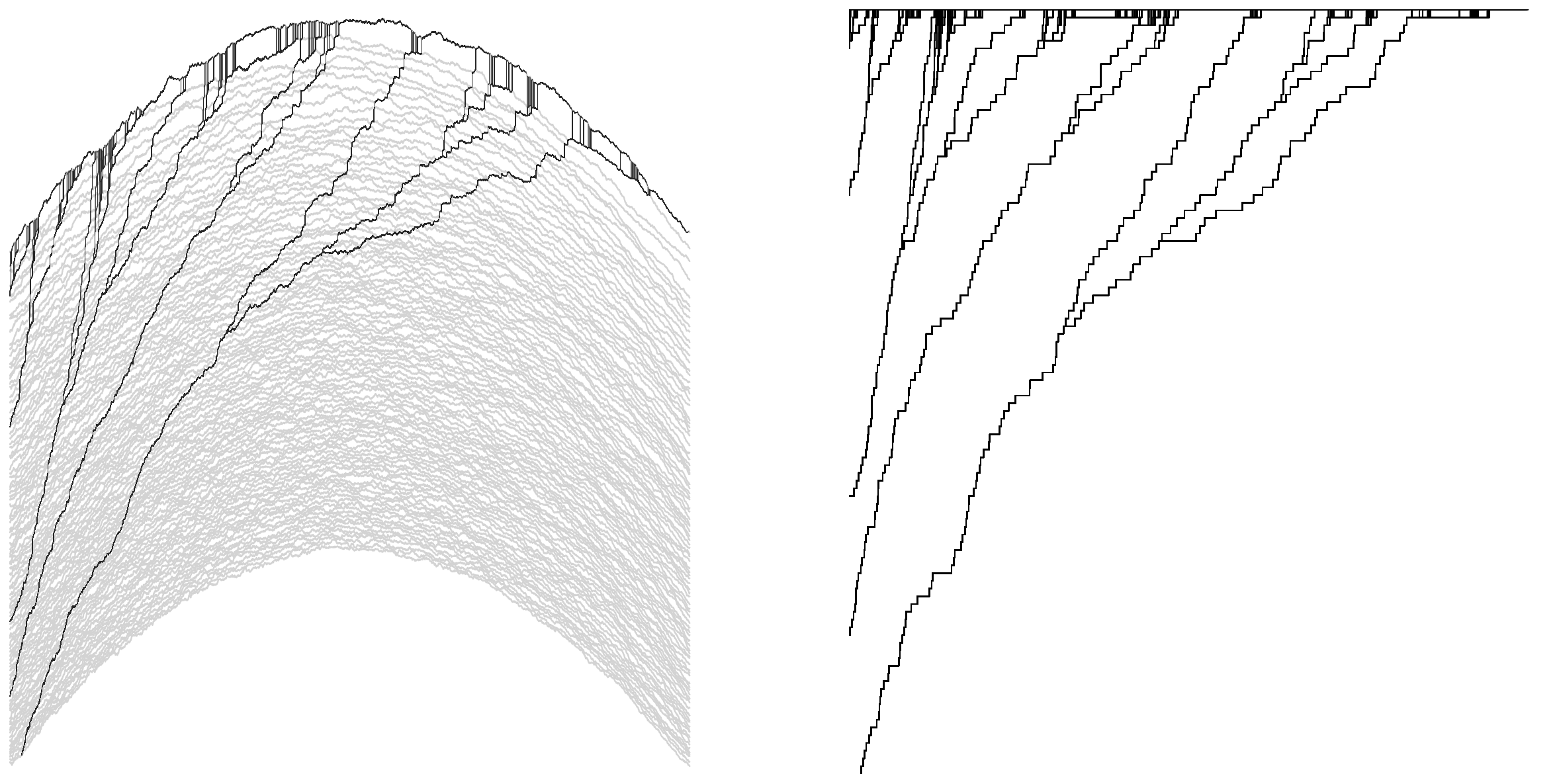}
    \caption{Illustration of a portion of $\mathcal{T}^x_\uparrow$, \textbf{left}: on the Airy line ensemble (in {\color{gray}gray}) and \textbf{right}: as zig-zag paths, some fixed $x> 0$ and geodesics with endpoints in some bounded interval. }
    \label{fig: geod up}
\end{figure}

We now show that the upward geodesic collections $\mathcal{T}^x_\uparrow$ for fixed $x> 0$ have a non-integer Hausdorff dimension, equal to three halves.

\begin{theorem}\label{thm: geodesic gasket}
    There exists an event of full probability $\Omega_0$ such that on $\Omega_0$, the Hausdorff dimension of $\mathcal{T}^x_\uparrow$ is $\frac{3}{2}$ simultaneously for all $x > 0$ . 
    
    Moreover, leftmost/rightmost collections of geodesics $\mathcal{T}^{x, \square}_{\uparrow}$, have a tree-like structure (they contain no loops). Thus, $\mathcal{T}^x_\uparrow$ can be constructed as a union of two trees.
\end{theorem}

\begin{proof}
Since the Hausdorff dimension of the horizontal segments of geodesics in $\mathcal{T}^x_\uparrow$ is $1$, we restrict our attention to computing the Hausdorff dimension of `horizontal fibres' $\mathcal{T}^x_\uparrow$.

We will work on the almost sure event of geodesic existence from item $(1)$ in lemma \ref{lemma: semi-inf geod properties} and the remark after Lemma \ref{lemma: geod location brown env lpp}.

    \underline{Upper bound}: 

    Fix $x> 0$. We will estimate the Hausdorff dimension of vertical segments of geodesics in $\mathcal{T}^x_\uparrow \cap [1, \infty) \times [1, 2]$, that is the collection of last jump times of semi-infinite geodesics in with asymptotic direction $x$ and endpoint $y$ at least $1$. Observe that for any such path, that is $\pi^\square[x\to (y, 1)]$, the restriction property of semi-infinite geodesics implies that $\pi^\square[x\to (y, 1)]|_{[0, y]}$ is a leftmost/rightmost geodesic in the Airy line ensemble from $(0, \pi^\square[x\to (y, 1)](0))$ to $(y, 1)$. By the metric composition law, Lemmas \ref{lemma: Metric Composition} and \ref{lemma: geod location brown env lpp} this means in particular that the last jump times $Z^\square_2(x, y)$ satisfy
    \[
    Z^\square_2(x, y) \in \operatorname*{argmax}_{0\le s\le y}\mathcal{A}[(0, \pi^\square[x\to (y, 1)](0))\to (s, 2)]+\mathcal{A}_1(y)-\mathcal{A}_1(s)
    \]
    or, equivalently
    \[
    Z^\square_2(x, y) \in \operatorname*{argmax}_{0\le s\le y}\mathcal{A}[(0, \pi^\square[x\to (y, 1)](0))\to (s, 2)]-\mathcal{A}_1(s)\,.
    \]
    To establish the upper bound on the Hausdorff dimension of $\mathcal{T}^x_\uparrow \cap [1, \infty) \times [1, 2]$, it thus suffices to show almost surely for al $\ell \in \N$, the Hausdorff dimension
    \[
    \bigg\{\operatorname*{argmax}_{0\le s\le y}\mathcal{A}[(0, \ell)\to (s, 2)]-\mathcal{A}_1(s): y\ge 1\bigg\}
    \]
    is at most $1/2$. By monotonicity of the Hausdorff dimension this is equivalent to showing the Hausdorff dimension of
    \[
    \bigg\{\operatorname*{argmax}_{0\le s\le y}\mathcal{A}[(0, \ell)\to (s, 2)]-\mathcal{A}_1(s): y\in [n, n+1]\bigg\}
    \]
    is at most $1/2$ almost surely for all $\ell, n\in \N$.

    Indeed, by the Brownian Gibbs property of the Airy line ensemble (recall from subsection \ref{subsec: Airy line ensemble}), the first $\ell$ levels of the Airy line ensemble are absolutely continuous with respect to $\ell$ independent rate two Brownian motions on $[n, n+1]$. This means it suffices to show for an environment of independent Brownian motions $B$
    \[
    \bigg\{\operatorname*{argmax}_{0\le s\le y}B[(0, \ell)\to (s, 2)]-B_1(s): y\in [n, n+1]\bigg\}
    \]
    is at most $1/2$ almost surely for all $\ell, n\in \N$. By independence of the lines in the Brownian environment and the distributional equality of $B[(0, \ell)\to (\cdot, 2)]$ to an $\ell-1$-level Dyson Brownian motion \cite{o2002representation}, which is mutually absolutely continuous with respect to rate two Brownian motion on compacts away from the origin (recall from subsection \ref{subsec: melons}), it suffices to show the Hausdorff dimensions of 
    \[
    \bigg\{\operatorname*{argmax}_{0\le s\le y}W(s): y\in [n, n+1]\bigg\}
    \]
    are at most $1/2$ almost surely for all $\ell, n\in \N$ where $W$ is now a rate four Brownian motion. Since any such time is a record time, it suffices to bound the Hausdorff dimension of record times 
    \[
    \bigg\{s\in [0, n+1]: W(s) = \max_{0\le r\le s}W(r)\bigg\}\,.
    \]
    By L\'evy's theorem for reflected Brownian motion, see \cite[Theorem~ 2.34]{morters_peres_2010}, $\max_{0\le r\le s}W(r)-W(s)$ has the same law on paths on $\mathcal{C}_0 [0, \infty)$ as $|W(\cdot)|$. Hence, since the Hausdorff dimension of the zeros of Brownian motion has Hausdorff dimension at most $1/2$ on any bounded interval, we conclude the about the record times, thereby establishing the desired upper bounds on the Hausdorff dimension. 

    To conclude the proof of the first part, observe one can by similar arguments consider $\mathcal{T}^x_\uparrow \cap [y, \infty) \times [1, 2]$ almost surely simultaneously for all $y\in \R$ by repeating the argument replacing the origin with any rational $r< y$. The argument for $\mathcal{T}^x_\uparrow \cap [y, \infty) \times [\ell, \ell-1]$ proceeds entirely analogously, as one is reduced to considering the last jump times of geodesics ending at level $\ell-1$. 

    \underline{Lower bound}: To establish the lower bound, we need a slight refinement of Proposition \ref{prop: semi-inf geod jump} for the jump times $Z^\square_2(x, y)$. 

    Proposition \ref{prop: semi-inf geod jump} gives almost surely for any fixed $x> 0, y\in \R$
    \[
    Z^\square_2(x, y) \in  \displaystyle\operatorname*{argmax}_{s\in (-\infty, y]}\bigg(\Delta_{(s, 1)}^{(0, 1)}(x)+\mathcal{A}_1(y)-\mathcal{A}_1(s)\bigg)\,.
    \]
    We will show this to be simultaneous for all $x> 0, y\in \R$.
    
    Lemma 2.24 in \cite{dauvergne2022disjointoptimizersdirectedlandscape} gives almost surely simultaneously for all $x > 0$ and $y< y'\in \R$ and any parabolic path $\pi:(-\infty, y]\to \N$ with $\lim_{z\to -\infty}\pi(z)/z^2 = 2x$,
    \[
    \lim_{z\to -\infty}\mathcal{A}[(z, \pi(z))\to (y, 1)]-\mathcal{A}[(z, \pi(z))\to (y', 1)] = \mathcal{S}(x, y)-\mathcal{S}(x,y')\,.
    \]

    Now, taking $\pi$ to be leftmost/rightmost geodesics $\pi^\square[x\to (y, 1)]$, we have by the metric composition law for all $z\le Z^\square_2(x, y)$
    \[
    Z^\square_2(x, y) \in  \displaystyle\operatorname*{argmax}_{s\in [z, y]}\bigg(\mathcal{A}[(z, \pi(z))\to (s, 1)]+\mathcal{A}_1(y)-\mathcal{A}_1(s)\bigg)
    \]
    or equivalently, for any fixed $\alpha < Z^\square_2(x, y)$, 
    \[
    Z^\square_2(x, y) \in  \displaystyle\operatorname*{argmax}_{s\in [\alpha, y]}\bigg(\mathcal{A}[(z, \pi(z))\to (s, 1)]-\mathcal{A}[(z, \pi(z))\to (y, 1)]-\mathcal{A}_1(s)\bigg)\,.
    \]
    By the preceding discussion we have the pointwise convergence for $\mathcal{A}[(z, \pi(z))\to (s, 1)]-\mathcal{A}[(z, \pi(z))\to (y, 1)]\to \mathcal{S}(x, s)-\mathcal{S}(x, y)$ as $z\to -\infty$. This means taking $z\to -\infty$,
    \[
    Z^\square_2(x, y) \in \displaystyle\operatorname*{argmax}_{s\in [\alpha, y]}\bigg(\mathcal{S}(x, s)-\mathcal{A}_1(s)\bigg)\,,
    \]
    or 
    \[
    Z^\square_2(x, y) \in \displaystyle\operatorname*{argmax}_{s\in [\alpha, y]}\bigg(\mathcal{S}(x, s)-\mathcal{S}(0, s)\bigg)\,.
    \]
    Since $\alpha<Z^\square_2(x, y)$ was arbitrary we finally have the first jump times $Z^\square_2(x,y)$ are the locations of the maxima of 
    \[
    s\mapsto \mathcal{S}(x, s)-\mathcal{S}(0,s) \text{ on } (-\infty, y].
    \]
    
    By shape estimates for Airy sheet \eqref{eq: airy shape bnds}, have that as $y\to \infty$, the above maximum diverges so arguing as in the proof of the lower bound in Theorem \ref{thm: airy sheet monotone haus dim} (that is by a `mass distribution' argument) we obtain the set of local variation of 
    \[
    s\mapsto \max_{0\infty< r\le s}\mathcal{S}(x, r)-\mathcal{S}(0,r) = \mathcal{S}(x, s)-\mathcal{S}(0,s)
    \]
    has Hausdorff dimension $1/2$. Since on this set, we have that the argmax must be unique, we deduce the set $ \{Z^\square_2(x, y): y\in \R\}$ must contain this set of local variation, which has Hausdorff dimension at least $1/2$.

    We have thus established that the Hausdorff dimension of first jump times of of geodesics in $\mathcal{T}^x_\uparrow$ is $1/2$ simultaneously for all $x>0$. 
    
    We now compute the full Hausdorff dimension of $\mathcal{T}^x_\uparrow$. To do this we have
    \[
    \mathrm{dim}_{\mathscr{H}}(\mathcal{T}^x_\uparrow) = \sup_{\ell\ge 2}\mathrm{dim}_{\mathscr{H}}(\mathcal{T}^x_\uparrow\cap \R\times [\ell, \ell-1])\,.
    \]
    We can drop the endpoints $\ell, \ell-1$ as they do not contribute to the Hausdorff dimension (they only account for horizontal segments of geodesics which take discrete values). 
    \[
    \mathrm{dim}_{\mathscr{H}}(\mathcal{T}^x_\uparrow) = \sup_{\ell\ge 2}\mathrm{dim}_{\mathscr{H}}(\mathcal{T}^x_\uparrow\cap \R\times (\ell, \ell-1))\,.
    \]
    We can now express $\mathcal{T}^x_\uparrow\cap \R\times (\ell, \ell-1) = A_\ell\times (\ell, \ell-1)$, where $A_\ell$ is the set of jump times $Z^\square_\ell(x, y)$, $y\in \R$. Note by the above we have computed 
    \[
    \mathrm{dim}_{\mathscr{H}}(A_\ell)\le \frac{1}{2}\text{ and } \mathrm{dim}_{\mathscr{H}}(A_2) = \frac{1}{2}\,.
    \]
    Thus, by standard arguments, see \cite{Hatano1971Notes} we can estimate
    \begin{align*}
    &\mathrm{dim}_{\mathscr{H}}(A_\ell)+ \mathrm{dim}_{\mathscr{H}}((\ell, \ell-1))(= \frac{3}{2} \text{ when }\ell = 2)\le \mathrm{dim}_{\mathscr{H}}(\mathcal{T}^x_\uparrow\cap \R\times (\ell, \ell-1))\\
    \le &\min\{1+ \mathrm{dim}_{\mathscr{H}}(A_\ell), 1+ \mathrm{dim}_{\mathscr{H}}((\ell, \ell-1))\} (= \frac{3}{2} \text{ when }\ell = 2)\,.
    \end{align*}
    
    For the last part, this follows from the fact that two geodesics $\pi^\square[x\to (y_1, 1)]$ and $\pi^\square[x\to (y_2, 1)]$ with $y_1< y_2$ $\square\in \{L, R\}$ are either entirely disjoint of coalesce for an interval and then become separate. If there were any `loops', the the uniqueness of leftmost/rightmost geodesics (item $(1)$ in Lemma \ref{lemma: semi-inf geod properties}) would be violated.
\end{proof}

We can now consider the entire network of semi-infinite geodesics in the Airy line ensemble and study its geometric properties. In particular, we compute its Hausdorff dimension.

\begin{theorem}\label{thm: geodesic network haus dim}
    Almost surely, the Hausdorff dimension of the \emph{full semi-infinite geodesic network} \emph{excluding} extremal ` $0$' and `$\infty$'-directed geodesics
    \[
    \mathcal{T}: = \bigcup_{(x,y)\in [0, \infty)\times \R}\mathcal{T}^x_\uparrow\cup \bigg(\mathcal{T}^y_\downarrow\setminus \big(\{y\}\times [1, \infty)\cup (-\infty, y]\times \{1\}\big) \bigg)  \text{ is } \frac{3}{2}\,.
    \]
\end{theorem}

\begin{proof}
    The lower bound of $3/2$ is immediate from Theorem \ref{thm: geodesic gasket}. 

    To prove the upper bound, one can proceed in the same way as for the proof of the upper bound in Theorem \ref{thm: geodesic gasket} to estimate the Hausdorff dimensions of 
    \[
    \bigcup_{(x,y)\in [0, \infty)\times \R}\mathcal{T}^x_\uparrow\cup \mathcal{T}^y_\downarrow \cap [n, n+1]\times [1, \infty)\,, n\in \Z\,.
    \]
\end{proof}

\section{Appendix}

\section{A Polish space for the embedded downward geodesic trees}\label{sec: polish meas}

In this section we briefly describe a natural topology on the space of closed subsets of the plane generated by a complete metric, allowing us to treat downward collections of geodesics as bona fide random objects. We start with a definition.
\begin{definition}
Let $\F(\R^2)$ be the family of all closed subsets of $\R^2$, including the empty set. The topology generated by
\[
 \{F:F\cap U\ne\varnothing\}\quad(U\subseteq\R^2\text{ open})\,,
 \qquad
 \{F:F\cap K=\varnothing\}\quad(K\subseteq\R^2\text{ compact})\,,
\]
is called the Fell topology.
\end{definition}

We also collect the fact that the Fell topology is generated by a complete metric, meaning it can be turned into a Polish space, making $\mathcal{F}(\R^2)$ amenable to probabilistic methods.

\begin{proposition}\label{prop: fell}
The Fell topology is generated by a complete metric $d_F$. Hence, $(\mathcal{F}(\R^2), d_F)$ is a Polish space. Moreover, its Borel sigma-field is generated by the open-set hitting events $\{F:F\cap U\ne\varnothing\}$.
\end{proposition}

An explicit construction of such a complete metric $d_F$ is available, by identifying the plane with its one point compactification, namely the (unit) sphere $S^2$ in $\R^3$. Let $N=(0,0,1)$, and setting
\[
 \sigma(u,v)=\frac{(2u,2v,u^2+v^2-1)}{1+u^2+v^2}.
\]
Thus the map $\sigma$ identifies the one-point compactification of $\R^2$ with $S^2$, with `infinity' mapped to $N$. In fact it is a homeomorphism when both spaces are equipped with their natural topologies. For $F\in\F(\R^2)$ put
\[
 F^*=\sigma(F)\cup\{N\},\qquad
 d_F(F,G)=d_H^{S^2}(F^*,G^*),
\]
where $d_H^{S^2}$ denotes the Hausdorff distance on $S^2$, viewed as a metric space with metric inherited from the Euclidean metric on $\R^3$. Note this constructions is analogous to the hypograph topology considered in \cite{VW25}.

See \cite[Chapter 5]{beer1993topologies} for more details about this topology.

Let $(\Omega,\mathscr F,\mathbb P)$ be some complete probability space carrying the Airy line ensemble as a measurable $\mathcal{C}^\N$-valued variable. Fix a deterministic $y\in\R$. 

\begin{lemma}[Measurable jump times]\label{lemma: jump meas}
For all $j\ge 2$, $q\in (\Q\cap(0, \infty))$ and fixed $y\in \R$, the semi-infinite geodesic jump times $Z_j(q,y)$ are random variables.
\end{lemma}
\begin{proof}
Fix $j\ge 2, q\in(\Q\cap(0, \infty))$. Recall from Lemma \ref{lemma: semi-inf geod properties} the existence of the almost surely unique $q$-directed semi-infinite geodesic $\pi[q\to(y,1)]$ with right endpoint $(y, 1)$. Lemma \ref{lemma: semi-inf geod properties} also implies that the semi-infinite jump times $Z_j(q,y)$ are almost sure limits of the finite-horizon jump times from level $j$ to $j-1$ of geodesics from $(-\sqrt{q/(2k)}, k)$ to $(y, 1)$ as $k\to \infty$. It thus suffices to that the jump times of geodesics over the Airy line ensemble from $(a, k)$ to $(b, 1)$ for $a < b$ finite and $k\ge 1$, are random variables. 

Fix $a< b$ finite and $k\ge 1$. Since $(\Omega,\mathscr F,\mathbb P)$ is assumed complete, by the remark after \ref{lemma: geod location brown env lpp} there exists an event $\Omega_0\in \mathscr{F}$ of full probability such that there exists a unique geodesic from $(a, k)$ to $(b, 1)$. By Lemma \ref{lemma: geod location brown env lpp}, the jump times are the now unique locations of maximisers over last passage values over the Airy line ensemble, which are random variables, concluding the proof.
\end{proof}

Recall from \eqref{eq: geod down tree}, the construction of the geodesic tree
\[
\mathcal{T}^y_\downarrow: = \bigcup_{x> 0}\bigcup_{\square \in \{L, R\}}\mathsf{ZigZag}(\pi^{\square}[x\to (y, 1)])\cup\{y\}\times [1, \infty)\cup (-\infty, y]\times \{1\}\subseteq (-\infty, y]\times[1, \infty)\,.
\]
We will now consider a 'dense subtree' thereof, namely
\[
\mathcal{T}^{y, \Q}_\downarrow: = \bigcup_{q \in (\Q\cap(0, \infty))}\mathsf{ZigZag}(\pi[q\to (y, 1)])\cup\{y\}\times [1, \infty)\cup (-\infty, y]\times \{1\}\subseteq (-\infty, y]\times[1, \infty)\,,
\]
and its closure in the Euclidean topology of $\R^2$,
\[
\widehat{\mathcal{T}}^{y, \Q}_\downarrow: = cl_{\R^2}\left(\mathcal{T}^{y, \Q}_\downarrow\right)\,.
\]
We will prove that $\widehat{\mathcal{T}}^{y, \Q}_\downarrow$ is a $(\mathcal{F}(\R^2), \mathcal{B}(d_F))$-valued random variable that up to null sets coincides with $\mathcal{T}^y_\downarrow$.

\begin{theorem}[Measurability of the downward geodesic tree]\label{thm: meas}
On the complete probability space $(\Omega, \mathscr{F}, \PP)$, the map
\[
 \omega\longmapsto \mathcal{T}^{y}_\downarrow(\omega)
\]
is measurable from $(\Omega,\mathscr F)$ to $(\F(\R^2),\mathcal B(d_F))$.
\end{theorem}
\begin{proof}

Fix $y\in \R$. We first show the closure of the dense subtree of $\mathcal{T}^y_\downarrow$, $\widehat{\mathcal{T}}^{y, \Q}_\downarrow$ is measurable from $(\Omega,\mathscr F)$ to $(\F(\R^2),\mathcal B(d_F))$.

Observe that for every open $U\subseteq\R^2$ and $T\subseteq \R^2$, the event $T$ hits $U$ is the same as the event $cl_{\R^2}(T)$ hits $U$. Hence
\begin{align}\label{eq: hit}
 \{\widehat{\mathcal{T}}^{y, \Q}_\downarrow\cap U\ne\varnothing\}=\{\mathcal{T}^{y, \Q}_\downarrow\cap U\ne\varnothing\}
 &=\{[1, \infty)\cup (-\infty, y]\times \{1\}\cap U\ne\varnothing\}\nonumber\\
 &\cup\bigcup_{q\in(\Q\cap(0, \infty))}\{\mathsf{ZigZag}(\pi[q\to (y, 1)])\cap U\ne\varnothing\}\,.
\end{align}
By the measurability of the jump times of the geodesic $s\pi[q\to (y, 1)]$, Lemma \ref{lemma: jump meas}, every event on the right is measurable. Thus by Proposition, \ref{prop: fell}, $\widehat{\mathcal{T}}^{y, \Q}_\downarrow$ is measurable from $(\Omega,\mathscr F)$ to $(\F(\R^2),\mathcal B(d_F))$.

We now show that $\widehat{\mathcal{T}}^{y, \Q}_\downarrow$ and $\mathcal{T}^{y}_\downarrow$ coincide on an event of full probability. Henceforth, we will work on the event of full probability which guarantees the existence and uniqueness of semi-infinite geodesics in the Airy line with rational asymptotic directions and right endpoint $y$. 

Theorem \ref{thm: geodesic tree} gives $\mathcal{T}^{y}_\downarrow$ is a closed subset of $\R^2$. This means in particular we have the inclusion $\widehat{\mathcal{T}}^{y, \Q}_\downarrow\subseteq \mathcal{T}^{y}_\downarrow$, since $\mathcal{T}^{y}_\downarrow$ contains all geodesic with rational asymptotic directions (as well as the two extremal, that is zero and $\infty$-directed geodesics). To prove the inclusion $\mathcal{T}^{y}_\downarrow\subseteq  \widehat{\mathcal{T}}^{y, \Q}_\downarrow$, observe that by the approximation lemma, Lemma \ref{lemma: geodesic approx}, a point in the closure of $\mathcal{T}^{y, \Q}_\downarrow$ is either in $\{[1, \infty)\cup (-\infty, y]\times \{1\}$, or is a limit of points on the zig-zags of geodesic with rational asymptotic directions. In the latter case, this means the aforementioned points lie up to some fixed depth. By passing to a monotone subsequence of such asymptotic directions, the above lemma implies that the point lies on some $\pi^\square[x\to (y, 1)]$, $x > 0$, concluding the proof of equality thus establishing that on a set of full probability (that is the set of geodesic uniqueness in the Airy line ensemble with rational endpoints also including $y\in \R$) we have, $\mathcal{T}^{y}_\downarrow = \widehat{\mathcal{T}}^{y, \Q}_\downarrow$ is measurable from $(\Omega,\mathscr F)$ to $(\F(\R^2),\mathcal B(d_F))$.
\end{proof}

\color{black}

\bibliographystyle{alpha}
\bibliography{refs}

@article{DOV,
author = {Duncan {D}auvergne and {J}anosch {O}rtmann and {B}{\'a}lint {V}ir{\'a}g},
title = {{The directed landscape}},
volume = {229},
journal = {Acta {M}athematica},
number = {2},
publisher = {Institut {M}ittag-{L}effler},
pages = {201 -- 285},
year = {2022},
doi = {10.4310/ACTA.2022.v229.n2.a1},
URL = {https://doi.org/10.4310/ACTA.2022.v229.n2.a1}
}

@article{sarkar2021brownian,
  author = {Sourav Sarkar and B{\'a}lint Vir{\'a}g},
title = {{{B}rownian absolute continuity of the {{KPZ}} fixed point with arbitrary initial condition}},
volume = {49},
journal = {The Annals of Probability},
number = {4},
publisher = {Institute of Mathematical Statistics},
pages = {1718 -- 1737},
year = {2021},
doi = {10.1214/20-AOP1491},
URL = {https://doi.org/10.1214/20-AOP1491}
}

@article{kardar1986dynamic,
  title={Dynamic scaling of growing interfaces},
  author={Kardar, Mehran and Parisi, Giorgio and Zhang, Yi-Cheng},
  journal={Physical Review Letters},
  volume={56},
  number={9},
  pages={889},
  year={1986},
  publisher={APS}
}

@article{corwin2014brownian,
  title={{B}rownian {G}ibbs property for {A}iry line ensembles},
  author={Corwin, Ivan and Hammond, Alan},
  journal={Inventiones Mathematicae},
  volume={195},
  pages={441--508},
  year={2014},
  publisher={Springer}
}

@inproceedings{hammond2019patchwork,
  title={A patchwork quilt sewn from {B}rownian fabric: regularity of polymer weight profiles in {B}rownian last passage percolation},
  author={Hammond, Alan},
  booktitle={Forum of Mathematics, Pi},
  volume={7},
  pages={e2},
  year={2019},
  organization={Cambridge University Press}
}

@misc{basu2018coalescencegeodesicsexactlysolvable,
      title={Coalescence of geodesics in exactly solvable models of last passage percolation}, 
      author={Riddhipratim Basu and Sourav Sarkar and Allan Sly},
      year={2018},
      eprint={1704.05219},
      archivePrefix={Preprint, available at arXiv},
      primaryClass={math.PR},
      howpublished = {Preprint, available at arXiv:1704.05219},
      url={https://arxiv.org/abs/1704.05219}, 
}

@article{Dauvergne_2021,
   title={Bulk properties of the {A}iry line ensemble},
   volume={49},
   ISSN={0091-1798},
   url={http://dx.doi.org/10.1214/20-AOP1492},
   DOI={10.1214/20-aop1492},
   number={4},
   journal={The Annals of Probability},
   publisher={Institute of Mathematical Statistics},
   author={Dauvergne, Duncan and Virág, Bálint},
   year={2021},
   month=may }

@article{o2002representation,
author = {Neil O'Connell and Marc Yor},
title = {{A representation for non-colliding random walks}},
volume = {7},
journal = {Electronic Communications in Probability},
number = {none},
publisher = {Institute of Mathematical Statistics and Bernoulli Society},
pages = {1 -- 12},
year = {2002},
doi = {10.1214/ECP.v7-1042},
URL = {https://doi.org/10.1214/ECP.v7-1042}
}

@misc{corwin2014ergodicityairylineensemble,
      title={Ergodicity of the {A}iry line ensemble}, 
      author={Ivan Corwin and Xin Sun},
      year={2014},
      eprint={1405.0464},
      archivePrefix={Preprint, available at arXiv:1405.0464},
      primaryClass={math.PR},
      url={https://arxiv.org/abs/1405.0464}, 
}

@misc{wu2025applicationsoptimaltransportdyson,
      title={Applications of optimal transport to {D}yson {B}rownian Motions and beyond}, 
      author={Xuan Wu},
      year={2025},
      howpublished={Preprint, available at arXiv:2412.17389},
}

@article{Hammond_2019,
   title={Exponents governing the rarity of disjoint polymers in {B}rownian last passage percolation},
   volume={120},
   ISSN={1460-244X},
   url={http://dx.doi.org/10.1112/plms.12292},
   DOI={10.1112/plms.12292},
   number={3},
   journal={Proceedings of the London Mathematical Society},
   publisher={Wiley},
   author={Hammond, Alan},
   year={2019},
   month=sep, pages={370–433} }

@article{ganguly2021random,
  author    = {Shirshendu Ganguly},
  title     = {Random Metric Geometries on the Plane and {K}ardar-{P}arisi-{Z}hang Universality},
  journal   = {Preprint, available at arXiv:2110.11287},
  year      = {2021}
}

@article{prahofer2002scale,
  author    = {Michael Pr{\"a}hofer and Herbert Spohn},
  title     = {Scale Invariance of the {PNG} Droplet and the {A}iry Process},
  journal   = {Journal of Statistical Physics},
  volume    = {108},
  number    = {5--6},
  pages     = {1071--1106},
  year      = {2002}
}

@article{corwin2014ergodicity,
  author    = {Ivan Corwin and Xin Sun},
  title     = {Ergodicity of the {A}iry Line Ensemble},
  journal   = {Electronic Communications in Probability},
  volume    = {19},
  pages     = {1--11},
  year      = {2014}
}

@article{Pimentel2016,
  author    = {L. P. R. Pimentel},
  title     = {Duality between coalescence times and exit points in last-passage percolation models},
  journal   = {Annals of Probability},
  volume    = {44},
  number    = {3},
  pages     = {1925--1983},
  year      = {2016}
}

@article{SeppalainenShen2020,
  author    = {T. Sepp{\"a}l{\"a}inen and X. Shen},
  title     = {Coalescence estimates for the corner growth model with exponential weights},
  journal   = {Electronic Journal of Probability},
  volume    = {25},
  number    = {85},
  pages     = {1--31},
  year      = {2020}
}

@article{Zhang2020,
  author    = {L. Zhang},
  title     = {Optimal exponent for coalescence of finite geodesics in exponential last passage percolation},
  journal   = {Electronic Communications in Probability},
  volume    = {25},
  pages     = {1--14},
  year      = {2020}
}

@article{BusaniFerrari2022,
  author    = {O. Busani and P. L. Ferrari},
  title     = {Universality of the geodesic tree in last passage percolation},
  journal   = {Annals of Probability},
  volume    = {50},
  number    = {1},
  pages     = {90--130},
  year      = {2022}
}

@article{RassoulAghaSeppalainenShen2023,
  author    = {F. Rassoul-Agha and T. Sepp{\"a}l{\"a}inen and X. Shen},
  title     = {Coalescence and total-variation distance of semi-infinite inverse-gamma polymers},
  journal   = {Preprint, available at arXiv:2305.10265},
  year      = {2023}
}

@article{GangulyHegdeZhang2023,
  author    = {S. Ganguly and M. Hegde and L. Zhang},
  title     = {{B}rownian bridge limit of path measures in the upper tail of {KPZ} models},
  journal   = {Preprint, available at arXiv:2311.12009},
  year      = {2023}
}

@inproceedings{Newman1995,
  author    = {C. M. Newman},
  title     = {A surface view of first-passage percolation},
  booktitle = {Proceedings of the International Congress of Mathematicians: August 3–11, 1994, Zürich, Switzerland},
  pages     = {1017--1023},
  publisher = {Springer},
  year      = {1995}
}

@article{DembinElboimPeled2024,
  author    = {B. Dembin and D. Elboim and R. Peled},
  title     = {Coalescence of geodesics and the {BKS} midpoint problem in planar first-passage percolation},
  journal   = {Geometric and Functional Analysis},
  volume    = {34},
  number    = {3},
  pages     = {733--797},
  year      = {2024}
}

@article{Alexander2023,
  author    = {K. S. Alexander},
  title     = {Geodesics, bigeodesics, and coalescence in first passage percolation in general dimension},
  journal   = {Electronic Journal of Probability},
  volume    = {28},
  pages     = {1--83},
  year      = {2023}
}

@article{MillerQian2023,
  author    = {J. Miller and W. Qian},
  title     = {Geodesics in the {B}rownian map: Strong confluence and geometric structure},
  journal   = {Preprint, available at arXiv:2008.02242},
  year      = {2023}
}

@article{AngelKolesnikMiermont2017,
  author    = {O. Angel and B. Kolesnik and G. Miermont},
  title     = {Stability of geodesics in the {B}rownian map},
  journal   = {Annals of Probability},
  volume    = {45},
  number    = {5},
  pages     = {3451--3479},
  year      = {2017}
}

@article{LeGall2010,
  author    = {J.-F. Le Gall},
  title     = {Geodesics in large planar maps and in the {B}rownian map},
  journal   = {Acta Mathematica},
  volume    = {205},
  number    = {2},
  pages     = {287--360},
  year      = {2010}
}

@article{Gwynne2021,
  author    = {E. Gwynne},
  title     = {Geodesic networks in Liouville quantum gravity surfaces},
  journal   = {Probability and Mathematical Physics},
  volume    = {2},
  number    = {3},
  pages     = {643--684},
  year      = {2021}
}

@article{ganguly2021local,
  title={Local and global comparisons of the Airy difference profile to Brownian local time},
  author={Ganguly, Shirshendu and Hegde, Milind},
  journal={arXiv preprint arXiv:2103.12029},
  year={2021}
}

@misc{tassopoulos2026kpzfixedpointbrownian,
      title={The {KPZ} fixed point and Brownian motion share the same null sets}, 
      author={Pantelis Tassopoulos and Sourav Sarkar},
      howpublishedf={Preprint, availabe at arXiv:2604.03358}, 
      year={2026},
}

@article{dauvergne2022disjointoptimizersdirectedlandscape,
  author        = {Dauvergne, Duncan and Zhang, Lingfu},
  title         = {Disjoint optimizers and the directed landscape},
  journal       = {Memoirs of the American Mathematical Society},
  note          = {To appear},
  eprint        = {2102.00954},
  archivePrefix = {arXiv},
  primaryClass  = {math.PR},
  year          = {2021}
}

@article{Basu2021Fractal,
  author    = {Basu, Riddhipratim and Ganguly, Shirshendu and Hammond, Alan},
  title     = {Fractal geometry of {Airy}$_{2}$ processes coupled via the {Airy} sheet},
  journal   = {The Annals of Probability},
  year      = {2021},
  volume    = {49},
  number    = {1},
  pages     = {485--505},
  publisher = {Institute of Mathematical Statistics},
  doi       = {10.1214/20-aop1444}
}

@article{Hatano1971Notes,
  author  = {Hatano, Kaoru},
  title   = {Notes on Hausdorff dimensions of {Cartesian} product sets},
  journal = {Hiroshima Mathematical Journal},
  year    = {1971},
  volume  = {1},
  number  = {1},
  pages   = {117--122},
  doi     = {10.32917/hmj/1206138139}
}

@book{morters_peres_2010, 
  place={Cambridge}, 
  title={Brownian Motion}, 
  publisher={Cambridge University Press}, 
  author={M{\"o}rters, Peter and Peres, Yuval}, 
  year={2010}
}

@article{sarkarthrehalves2022,
author = {Duncan {D}auvergne and {S}ourav {S}arkar and B{\'a}lint Vir{\'a}g},
title = {{Three-halves variation of geodesics in the directed landscape}},
volume = {50},
journal = {The Annals of Probability},
number = {5},
publisher = {Institute of Mathematical Statistics},
pages = {1947 -- 1985},
year = {2022},
doi = {10.1214/22-AOP1574},
URL = {https://doi.org/10.1214/22-AOP1574}
}

@misc{dauvergne2022scalinglimitlongestincreasing,
      title={The scaling limit of the longest increasing subsequence}, 
      author={Duncan Dauvergne and Bálint Virág},
      year={2022},
      eprint={2104.08210},
      archivePrefix={arXiv},
      primaryClass={math.PR},
      howpublished = {Preprint, available at arXiv:2104.08210}
}

@article{balazs2020non,
  author    = {Bal{\'a}zs, M{\'a}rton and Busani, Ofer and Sepp{\"a}l{\"a}inen, Timo},
  title     = {Non-existence of bi-infinite geodesics in the exponential corner growth model},
  journal   = {Forum of Mathematics, Sigma},
  volume    = {8},
  pages     = {e46},
  year      = {2020},
  publisher = {Cambridge University Press},
  doi       = {10.1017/fms.2020.31}
}

@article{cator2019attractiveness,
  author    = {Cator, Eric A. and L{\'o}pez, Sergio I. and Pimentel, Leandro P. R.},
  title     = {Attractiveness of {Brownian} queues in tandem},
  journal   = {Queueing Systems},
  volume    = {92},
  number    = {1--2},
  pages     = {25--45},
  year      = {2019},
  publisher = {Springer},
  doi       = {10.1007/s11134-019-09609-y}
}

@misc{Gangulyetal2022,
  author  = {Ganguly, Shirshendu and Zhang, Lingfu},
  title   = {Fractal geometry of the space-time difference profile in the directed landscape via construction of geodesic local times},
  journal = {arXiv},
  year    = {2022},
  howpublished = {Preprint, available at arXiv:2204.01674}
}

@book{anderson2010introduction,
  author    = {Anderson, Greg W. and Guionnet, Alice and Zeitouni, Ofer},
  title     = {An Introduction to Random Matrices},
  year      = {2010},
  publisher = {Cambridge University Press},
  series    = {Cambridge Studies in Advanced Mathematics},
  volume    = {118},
  doi       = {10.1017/CBO9780511801334},
  isbn      = {9780521194525}
}

@misc{baryshnikov2000random,
  author       = {Baryshnikov, Yuliy},
  title        = {Random Matrices and Queues in Series},
  year         = {2000},
  month        = dec,
  howpublished = {Algorithms Seminar, INRIA},
  note         = {Talk given December 11, 2000; summary by Marianne Durand},
  url          = {https://algo.inria.fr/seminars/sem00-01/baryshnikov.pdf}
}

@incollection{ledoux2007deviation,
  author    = {Ledoux, M.},
  title     = {Deviation Inequalities on Largest Eigenvalues},
  booktitle = {Geometric Aspects of Functional Analysis},
  editor    = {Milman, Vitali D. and Schechtman, Gideon},
  series    = {Lecture Notes in Mathematics},
  volume    = {1910},
  pages     = {167--219},
  year      = {2007},
  publisher = {Springer},
  address   = {Berlin, Heidelberg},
  doi       = {10.1007/978-3-540-72053-9_10},
  url       = {https://doi.org/10.1007/978-3-540-72053-9_10}
}

@article{BatesGangulyHammond2022,
  author  = {Bates, Erik and Ganguly, Shirshendu and Hammond, Alan},
  title   = {{Hausdorff} dimensions for shared endpoints of disjoint
             geodesics in the directed landscape},
  journal = {Electronic Journal of Probability},
  volume  = {27},
  year    = {2022},
  pages   = {1--44},
  note    = {Paper No. 1},
  doi     = {10.1214/21-EJP706}
}

@article{VW25,
  author        = {Vir{\'a}g, B{\'a}lint and Wu, Xuan},
  title         = {Actions in the {Airy} line ensemble and convergence to the {Airy} sheet},
  year          = {2025},
  journal       = {arXiv preprint arXiv:2511.11207},
  eprint        = {2511.11207},
  archivePrefix = {arXiv},
  primaryClass  = {math.PR},
  doi           = {10.48550/arXiv.2511.11207},
  url           = {https://arxiv.org/abs/2511.11207}
}

@book{beer1993topologies,
  author    = {Beer, Gerald},
  title     = {Topologies on Closed and Closed Convex Sets},
  series    = {Mathematics and Its Applications},
  volume    = {268},
  publisher = {Kluwer Academic Publishers},
  address   = {Dordrecht},
  year      = {1993},
  doi       = {10.1007/978-94-015-8149-3}
}

@article{gustavsson2005gaussian,
  author  = {Gustavsson, Jonas},
  title   = {Gaussian fluctuations of eigenvalues in the {GUE}},
  journal = {Annales de l'Institut Henri Poincar\'e,
             Probabilit\'es et Statistiques},
  year    = {2005},
  volume  = {41},
  number  = {2},
  pages   = {151--178},
  doi     = {10.1016/j.anihpb.2004.04.002}
}

@article{ledoux2010small,
  author  = {Ledoux, Michel and Rider, Brian},
  title   = {Small deviations for beta ensembles},
  journal = {Electronic Journal of Probability},
  year    = {2010},
  volume  = {15},
  number  = {41},
  pages   = {1319--1343},
  doi     = {10.1214/EJP.v15-798}
}

@misc{dauvergne2026downthemiddle,
  author        = {Dauvergne, Duncan and Syed, Fardin},
  title         = {Down-the-middle isometries for the {Airy} and {KPZ} sheets},
  year          = {2026},
  eprint        = {2609.26902},
  archivePrefix = {arXiv},
  url           = {https://arxiv.org/abs/2609.26902}
}

@article{konig2001eigenvalues,
  author  = {K{\"o}nig, Wolfgang and O'Connell, Neil},
  title   = {Eigenvalues of the {Laguerre} process as non-colliding
             squared {Bessel} processes},
  journal = {Electronic Communications in Probability},
  year    = {2001},
  volume  = {6},
  number  = {11},
  pages   = {107--114},
  doi     = {10.1214/ECP.v6-1040}
}

@inproceedings{newman1995surface,
  author    = {Newman, Charles M.},
  title     = {A surface view of first-passage percolation},
  booktitle = {Proceedings of the International Congress of Mathematicians,
               Vol. 1, 2 (Z{\"u}rich, 1994)},
  pages     = {1017--1023},
  publisher = {Birkh{\"a}user},
  address   = {Basel},
  year      = {1995}
}

@article{hoffman2008geodesics,
  author  = {Hoffman, Christopher},
  title   = {Geodesics in first passage percolation},
  journal = {The Annals of Applied Probability},
  volume  = {18},
  number  = {5},
  pages   = {1944--1969},
  year    = {2008}
}

@article{georgiou2017stationary,
  author  = {Georgiou, Nicos and Rassoul-Agha, Firas
             and Sepp{\"a}l{\"a}inen, Timo},
  title   = {Stationary cocycles and {Busemann} functions
             for the corner growth model},
  journal = {Probability Theory and Related Fields},
  volume  = {169},
  number  = {1-2},
  pages   = {177--222},
  year    = {2017}
}

\end{document}